\documentclass[12pt, notitlepage]{amsart}
\usepackage{amsmath, amsthm, amsfonts, amssymb, cite}
\usepackage{wrapfig}
\usepackage{stackrel}
\usepackage{color}
\usepackage{graphicx}
\usepackage{float}
\usepackage{enumitem}
\usepackage{mathtools}

\usepackage{multicol}
\usepackage[normalem]{ulem}
\usepackage[utf8]{inputenc}

\newcommand{\lt}{\left}
\newcommand{\rt}{\right}
\newcommand{\bpm}{\begin{pmatrix}}
\newcommand{\epm}{\end{pmatrix}}
\newcommand{\bsm}{\lt(\begin{smallmatrix}}
\newcommand{\esm}{\end{smallmatrix}\rt)}
\newcommand{\beq}{\begin{equation}}
\newcommand{\eeq}{\end{equation}}
\newcommand{\bmat}{\begin{matrix*}}
\newcommand{\emat}{\end{matrix*}}

\renewcommand{\d}{d} 

\newcommand{\Z}{\ensuremath{\mathbb{Z}}}

\newcommand{\R}{\ensuremath{\mathbb{R}}}

\newcommand{\g}{\ensuremath{\mathfrak{g}}}
\renewcommand{\a}{\mathfrak{a}}
\renewcommand{\b}{\mathfrak{b}}
\let\turc\c
\usepackage{etoolbox}
\robustify\turc 
\renewcommand{\c}{\mathfrak{c}}

\newcommand{\vep}{\varepsilon}

\DeclareMathOperator{\SL}{SL}

\renewcommand{\th}{\textsuperscript{th}}
\newcommand{\inv}{^{-1}}
\newcommand{\hf}{\frac{1}{2}}
\newcommand{\qtr}{\frac{1}{4}}

\renewcommand{\Re}{\operatorname{Re}}
\renewcommand{\Im}{\operatorname{Im}}

\usepackage[usenames,dvipsnames]{xcolor}

\newcommand{\sums}{\sideset{}{^*}{\sum}}

\usepackage[inner=1.0in,outer=1.0in,bottom=1.0in, top=1.0in]{geometry}

\usepackage{amsmath, mathtools}

\newcommand{\mattletter}{k_0'}
\newcommand{\mehmetletter}{\mathcal{I}}

\newcommand{\taco}{h}

\newcommand{\exponent}{\delta}
\newcommand{\kloostermanmodulus}{\gamma}

\newcommand{\mr}{\ensuremath{\mathbb R}}

\newcommand{\mymod}{\ensuremath{\negthickspace \negmedspace \pmod}}
\newcommand{\shortmod}{\ensuremath{\negthickspace \negthickspace \negthickspace \pmod}}

\newcommand{\half}{\ensuremath{ \frac{1}{2}}}

\newcommand{\intR}{\int_{-\infty}^{\infty}}

\newcommand{\sumstar}{\sideset{}{^*}\sum}
\newcommand{\sumprime}{\sideset{}{'}\sum}

\DeclareMathOperator{\flooroot}{flrt}

\theoremstyle{plain}            
        \newtheorem{mytheo}{Theorem} [section]
        
        \newtheorem{myprop}[mytheo]{Proposition}
        \newtheorem{mycoro}[mytheo]{Corollary}
     \newtheorem{mylemma}[mytheo]{Lemma}
        \newtheorem{mydefi}[mytheo]{Definition}

\numberwithin{equation}{section}
\numberwithin{figure}{section}

\title{The Fifth Moment of modular $L$-functions 
}
\author{Eren Mehmet K{\i}ral, Matthew P. Young}

\address{Department of Mathematics \\
	  Texas A\&M University \\
	  College Station \\
	  TX 77843-3368 \\
		U.S.A.}
\email{ekiral@tamu.edu}		
\email{myoung@math.tamu.edu}

\thanks{This material is based upon work of M.Y. supported by the National Science Foundation under agreement No. DMS-1401008.  Any opinions, findings and conclusions or recommendations expressed in this material are those of the authors and do not necessarily reflect the views of the National Science Foundation. \\
 The first author was partially supported by an AMS-Simons travel grant.}

\begin{document}

\begin{abstract}
We prove a sharp bound on the fifth moment of modular $L$-functions of fixed small weight, and large prime level.
\end{abstract}

\maketitle

\section{Introduction}
Let $q$ be a prime and $\kappa \in \{4,6,8,10,14\}$. Let $H_\kappa(q)$ be 
    the set of weight $\kappa$ Hecke eigenforms on $\Gamma_0(q)$.
    For any $f \in H_\kappa(q)$ (note that any such $f$ is automatically a newform), let $\lambda_f(n)$ denote its
    $n$\th\ Hecke eigenvalue.  
 Our main result is  the following theorem: 
\begin{mytheo}
\label{thm:mainthm}
 We have
 \begin{equation}
  \sum_{f \in H_{\kappa}(q)} L(1/2, f)^5 \ll q^{1 + \theta + \varepsilon},
 \end{equation}
as $q\to \infty$ among primes.
Here $\theta$ is the best-known progress towards the Ramanujan-Petersson conjecture.
\end{mytheo}
The currently best-known value $\theta = 7/64$  is given by the work of Kim and Sarnak \cite{KimSarnak}.
 The central value $L(1/2,f)$ is nonnegative 
 by \cite[Corollary 2]{kohnen1985fourier} and \cite{Waldspurger}, 
 so upon dropping all but one term, we deduce:
\begin{mycoro}
\label{coro:subconvexity}
 For any $\varepsilon > 0$, we have
 \begin{equation}
  L(1/2, f) \ll_\varepsilon q^{\frac{1 + \theta}{5} + \varepsilon}.
 \end{equation}
\end{mycoro}
Previously, Duke, Friendlander, and Iwaniec \cite{DFISubconvexity2} bounded the amplified fourth moment in this family, and 
Kowalski, Michel, and VanderKam \cite{KMVdK} asymptotically evaluated a mollified fourth moment.  Recently, Balkanova and Frolenkov \cite{BalkanovaFrolenkov} improved the error term in these fourth moment problems, and thereby deduced the so-far best-known subconvexity result of $L(1/2,f) \ll q^{\frac{27 - 30\theta}{112 - 128\theta}}$. 
Corollary \ref{coro:subconvexity} improves this further.  Our method of proof takes a different course from these works, and we never solve a shifted convolution problem.

This paper has some common features with the cubic moment studied by Conrey and Iwaniec \cite{CI}.  Their work also bounds a moment that is one larger than what one may consider the ``barrier'' to subconvexity.  That is, for the family of $L$-functions they consider, an upper bound on the second moment that is Lindel\"of-on-average leads back precisely to the convexity bound on an individual $L$-value. Similarly, the fourth moment is the ``barrier'' in the family of Theorem \ref{thm:mainthm}.  In a sense, going one full moment beyond the barrier is a way of amplifying with the $L$-function itself.
As far as the authors are aware, prior to Theorem \ref{thm:mainthm}, the only known instances of a sharp upper bound on a moment that is one larger than the barrier moment are the cubic moment and its generalizations \cite{CI} \cite{IvicCubic}
\cite{PetrowCubic} \cite{YoungCubic} \cite{PetrowYoung}.

In Section \ref{section:sketch}, we give a simplified sketch of the argument.  The main overall difficulty in the problem is that we require a significant amount of cancellation in multivariable sums with divisor functions and Kloosterman sums.  The main thrust of the argument is to apply summation formulas that shorten the lengths of summation, eventually obtaining a sum of Kloosterman sums.  To this, we apply the Bruggeman-Kuznetsov formula, which leads to a fourth moment of Hecke-Maass $L$-functions twisted by $\lambda_j(q)$, with an additional average over the level.  
This is another incarnation of a Kuznetsov/Motohashi-type formula
where a moment problem in  one family of $L$-functions is related to another moment in a ``dual'' family (see
\cite[Section 1.1.3]{MichelVenkatesh}).
Along the way, we encounter many ``fake'' main terms, which turn out to be surprisingly difficult to estimate.  A straightforward bound on these would only lead to $O(q^{5/4+\varepsilon})$ in Theorem \ref{thm:mainthm}, which would be trivial. 
We expect that all the ``fake'' main terms calculated in this paper should essentially cancel, but doing so is a daunting prospect.  Instead, we show that with an appropriate choice of weight functions in the approximate functional equations, then all the fake main terms are bounded consistently with Theorem \ref{thm:mainthm}.
The amplified/mollified fourth moment (cf. \cite{DFISubconvexity2} \cite{KMVdK}) also required a difficult analysis of the main terms, which arose from solving the shifted convolution problem.  As such, it is not clear how to compare the main term calculations here with \cite{DFISubconvexity2} \cite{KMVdK}.  The article \cite[Section 1.2]{BlomerHarcosMichel2007BoundsForModularLFunctionsInLevel} has a more through discussion of the main term analysis with the shifted convolution approach.

One of the practical difficulties in applying the Bruggeman-Kuznetsov formula in applications is that one needs to recognize the particular shape of sum of Kloosterman sums one encounters (with coprimality and congruence conditions, etc.) as one associated to a group $\Gamma$, pair of cusps $\a, \b$, and nebentypus $\chi$.
To this end, in \cite{DoubleCosetPaper}  we have identified all the Kloosterman sums belonging to the congruence subgroup $\Gamma_0(N)$ and at general Atkin-Lehner cusps (i.e., those cusps equivalent to $\infty$ under an Atkin-Lehner involution) with general Dirichlet characters.  It turns out that there is a ``correct'' choice of scaling matrix to use when computing the Fourier coefficients and Kloosterman sums, a choice that ensures the multiplicativity of Fourier coefficients at the Atkin-Lehner cusps.  In Section \ref{section:Kuznetsov}, we record the special cases of these Kloosterman sums that are required in this work.

Another practical difficulty is estimating oscillatory integral transforms with weight functions depending on multiple variables, with numerous parameters.  It is desirable to perform stationary phase 
analysis
on a single variable at a time, yet still keep track of the behavior of the remaining variables in a succinct way.  
  We have codified some properties of a family of weight functions that allows us to do this efficiently.  The key stationary phase result, which is a modest generalization of work in \cite{BKY},
   is stated as Lemma \ref{lemma:statphase} below, with a proof appearing in \cite{InertFunctionPaper}.

In the spectral analysis of a sum of Kloosterman sums, it is necessary to treat the Maass forms, holomorphic forms, and continuous spectrum.  In our situation, the Maass forms and holomorphic forms are rather similar, and lead to a twisted fourth moment of $GL_2$ cuspidal $L$-functions.  The continuous spectrum is similar in many respects, but a key difference is that the ``dual" family of $L$-functions is essentially a sum of a product of eight Dirichlet $L$-functions at shifted arguments.  One naturally wishes to 
treat the continuous spectrum on the same footing as the discrete spectrum,
which requires shifting some contour integrals past the poles of the Dirichlet $L$-functions (which occur only when the character is principal).  There is potentially a large loss in savings from these poles on the $1$-line compared to the contribution from the $1/2$-line.  
Luckily, it turns out that there is some extra savings in the residues of the Dirichlet series (essentially, from considering only the principal characters) that balances against this loss.
This savings ultimately arises from a careful calculation of the Fourier expansion of the Eisenstein series on $\Gamma_0(N)$ with arbitrary $N$, attached to an arbitrary cusp, expanded around any Atkin-Lehner cusp.  This calculation occurs in \cite{DoubleCosetPaper}.

An astute reader may note that $\kappa =2$ is not covered by Theorem \ref{thm:mainthm}.  In fact, there is only a single instance where 
our proof
 requires that $\kappa >2$, namely in the study of the continuous part of the spectrum at \eqref{eq:ScBoundNonoscillatoryCase2}.  Perhaps with further analysis one might incorporate the weight $2$ case, by a more careful analysis of the residues of the Dirichlet $L$-functions. 
The restriction that $q$ is a prime and that $\kappa \leq 14$, $\kappa \neq 12$, means that the cuspforms $f \in H_\kappa(\Gamma_0(q))$ are automatically newforms. 
It is reasonable to expect that using a more general Petersson formula for newforms (e.g., see \cite{PetrowYoung}) could relax these assumptions, but the arithmetical complexity would be increased.


\section{High-level sketch}
\label{section:sketch}
Here we include an outline of the major steps used in the proof, intended for an expert audience.
By approximate functional equations and the Petersson formula, we arrive at
\begin{equation}
\label{eq:fifthmomentS}
 \mathcal{S} := \sum_{m \ll q} \sum_{n \ll q^{3/2}} \sum_{c \equiv 0 \shortmod{q}} \frac{\tau(m) \tau_3(n)}{c \sqrt{mn}} S(m,n;c) J_{\kappa-1} \Big(\frac{4 \pi \sqrt{mn}}{c}\Big),
\end{equation}
and we wish to show $\mathcal{S} \ll q^{\theta+\varepsilon}$.  The hardest case to consider is $m \asymp q$, $n \asymp q^{3/2}$, and $c \asymp \sqrt{mn} \asymp q^{5/4}$, in which case $J_{\kappa-1}(x) \approx 1$.  In practice, one needs to treat the two ranges of the Bessel function (i.e., $x \ll 1$ and $x \gg 1$) differently.  
In this sketch, we focus on the transition region of the $J$-Bessel function where $x \asymp 1$.

The Weil bound applied to the Kloosterman sum shows $\mathcal{S} \ll q^{7/8+\varepsilon}$, far from $q^{\theta+ \varepsilon}$.

The immediate problem with \eqref{eq:fifthmomentS} is that Voronoi summation applied to $m$ or $n$ leads to a dual sum that is longer than before.  The conventional wisdom is that this is a bad move.  However, there do not seem to be any other moves available, so it may be necessary to take a loss in the first step.  We may attempt to minimize the loss here by opening $\tau(m) = \sum_{m_1 m_2 =m} 1$, supposing $m_1 \leq m_2$ by symmetry, and applying Poisson summation in $m_2$ modulo $c$.  This leads to
\begin{equation*}
 \mathcal{S} \approx \sum_{m_1 \ll q^{1/2}}
 \sum_{\substack{c \equiv 0 \shortmod{q} \\ c \asymp q^{5/4}}}
 \sum_{k \ll q^{3/4}} \sum_{n \asymp q^{3/2}} 
 \frac{\tau_3(n)}{\sqrt{m_1 k nc}} e\Big(\frac{m_1 n \overline{k}}{c}\Big).
 \end{equation*}
Note that the trivial bound now gives $\mathcal{S} \ll q$, so we lost a factor $q^{1/8}$ from going the ``wrong way'' in Poisson.  However, now we may gain from the structure of the arithmetical part by applying the well-known reciprocity formula
\begin{equation*}
 e\Big(\frac{m_1 n \overline{k}}{c}\Big) = 
 e\Big(-\frac{m_1 n \overline{c}}{k}\Big) e\Big(\frac{m_1 n}{ck}\Big).
\end{equation*}
This effectively switches the roles of $c$ and $k$, at the expense of introducing the potentially-oscillatory factor $e_{ck}(m_1 n)$  into the weight function.  However, when all variables are near their maximial sizes, then this factor is not oscillatory, so we shall ignore it in this sketch.  

{\bf Side remark.}  If one applies Voronoi to the sum over $m$ (which is more in line with the previous works on the amplified/mollified fourth moment), then one encounters a shifted divisor sum of the form $\sum_{m-n= h} \tau_2(m) \tau_3(n)$.  Such sums have been considered by various authors, with the most advanced results being the recent work of
B. Topa\turc{c}o\u{g}ullar{\i} \cite{Topacogullari}.

One way to proceed next would be to convert the additive character into Dirichlet characters (modulo $k$), which has a nice benefit of separating the variables, a key step in \cite{CI}.  This would lead to a fifth moment of Dirichlet $L$-functions twisted by Gauss sums, with an averaging over the modulus.  One may check that Lindel\"{o}f applied to these $L$-functions 
only gives 
$\mathcal{S} \ll q^{1/4+\varepsilon}$ which in a sense gets back to the convexity bound.  

Now it is beneficial to apply Voronoi summation in $n$ modulo $k$ (one may view this as opening $\tau_3(n) = \sum_{n_1 n_2 n_3 = n} 1$, and applying Poisson in each $n_i$).  This leads to
\begin{equation}
\label{eq:fifthmomentKloosterman}
 \mathcal{S} \approx 
 \sum_{m_1 \ll q^{1/2}}
 \sum_{\substack{c \equiv 0 \shortmod{q} \\ c \asymp q^{5/4}}}
 \sum_{k \ll q^{3/4}} \sum_{p \ll q^{3/4}} 
 \frac{\tau_3(p)}{k\sqrt{m_1 p c}} S(p, c \overline{m_1}, k).
\end{equation}
The trivial bound now is $q^{5/8}$, consistent with saving $q^{1/8}$ in each of the $n_i$ variables, just as we lost $q^{1/8}$ by Poisson in the initial $m_2$ variable.  One could also apply Poisson in $m_1$ to save another $q^{1/8}$, but then the arithmetical sum becomes a hyper-Kloosterman sum, which increases the difficulty of the problem (N. Pitt has studied this problem \cite{Pitt}, but it seems very hard to obtain enough cancellation using this approach).  Here we have a Kloosterman sum to which we may apply the Bruggeman-Kuznetsov formula of level $m_1$.  Using this, we obtain
\begin{equation}
\label{eq:SApproxSketch}
 \mathcal{S} \approx 
 \sum_{m_1 \ll q^{1/2}} 
 \sum_{\substack{t_j \ll q^{\epsilon} \\ \text{level $m_1$}} } 
 \sum_{\substack{c \equiv 0 \shortmod{q} \\ c \asymp q^{5/4}}}
\sum_{p_1,p_2,p_3 \ll q^{1/4}} 
 \frac{\nu_j(p_1c)\nu_j(p_2p_3)}{\sqrt{ p_1p_2p_3 c}}.
\end{equation}
We can essentially write this as
\begin{equation}
\label{eq:FourthMomentTwisted}
 \mathcal{S} \approx  \sum_{m_1 \ll q^{1/2}} 
 \sum_{\substack{t_j \ll q^{\epsilon} \\ \text{level $m_1$}} }  \nu_j(1)^2 \frac{\lambda_j(q)}{\sqrt{q}} L(1/2, u_j)^4.
\end{equation}
Here the scaling on the spectral data is that $\sum_{t_j \ll T} \nu_j(1)^2 \ll T^2 m_1^{\varepsilon}$.  Thus we have converted to a twisted fourth moment of Maass form $L$-functions, and one can see how $q^{\theta+\varepsilon}$ emerges by bounding $|\lambda_j(q)| \ll q^{\theta+\varepsilon}$, and using a Lindel\"{o}f-on-average bound for the spectral fourth moment (which in turn is ``easy'', following from the spectral large sieve inequality).  

The above discussion implicitly assumed that the $p_i$ are nonzero.  The zero frequencies (where some or all $p_i = 0$) turn
 out to be the ``fake'' main terms alluded to in the introduction.  
 
 To handle these, we compute the weight function explicitly, and evaluate the sums over $k, m_1$, and $c$ as zeta quotients. We later bound the integral by moving lines of integration, and apparent poles of the integrand are cancelled by a choice of the weight function in the approximate functional equation.
To elaborate on this point, consider an overly-simplified model with a sum of the form $S =\sum_{n \geq 1} \frac{1}{\sqrt{n}} V(\frac{n}{\sqrt{q}})$, where $V(x) = \frac{1}{2 \pi i} \int_{(1)} \frac{G(s)}{s} x^{-s} ds$, and $G(s)$ is analytic satisfying $G(0) = 1$, with rapid decay in the imaginary direction.  The trivial bound applied to $S$ gives $S = O(q^{1/4})$, using that $V(x) \ll (1+x)^{-100}$.  Alternatively, we have $S = \frac{1}{2 \pi i} \int_{(1)} \zeta(1/2 + s) q^{s/2} \frac{G(s)}{s} ds$, which by shifting contours to the line $\text{Re}(s) = \varepsilon >0$ gives $S = G(1/2) q^{1/4}  + O(q^{\varepsilon})$.  If $G(1/2) = 0$ (which one is free to assume in the context of the approximate functional equation), then in fact one has an improved bound of $S = O(q^{\varepsilon})$.  This  is the principal idea behind the estimation of the fake main terms. The main difficulty in practice is that one has a much more complicated sum with multiple variables and weight functions that arise as integral transforms, and it requires 
significant
work to recognize instances of this basic idea.  One should also observe that the above method of estimating $S$ is highly reliant on the specific form of the weight function $V$; if it were multiplied by a compactly-supported bump function (say one part of a dyadic partition of unity), then one could not deduce $S=O(q^{\varepsilon})$ anymore.  Since we shall apply dyadic partitions of unity in the forthcoming treatment, for the purposes of estimating these fake main terms, it is crucial to re-assemble the partitions. 

The role of the $m_1$-variable within the proof has some curious features.  In the sketch above up through \eqref{eq:FourthMomentTwisted}, the
 $m_1$ variable was hardly used.
Precisely, we never applied a summation formula nor obtained any cancellation from this variable.  Nor did we use any reciprocity involving $m_1$ to lower a modulus.  However, non-trivial estimations involving $m_1$ do appear in other parts of the proof.  In the evaluation of one type of fake main term  in Section \ref{section:TwopiAreZero}, we evaluate the $m_1$-sum similarly to the discussion in the previous paragraph;  the lack of pole at $s=1/2$ amounts to square-root cancellation in this variable. 
The other location is in estimating the continuous spectrum analog of \eqref{eq:SApproxSketch} which so far was neglected in this sketch.  One may show that the continuous spectrum analog of \eqref{eq:FourthMomentTwisted} is $O(q^{\varepsilon})$ using that the number of cusps on $\Gamma_0(m_1)$ is at most $O(m_1^{1/2+\varepsilon})$.  
However, on average over $m_1$, the number of cusps is $O(m_1^{\varepsilon})$, which leads to a bound that saves an additional factor $q^{1/4}$.  In a sense, this discussion indicates that the continuous spectrum is smaller in measure \emph{in the level aspect} than the discrete spectrum, on average over the level.

\section{Kloosterman Sums and Bruggeman-Kuznetsov  Formula}
\label{section:Kuznetsov}

\subsection{Cusps, scaling matrices, and Kloosterman sums}

	We mostly follow the notation of \cite{IwaniecSpectralBook}. 
	Let $N$ be a positive integer and $\Gamma = \Gamma_0(N)$.
            Let $\a$ be a cusp and 
                $\Gamma_\a = \{ \gamma \in \Gamma : \gamma \a = \a\}$
            be the stabilizer of the cusp $\a$ in $\Gamma$. 
            A matrix $\sigma_\a \in \SL_2(\R)$ satisfying      
            \begin{equation} \label{eq:scalingMatrixProperties}
                \sigma_\a \infty = \a, \quad \text{ and } \quad \sigma_\a\inv \Gamma_\a \sigma_\a = \left\{ \pm \left(\begin{smallmatrix} 1 & n\\ 0&1 \end{smallmatrix} \right) : n \in \Z\right\},
            \end{equation}
            is called a scaling matrix for the cusp $\a$. 
            
%
        
        \begin{mydefi}
             Let $f$ be a Maass form for the group $\Gamma$. The Fourier coefficients of $f$ at a cusp $\a$, denoted
             $\rho_{\a f}(n)$, are defined by 
             \begin{equation}
              \label{eq:FourierExpansionMaassForm}
              f(\sigma_\a z) = \sum_{n \neq 0} 
              \rho_{\a f}(n)
              e(nx)    W_{0,i t_j}(4 \pi |n| y),
             \end{equation}
             where $W_{0,i t_j}$ is the Whittaker function defined by
             \begin{equation*}
              W_{0,i t_j}(4 \pi y) = 2 \sqrt{y} K_{i t_j} (2 \pi y).
             \end{equation*}

        \end{mydefi}
         
The Fourier coefficients $\rho_{\a f}(n)$ depend on the choice of scaling matrix $\sigma_{\a}$, and it may be more accurate to denote them $\rho_{\sigma_{\a}, f}(n)$.

	\begin{mydefi}
	\label{def:KloostermanSum}
            For $\a$ and $\b$ cusps for $\Gamma$, we define the Kloosterman sum associated to $\a,\b$  with modulus $c$ as
            \begin{equation}\label{eq:kloostermanDefinition}
		S_{\a \b}  (m,n;c) =
		\sum_{\gamma = \bsm a & b\\ c &d \esm \in \Gamma_\infty \backslash \sigma_\a\inv \Gamma \sigma_\b / \Gamma_\infty }
		e\left(\frac{am + dn }{c} \right).
            \end{equation}
        \end{mydefi}

 

	\begin{mydefi} The set of allowed moduli is
            \begin{equation}\label{eq:allowedKloostermanModuli}
                \mathcal{C}_{\a \b} = \left\{\kloostermanmodulus > 0: \bsm * &*\\ \kloostermanmodulus &* \esm \in \sigma_\a\inv \Gamma \sigma_\b \right\}.
            \end{equation}
        \end{mydefi}
        Notice that if $\kloostermanmodulus \notin \mathcal{C}_{\a \b}$ the Kloosterman sum of modulus $\kloostermanmodulus$ is an empty sum. 

\subsection{Atkin-Lehner cusps}\label{subsec:AtkinLehnerScaling}
        Assume that $N = rs$ with $(r,s) = 1$. We  call a cusp of the form $\mathfrak{a} = 1/r$ (with $(r,s) = 1$) an \emph{Atkin-Lehner cusp}. The Atkin-Lehner cusps are precisely those that are equivalent to $\infty$ under an Atkin-Lehner operator, justifying their name.

A newform is an eigenfunction of all the Hecke operators, as well as all the Atkin-Lehner operators.     
It turns out that one may choose a scaling matrix $\sigma_{1/r}$ for the Atkin-Lehner cusp $1/r$ to be an Atkin-Lehner operator (see \cite[Section 2.2]{DoubleCosetPaper}).  
Therefore, for such a choice of scaling matrix, we have
\begin{equation}
\label{eq:AtkinLehnerEigenfunction}
 \rho_{\frac{1}{r} f}(n) = \eta_s(f) \rho_{\infty f}(n),
\end{equation}
where $f$ is a newform and $\eta_s(f) = \pm 1$ is the eigenvalue of the Atkin-Lehner operator $W_s$. 

        \begin{myprop} \label{prop:KloostermanSumsAtInfinityAndAtkinLehnerCusp}
             Let $N = rs$ with $(r,s) = 1$, and choose $\sigma_{1/r}$ as above. 
             Then the set of allowed moduli for the pair 
             of cusps $\infty, \tfrac{1}{r}$ is 
             \begin{equation} \label{eq:infinity1rAllowedModuli}
                  \mathcal{C}_{\infty, 1/r} = \left\{\kloostermanmodulus = c \sqrt{s} > 0 : c \equiv 0 \mymod{r}, \thinspace (c,s) = 1\right\},
             \end{equation}
             and for such $\kloostermanmodulus = c\sqrt{s} \in \mathcal{C}_{\infty,1/r}$, the Kloosterman sum to modulus $\kloostermanmodulus$ is given by
             \begin{equation}
             \label{eq:Sinfinity1/rKloosterman}
                  S_{\infty,1/r}  (m,n;c\sqrt{s}) = S(\overline{s} m,n;c),
             \end{equation}
            where the $S$ on the right denotes an ordinary Kloosterman sum.  Consequently,
\begin{equation}
\label{eq:KloostermanSumsExampleCase}
 \sum_{\substack{(c,s) = 1 \\ c \equiv 0 \shortmod{r}}}  S(\overline{s} m, n;c) f(c) = \sum_{\kloostermanmodulus \in \mathcal{C}_{\infty, 1/r}} S_{\infty,1/r}(m,n;\gamma) f\Big(\frac{\gamma}{ \sqrt{s}}\Big),
\end{equation}
where $f$ is any function so that the sums converge.
        \end{myprop}
For this computation, see   \cite[Section 14]{motohashi2007riemann}, in particular (14.8).  Note that \eqref{eq:Sinfinity1/rKloosterman} differs from a formula in
         \cite[p.58]{IwaniecClassicalBook} by an additive character, which is due to a different choice of the scaling matrix. See also \cite{DoubleCosetPaper} for generalization with different cusps and characters.

\subsection{Bruggeman-Kuznetsov formula} \label{subsec:BruggemanKuznetsov}
	We record the spectral expansion of a sum of Kloosterman sums in a spectral basis of the space $L^2(\Gamma_0(N))$. Let $\{u_j\}$ be a basis of cusp forms. Assume that $u_j$ is an eigenfunction of the Laplace-Beltrami operator with eigenvalue $\qtr + t_j^2$. Call $t_j$ the spectral parameter of $u_j$.  
	Define $\rho_{\a j}(n) = \rho_{u_j}(\sigma_{\a}, n)$ as in \eqref{eq:FourierExpansionMaassForm}; our choice of $\sigma_{\a}$, in practice, will be an Atkin-Lehner operator.

 	Likewise, write the Fourier expansion of the Eisenstein series as 
	\begin{equation}\label{eq:EisensteinFourierExpansion}
		E_\c(\sigma_\a z,u ) 
		= \delta_{\a\c} y^{u} + \rho_{\a\c}(0,u) y^{1 - u} 
			+ \sum_{n \neq 0} \rho_{\a\c}(n,u)
			e(nx) W_{0,u-\frac12}(4 \pi |n| y).
	\end{equation}	 
	Consulting \cite[Theorem 3.4]{IwaniecSpectralBook}, we have
	\begin{equation}
	\label{eq:rhodef}
		\rho_{\a\c}(n,u)
		= \begin{cases}
		\phi_{\a\c}(n,u) \frac{\pi^u}{\Gamma(u)} |n|^{u - 1} ,  & \text{ if } n \neq 0
		\\ 
		 \delta_{\a\c}y^u + \phi_{\a\c}(u) y^{1-u}, & \text{ if } n = 0,
		 \end{cases}
	\end{equation}
	where
	\begin{equation}\label{eq:phiDefinition}
		\phi_{\a\c}(n,u) 
		= 
		\sum_{\substack{(\gamma,\delta) \text{ such that} \\ \rho =  \bsm *&*\\\gamma&\delta \esm \in  \Gamma_\infty \backslash \sigma_\c\inv \Gamma \sigma_\a /\Gamma_\infty  }}  
		\frac{1}{\gamma^{2u}} e\left(\frac{n\delta}{\gamma}\right)
		= \sum_{\gamma \in \mathcal{C}_{\c \a}} \frac{S_{\c  \a}(0,n;\gamma)}{\gamma^{2u}} 
		,
	\end{equation}
	and $\phi_{\a\c}(u) = \phi_{\a\c}(0,u)$.  
Note that our ordering of the cusps in the notation $\rho_{\a \c}, \phi_{\a \c}$ is reversed from that of \cite{IwaniecSpectralBook}, and also that \cite[(3.22)]{IwaniecSpectralBook} should have $\mathcal{S}_{\a \c}(n,0;c)$ in place of $\mathcal{S}_{\a \c}(0,n;c)$ to be consistent with \cite[(2.23)]{IwaniecSpectralBook}. 
	We give an explicit computation of $\phi_{\a \c}(n,u)$ with Proposition \ref{prop:EisensteinFourierCoefficientFormulaDirichletCharacters} below.

	For aesthetic purposes, define as in \cite[(8.5), (8.6)]{IwaniecSpectralBook} 
	\begin{equation} \label{eq:nuDefinition}
		\nu_{\a j}(n) = \left(\frac{4\pi |n| }{\cosh(\pi t_j) }\right)^\hf \rho_{\a j} (n), \qquad
		\nu_{\a\c}(n,u) = 
		\left(\frac{4\pi |n| }{\cos(\pi (u-\frac12)) }\right)^\hf
		\rho_{\a\c}  (n,u). 
	\end{equation}
	
	Let $g\in H_k(N)$, that is, let $g$ be a holomorphic level $N$ weight $k$ modular cusp form.  Define the Fourier expansion of $g$ at a cusp $\a$ by
	\[
		g \vert_{\sigma_{\a}} (z) = \sum_{n=1}^\infty \rho_{\a g}(n) n^{\frac{k-1}{2}} e(nz).
	\]
Also define 
\begin{equation}
 \nu_{\mathfrak{a} g}(n) = \Big(\frac{\pi^{-k}\Gamma(k)}{4^{k-1}}\Big)^{1/2} \rho_{\mathfrak{a} g}(n),
\end{equation}
similarly to \cite[(9.42)]{IwaniecSpectralBook}, but note that $m^{\frac{k-1}{2}}$ was already extracted  in the definition of $\rho_{\a g}(m)$.

	With the notation as above, define for nonzero $m$ and $n$,
	\begin{equation}\label{eq:GeneralSumOfKloostermanSum}
		\mathcal{K} = \sum_{\kloostermanmodulus \in \mathcal{C}_{\a,\b} } S_{\a\b}  (m,n;\kloostermanmodulus) \phi(\kloostermanmodulus).
	\end{equation}
	We then quote the  literature for a spectral formula for this sum. 
	Many authors state the Bruggeman-Kuznetsov formula with a weight function of the form $\kloostermanmodulus^{-1} F(\frac{4 \pi \sqrt{mn}}{\kloostermanmodulus})$ in place of $\phi(\kloostermanmodulus)$, which amounts to the substitution $F(t) = \frac{4\pi \sqrt{mn}}{t} \phi(\frac{4 \pi \sqrt{mn}}{t})$.
	\begin{mytheo}[\cite{IwaniecSpectralBook} Chapter 9] \label{thm:KuznetsovTraceFormula}
		 Let $\mathcal{K}$ be as in \eqref{eq:GeneralSumOfKloostermanSum}.  Assuming $\phi$ is smooth with compact support on $(0, \infty)$, 
		 we have 
		\[
			\mathcal{K} = \mathcal{K}_d + \mathcal{K}_c + \mathcal{K}_h.
		\]
		Here $\mathcal{K}_h = 0$ if $mn<0$ and otherwise,
		\begin{equation}\label{eq:holomorphicKuznetsovSpectrum}
			\mathcal{K}_h = \sum_{k >0, \text{ even}} \phi_h(k)  i^k   
			\sum_{g\in H_k(N) }  \nu_{\a g}(m) \overline{\nu_{\b g}(n)}.
		\end{equation}
		The discrete spectrum contribution is given as 
		\begin{equation}\label{eq:discreteKuznetsovSpectrum}
			\mathcal{K}_d = \sum_{t_j} \phi_\pm(t_j) \nu_{\a j}(m) \overline{\nu_{\b j}(n)},
		\end{equation}
		where the summation is over the spectral parameters $t_j$ of a chosen orthonormal basis of cusp forms $\{u_j\}_j$. The continuous spectrum contribution is
	\begin{equation}\label{eq:continuousKuznetsovSpectrum}
			\mathcal{K}_c = \sum_{\mathcal{\c} } \frac{1}{4\pi}
			\int_{-\infty}^\infty \phi_\pm(t) \nu_{\a\c}(m, \tfrac12  + i t )\overline{\nu_{\b\c}(n, \tfrac12 + it)} \d t,
		\end{equation}
		where the choice $\phi_+$ versus $\phi_-$ depends on whether $mn>0$ or $mn<0$. 

		Here the integral transform for $\phi_h$ is given as 
		\begin{equation} \label{eq:holomorphicIntegralTransform}
			\phi_h(k) =  \int_0^\infty J_{k-1}(x) \frac{4\pi \sqrt{mn}}{x} 
			\phi\left(\frac{4\pi \sqrt{mn}}{x} \right) \frac{\d x}{x}
			=  \left(J_{k-1} * (x \cdot \phi)\right) (4\pi \sqrt{mn}).
		\end{equation}
With
		\[
			B_{2it}^+(x) = \frac{i}{2\sinh(\pi t)} \left(J_{2it}(x) - J_{-2it}(x) \right),
		\]
		then
		\begin{equation}\label{eq:plusIntegralTransform}
			\phi_+ (t) 
			=  \int_0^\infty B^+_{2it}(x) \frac{4\pi \sqrt{mn}}{x} 
					\phi\left(\frac{4\pi \sqrt{mn}}{x}\right) \frac{\d x}{x}
			=  \left(B^+_{2it} * (x \cdot \phi) \right) (4\pi \sqrt{mn}).
		\end{equation}
		Similarly, with
		\[
			B_{2it}^-(x) = \frac{2}{\pi} \cosh(\pi t) K_{2it}(x),
		\]
		we have
		\begin{equation}\label{eq:minusIntegralTransform}
			\phi_-(t)
			= \int_0^\infty B_{2it}^-(x) \frac{4\pi \sqrt{|mn|}}{x}  
				\phi\left(\frac{4\pi \sqrt{|mn|}}{x}\right) \frac{\d x}{x}
			= \left(B_{2it}^- * ( x \cdot \phi)\right)(4\pi \sqrt{|mn|}).
		\end{equation}
	\end{mytheo}

	{\bf Remarks.} Here we have implemented some corrections of \cite{IwaniecSpectralBook} noted by Blomer, Harcos, and Michel \cite{BlomerHarcosMichel2007BoundsForModularLFunctionsInLevel}.
	Moreover, the right hand side slightly differs from the formulas in \cite{IwaniecSpectralBook} in that the roles of $m$ and $n$ are reversed, consistent with the remark following Definition \ref{def:KloostermanSum}. 
	
	It is important to emphasize that the same scaling matrices must occur in both the definition of the Kloosterman sum and in the definition of the Fourier coefficients.
	
	We occasionally use the above integral representations, but predominantly prefer  Mellin-type integrals, and we next state those formulas.
	The integral transforms $\phi_h$, $\phi_+$ and $\phi_-$ are realized as convolutions on the group $(\R^+, \frac{\d x}{x} )$ and therefore their Mellin transforms can be easily computed.

	\begin{myprop}\label{prop:MellinOfKuznetsovBessel}
		The integral transforms $\phi_h$ and $\phi_\pm$ have the alternative formulas
		\begin{equation}\label{eq:holomorphicIntegralTransformInMellin}
			\phi_h(k) 
			=  \frac{1}{2\pi i} \int_{(1)} \frac{2^{s-1 }\Gamma\left(\frac{s + k-1}{2}\right)}{\Gamma\left(\frac{k+1-s}{2}\right)} 
				\widetilde{\phi}(s + 1) (4\pi \sqrt{mn})^{-s} \d s,  
		\end{equation}
		and
\begin{equation}
\label{eq:plusminusIntegralTransformInMellin}
\phi_{\pm}(t) =  \frac{1}{2 \pi i} \int_{(2)} h_{\pm}(s,t) \widetilde{\phi}(s +1) (4\pi \sqrt{mn})^{-s} \d s,
\end{equation}
where
\begin{equation*}
\label{eq:hplusminusDefinition}
h_{\pm}(s,t) = \begin{cases}
 \frac{1}{\pi} 2^{s-1 }\cos(\pi s/2) \Gamma(\tfrac s2  + it )\Gamma(\tfrac s2 -it), \qquad &\pm = + \\
 \frac{1}{\pi} 2^{s-1 } \cosh(\pi t) \Gamma(\tfrac{s}{2} + it) \Gamma(\tfrac{s}{2} - it), \qquad &\pm = -.
\end{cases}
\end{equation*}	
	\end{myprop}

	\begin{proof}
		By Mellin inversion, we have
		\begin{equation}
		\label{eq:phihofkformula}
			 \phi_h(k) 
			= \frac{1}{2\pi i } \int_{(1)} \mathcal{M}(J_{k-1} * (x \cdot \phi),s) (4\pi \sqrt{mn})^{-s} \d s .
		\end{equation}
		The  
		 Mellin transform satisfies the property
		$
			\mathcal{M}(f * g, s) = \mathcal{M}(f, s) \mathcal{M}(g,s).
		$
		The Mellin transform of the $J$-Bessel function (see \cite[6.8 (1)]{ErdelyiTablesVol1}) is given as 
		\begin{equation} \label{eq:JBesselMellinTransform}
			\int_0^\infty J_\nu(x) x^s \frac{\d x}{x} = \frac{2^{s-1}\Gamma(\frac{s + \nu }{2})}{\Gamma(\frac{\nu  -s }{2} + 1)}.
		\end{equation}
		Also note that $\widetilde{x \phi}(s) = \widetilde{\phi}(s+ 1)$. Therefore \eqref{eq:phihofkformula} can be recast as \eqref{eq:holomorphicIntegralTransformInMellin}, as desired.

		For $\phi_-$ we have by \cite{ErdelyiTablesVol1} \S6.8 (26) that 
		\begin{equation}\label{eq:KBesselMellinTransform}
			\int_0^\infty K_{2it}(x) x^{s} \frac{\d x}{x} = 2^{s-2} \Gamma(\tfrac{s}{2} + it) \Gamma(\tfrac{s}{2} - it),
		\end{equation}
		and therefore we obtain the minus case of \eqref{eq:plusminusIntegralTransformInMellin}.

The plus case follows from using \eqref{eq:JBesselMellinTransform}, the reflection formula for the gamma function, and the addition formulas for $\sin$. 
	\end{proof}

\subsection{Spectral large sieve} \label{subsec:spectralLargeSieve}
On $\Gamma_0(N)$, we normalize the Petersson inner product by
\begin{equation}
 \langle g_1, g_2 \rangle  = \int_{\Gamma_0(N) \backslash \mathbb{H}} g_1(z) \overline{g_2}(z) y^{\kappa} \frac{dx dy}{y^2}.
\end{equation}
Quoting from \cite{BlomerHarcosMichel2007BoundsForModularLFunctionsInLevel}, 
we have if $u_j$ ($g$, respectively) is a $L^2$-normalized cuspidal Hecke-Maass (holomorphic, resp.) newform of level $N$ with trivial nebentypus, then 
\begin{equation}
\label{eq:nuNormalizationMaass}
 |\nu_{\infty  j}(1)|^2 
 =   N^{-1} (N(1+ |t_j|))^{o(1)},
\quad \text{and} \quad
|\nu_{\infty g}(1)|^2
= N^{-1} (Nk)^{o(1)}.
\end{equation}
With the normalization \eqref{eq:nuDefinition}, and assuming $\a$ is an Atkin-Lehner cusp, the spectral large sieve inequalities give 
\begin{equation*}
 \sum_{|t_j| \leq T} \Big| \sum_{m \leq M} a_m \nu_{\mathfrak{a} j}(m) \Big|^2 \ll \Big(T^2 + \frac{M}{N}\Big)(MNT)^{\varepsilon} \sum_{m \leq M} |a_m|^2,
\end{equation*}
and
\begin{equation}\label{eq:spectralLargeSieveContinuousPart}
\sum_{\mathfrak{c} } \int_{|t| \leq T} \Big| \sum_{m \leq M} a_m \nu_{\mathfrak{a} \mathfrak{c}}(m,\tfrac12 + it) \Big|^2 \d t \ll \Big(T^2 + \frac{M}{N}\Big) (MNT)^{\varepsilon} \sum_{m \leq M} |a_m|^2,
\end{equation}
and
\begin{equation*}
 \sum_{k \leq T}  \sum_{g \in H_k(N)} \Big|\sum_{m \leq M} a_m \nu_{\mathfrak{a} g}(m)\Big|^2 \ll \Big(T^2 + \frac{M}{N}\Big)(MNT)^{\varepsilon} \sum_{m \leq M} |a_m|^2.
\end{equation*}

\subsection{Newforms and oldforms} \label{subsec:oldForms}
Atkin and Lehner showed the orthogonal decomposition
\begin{equation*}
 S_{\kappa}(N) = \bigoplus_{LM = N} \bigoplus_{f \in H_{\kappa}^*(M)} S_{\kappa}(L;f),
\end{equation*}
where $S_{\kappa}(L;f)$ is the span of forms $f_{\vert \ell}$, with $\ell  \mid  L$, where 
\begin{equation}
\label{eq:fslashelldef}
f_{\vert \ell}(z) = \ell^{{\kappa}/2} f(\ell z).
\end{equation}
Their proof works with virtually no changes to cover the case of Maass forms (which have weight $0$, in our context).  For the rest of this section, we focus on the Maass case, but with a general weight ${\kappa}$ (in order to most easily translate the results to the holomorphic case).  

The formula \eqref{eq:fslashelldef} means that (let us agree to drop the subscript $\infty$ when working with the Fourier expansion at $\infty$)
\begin{equation}
\label{eq:fvertellFourierCoefficientFormula}
 \nu_{f \vert_{\ell}}(n) = \ell^{1/2} \nu_{f}(n/\ell).
\end{equation}

Blomer and Mili\'{c}evi\'{c} have shown in \cite[Section 6]{BlomerMilicevic2015SecondMoment} that there exists a basis of $S_{\kappa}(L;f)$ of the following type.  Let $f^*$ denote a newform of level $M | N$, 
$L^2$-normalized as a \emph{level $N$} form, which implies 
$|\nu_{\infty f^*}(1)|^2 = N^{-1} (N(1+|t_j|))^{o(1)}$.
Then there exists an orthonormal basis for $S_{\kappa}(L;f)$ of the form $g_m = \sum_{\ell| L} c_{\ell, m} f^* \vert_{\ell}$, where $c_{\ell,m} \ll N^{\varepsilon}$.  
For an Atkin-Lehner cusp $\a$, we have $|\nu_{\mathfrak{a}  f^*}(1)|= |\nu_{\infty  f^*}(1)|$, by \eqref{eq:AtkinLehnerEigenfunction}  .

We need the following information on the Fourier coefficients of $f^* \vert_{\ell}$ at Atkin-Lehner cusps:
\begin{mylemma}
\label{lemma:FourierExpansionDifferentCusps}
Suppose $\mathfrak{a}$ is an Atkin-Lehner cusp of $\Gamma_0(N)$, and $f^*$ is a newform of level $M$ with $LM=N$.  Then the set of lists of Fourier coefficients 
\{$(\nu_{\mathfrak{a}  f^*\vert_{\ell}}(n))_{n \in \mathbb{N}}$ : $\ell | L$\}
is, up to signs, the same as the set of lists of Fourier coefficients 
\{$(\nu_{\mathfrak{\infty}  f^*\vert_{\ell}}(n))_{n \in \mathbb{N}}$ : $\ell | L$\}.
\end{mylemma}
See \cite[Lemma 2.5]{DoubleCosetPaper} for a proof.

It is crucial for our later purposes to bound the Hecke eigenvalues of newforms at primes dividing the level.
Let $f^*$ be a newform (Maass or holomorphic) of level $M$ as above. 
If $p|M$, then 
\begin{equation}\label{eq:smallHecke}
|\lambda_{f^*}(p)| = p^{-1/2},
\end{equation}
for which see \cite[Theorem 3 (iii)]{WinnieLiNewforms} or \cite[Theorem 3 (iii)]{Atkin-Lehner} (the proofs carry over to Maass forms with virtually no changes).

%


\section{Inert functions and oscillatory integrals}\label{section:InertFunctions}
  \subsection{Basic Definition}

    We begin with a class of functions defined by derivative bounds.
    \begin{mydefi}\label{inert}
A family $\{w_T\}_{T\in \mathcal{F}}$ of smooth  
functions supported on a product of dyadic intervals in $\R_{>0}^d$ is called $X$-inert if for each $j=(j_1,\ldots,j_d) \in \Z_{\geq 0}^d$ we have 
\begin{equation}\label{eq:inert}
C(j_1,\ldots,j_d):= \sup_{T \in \mathcal{F}} \sup_{(x_1, \ldots, x_d) \in \R_{>0}^d} X_T^{-j_1- \cdots -j_d}\left| x_{1}^{j_1} \cdots x_{d}^{j_d} w_T^{(j_1,\ldots,j_d)}(x_1,\ldots,x_d) \right| < \infty.
\end{equation}
\end{mydefi}

    In our desired applications, our family of inert functions $\{ w_T \}_{T \in \mathcal{F}}$ 
    will be indexed by tuples $T$ of the form $T = (M_1, M_2, N_1, N_2, N_3, C, a, \dots)$, as well as some other 
    parameters that arise as dual variables after Poisson summation. Each of these parameters is polynomially bounded in $q$.  Furthermore, the relevant values of $X$ will 
    be $c(\varepsilon) q^{\varepsilon}$ for some constant $c(\varepsilon)$.

In addition, the weight functions encountered in this paper will typically be represented in the form
\begin{equation}
 P(T) e^{i\phi(x_1, \dots, x_d)} w_T(x_1, \dots, x_d),
\end{equation}
where $P(T)$ is some simple function depending on the tuple $T$ only, $\phi(x_1, \dots, x_d)$ is the phase, and $w_T$ is an inert function.  We wish to understand how such a function behaves under Fourier and Mellin transformations.  In Section \ref{section:InertFourier}, we analyze the Fourier and Mellin transforms in case $\phi = 0$, and in Section \ref{subsec:stationaryPhase} we discuss the stationary phase analysis of the Fourier transform in the presence of a phase $\phi$.

%

  \subsection{Fourier and Mellin transforms}
    \label{section:InertFourier}
    Inert functions behave regularly under the Fourier transform.  Suppose that $w_T(x_1, \dots, x_d)$ is $X$-inert, and let
    \begin{equation*}
        \widehat{w_T}(t_1, x_2, \dots, x_d) = \intR w_T(x_1, \dots, x_d) e(-x_1 t_1) d x_1
    \end{equation*}
    denote its Fourier transform in the $x_1$-variable.  Suppose that the support of $w_T$ is such that $x_i \asymp X_i$ 
    for each $i$.  Now $\widehat{w_T}$ is not compactly-supported in $t_1$, so it will not be inert.  However, if we let 
    $W_{T, Y_1}(t_1, x_2, \dots x_d) = w_{Y_1}(t_1) \widehat{w_T}(t_1, x_2, \dots x_d)$ where $\{w_{Y_1} : Y_1 > 0 \}$ is 
    a $1$-inert family, supported on $t_1 \asymp Y_1$ (or $-t_1 \asymp Y_1$) then $X_1^{-1} W_{T, Y_1}$ forms an $X$-inert 
    family.  Moreover, we have by repeated integration by parts, that 
    \begin{equation*}
         W_{T,Y_1}(t_1, x_2, \dots, x_d) \ll X_1  \Big(1 + \frac{|t_1| X_1 }{X}\Big)^{-A} \asymp  \Big(1 + \frac{Y_1 X_1}{X}\Big)^{-A},
    \end{equation*}
    so that in practice we may restrict our attention to $Y_1 \ll \frac{X q^{\varepsilon}}{X_1}$.  See \cite{InertFunctionPaper} for more details.

A similar integration by parts argument also treats the Mellin transform, and we record the result as follows:
\begin{mylemma}
\label{lemma:MellinTransformInertSection}
	Let $w_T(x_1,x_2,\ldots, x_d)$ be a family of $X$-inert functions such that $x_1$ is supported in the dyadic interval $[X_1 , 2X_1]$.
	Let
	\[
		\widetilde{w}_T(s, x_2, \ldots, x_d) = \int_0^\infty w_T(x, x_2, \ldots , x_d) x^s \frac{\d x}{x}.
	\]
	Then we have $\widetilde{w}_T(s,x_2,\ldots, x_d) = X_1^s W_T(s,x_2, \ldots x_n)$ where $W_T(s, \cdot)$ 
	is a family of $X$-inert functions in all the remaining $x_i$, which is entire in $s$ and has rapid decay for $|\Im(s)| \gg X^{1+\varepsilon}$.
\end{mylemma}

\subsection{Stationary phase} \label{subsec:stationaryPhase}
Next we synthesize some results from \cite{BKY} and \cite{InertFunctionPaper}.
\begin{mylemma}
 \label{lemma:exponentialintegral}
 \label{lemma:statphase}
 Suppose that $w = w_T(t,t_2, \dots, t_d)$ is a family of $X$-inert functions supported on $t \asymp Z$, $t_i \asymp X_i$ for $i=2,\dots, d$.  Suppose that on the support of $w_T$, $\phi$ satisfies
\begin{equation}
\label{eq:phiderivatives}
 \frac{\partial^{a_1 + a_2 + \dots + a_d}}{\partial t^{a_1} \dots \partial t_d^{a_d}} \phi(t , t_2, \dots, t_d) \ll \frac{Y}{Z^{a_1}} \frac{1}{X_2^{a_2} \dots X_d^{a_d}},
 \end{equation}
i.e., $Y\inv \phi $ is $1$-inert. Suppose that $Y/X^2 \gg q^\delta$ for some $\delta>0$. Let
\begin{equation*}
 I = \intR w_T(t, t_2, \dots, t_d) e^{i \phi(t,t_2, \dots, t_d)} dt.
\end{equation*}
\begin{enumerate}
 \item If $|\frac{\partial}{\partial t} \phi(t,t_2,\dots,t_d)| \gg \frac{Y}{Z}$ for all $t$ in the support of $w_T$, then $I \ll_A q^{-A}$, for $A >0$ arbitrarily large.
\item If $\frac{\partial^2}{\partial t^2} \phi(t, t_2, \dots, t_d) \gg \frac{Y}{Z^2}$ for all $t, t_2, \dots, t_d$ in the support of $w$, and there exists $t_0 \in \mr$ such that $\phi'(t_0) = 0$ (here, $\phi'$ denotes the derivative with respect to $t$, and note $t_0$ is necessarily unique), then 
\begin{equation}
\label{eq:IasymptoticMainThm}
I =  \frac{Z}{\sqrt{Y}} e^{i \phi(t_0, t_2, \dots, t_d)} W_T(t_2, \dots, t_d) + O(q^{-A}),
\end{equation}
for some $X$-inert family of functions $W_T$.
 \end{enumerate}
\end{mylemma}
Part (1) above follows from Lemma 8.1 of \cite{BKY}.  The one-variable version of (2) above is contained in Proposition 8.2 of \cite{BKY}, which was improved to many variables in \cite{InertFunctionPaper}.

\subsection{A convention}
\label{section:InertConvention}
We often re-normalize a family of inert functions.  For a simple example to illustrate this, say $w_T(x)$ is $X$-inert,   supported on $x \asymp N$.  We can write $x^{-1/2} w_T(x) = N^{-1/2} W_T(x)$, where $W_T(x) = (x/N)^{-1/2} w_T(x)$.  Then $W_T$ forms a new $X$-inert family with a different list of constants $C(j)$.  When doing this too many times it becomes difficult to find notation for all the new functions that arise, so we may on occasion replace $W_T$ by $w_T$, which is supposed to represent a generic inert function.

Another useful convention is, when focusing only a particular variable (say  $n$), we may write  $w_N(n,\cdot)$  where the $\cdot$ is a placeholder for the remaining variables.  Writing all the variables is  unwieldly, and the notion of inertness keeps track of the important behavior of the weight function with respect to the remaining variables.

We will also say that a family of inert functions $\left\{w_T(x_1,\ldots, x_d)\right\}$ such that each variable $x_i$ is supported in $[X_i, 2X_i]$ is  \emph{very small} to mean a quantity which is of size $O_A((X_1\dots  X_dq)^{-A})$ for every $A>0$, and uniformly in the family $T\in \mathcal{F}$.
More generally, we will use this terminology ``very small'' for more general quantities, not just inert functions.  In practice, we will largely ignore very small error terms.

\section{Preliminaries } \label{sec:preliminaries}

\subsection{Petersson Trace Formula}
   The Petersson trace formula reads
    \[
        \sum_{f \in H_{\kappa}(q)} w_f \lambda_f(n) \lambda_f(m) 
        = \delta_{n=m} 
        + 2\pi i^{-\kappa} \sum_{c \equiv 0 \shortmod{q}} 
        \frac{S(m,n;c)}{c} J_{\kappa-1} \Big(\frac{4\pi \sqrt{mn}}{c}\Big),
    \]  
    where $w_f =     q^{-1+o(1)}$ are the Petersson weights.
    Define
    \[
        \mathcal{M} = \mathcal{M}(q) = \sum_{f \in H_{\kappa}(q)} w_f L(\tfrac12, f)^5.
    \]
    Our main result, Theorem \ref{thm:mainthm}, is equivalent to
    \begin{equation}\label{eq:goalAlpha}
        \mathcal{M} \ll_{\kappa,\vep} q^{\theta + \vep}.
    \end{equation}

\subsection{The approximate functional equations.}
Let $\kappa$ be a positive even integer, $q$ a prime, and $f$ a Hecke cusp form of weight $\kappa$ and level $q$.
Put 
\begin{equation*}
	\gamma(s,\kappa) = \pi^{-s} \Gamma\left( \frac{s + \frac{\kappa-1}{2}}{2}\right) \Gamma\left(\frac{s + \frac{\kappa+1}{2}}{2}\right).
\end{equation*}
Let $G_i$ $(i = 1,2)$ be an even entire function decaying rapidly in vertical strips such that $G_i(0) = 1$. 
Define
\[
	V_1(x) = \frac{1}{2\pi i} \int_{(1)} \frac{G_1(u)}{u} \frac{\gamma( \tfrac12+u, \kappa)}{\gamma(\tfrac12, \kappa)} x^{-u} \d u,  \quad V_2(x) = \frac{1}{2\pi i} \int_{(1)} \frac{G_2(u)}{u} \frac{\gamma(\tfrac12 +u, \kappa)^2}{\gamma(\tfrac12, \kappa)^2} x^{-u} \d u.
\]
If $x \gg q^{\vep}$ then by shifting the contour of integration arbitrarily far to the right, we obtain that $V_i(x) \ll_{\kappa, A} (xq)^{-A}$.  Here  and throughout, we view $\kappa$ as fixed, and $q$ as becoming large.
For later use, it will be important to assume $G_i(1/2)=0$.

\begin{myprop}
\label{prop:AFELsquared}
With notation as above, we have
\[
	L(\tfrac12,f)^2 = 2 \sum_{m=1}^{\infty} \frac{\lambda_f(m) \tau_2(m)}{\sqrt{m} } V\left(\frac mq\right),
\]
where $\tau_2(m)$ is the (two-fold) divisor function, and
\[
	V(x) = \sum_{(e,q) = 1}^\infty V_{2}(e^2 x)/e = \frac{1}{2\pi i} \int_{(1)} \widetilde{V_{2}}(u) \zeta_q(1 + 2u) x^{-u} \d u,
\]
where $\zeta_q(s) = (1- q^{-s}) \zeta(s)$ is the Riemann zeta function with the $q\th$ Euler factor missing.
\end{myprop}

\begin{proof}
By the Hecke relation, we have
\begin{equation}
\label{eq:LsquaredWithHecke}
	L^2(s,f) 
	= \sum_{m_1,m_2 = 1}^\infty \sum_{\substack{e | (m_1,m_2)\\ (e,q) =1 }} \frac{\lambda_f(m_1m_2/e^2)}{(m_1m_2)^{s}} 
	= \sum_{(e,q) =1}^\infty \frac{1}{e^{2s}} \sum_{m=1}^\infty \frac{\tau_2(m) \lambda_f(m) }{m^s}.
\end{equation}
Then from the functional equation $L^2(s,f) \gamma(s,\kappa)^2  q^s =: \Lambda(s,f)^2 = \Lambda(1-s,f)^2$ we get the formula, as in \cite[Theorem 5.3]{iwaniec2004analytic}.
\end{proof}

\begin{myprop}
\label{prop:AFELcubed}
	Let $\vep_f$ be the sign of the functional equation for $L(s,f)$. Then 
	\begin{equation}
	\label{eq:LcubedAFE}
		L(\tfrac12,f)^3 = (1 + \vep_f)^3 \sum_{\substack{a=1\\ (a,q) =1}}^\infty \frac{\mu(a)}{a^{3/2}}
		\sum_{n=1}^\infty \frac{\lambda_f(na) }{\sqrt{n} } \tau_3(n, F_{a, \sqrt{q}}),
	\end{equation}
	where 
	\begin{equation}
	\label{eq:tau3Def}
		\tau_3(n,F_{a, \sqrt{q}}) = \sum_{n_1n_2n_3 = n} F_a\left(\frac{n_1}{\sqrt{q}},\frac{n_2}{\sqrt{q}},\frac{n_3}{\sqrt{q}}\right),
	\end{equation}
	and
	\begin{multline}
	\label{eq:FaDef}
		F_a(x_1,x_2,x_3) = \sum_{\substack{e_1,e_2,e_3\\ (e_1e_2e_3,q)=1}}
		\frac{1}{e_1e_2e_3} V_{1}(a x_1 e_1e_2) V_{1}(ax_2 e_1e_3) V_{1}(ax_3 e_2e_3) \\
		=  \iiint\limits_{(1) (1) (1)} \prod_{i=1}^3 \frac{\gamma(\tfrac12 + u_i,\kappa) G(u_i) }{(ax_i)^{ u_i} \gamma(\tfrac12,\kappa) u_i} \zeta_q(1 + u_1 + u_2) \zeta_q(1 + u_1 + u_3) \zeta_q(1 + u_2 + u_3)  \frac{\d u_1 \d u_2 \d u_3}{(2\pi i)^3} .
\end{multline}
\end{myprop}
\noindent {\bf Remark.}  One may easily check that 
\begin{equation}\label{eq:FaInert}
 x_1^{j_1} x_2^{j_2} x_3^{j_3} 
 \frac{\partial^{j_1+j_2+j_3}}{\partial x_1^{j_1}  \partial x_2^{j_2}
 \partial x_3^{j_3}}
 F_{a}(x_1, x_2, x_3) \ll_{j_1, j_2, j_3, A, \varepsilon} \prod_{i=1}^{3} (ax_i)^{-\varepsilon} (1 + ax_i)^{-A}.
\end{equation}
In the terminology introduced later in Section \ref{section:InertFunctions}, the property \eqref{eq:FaInert} means that $F_a$ satisfies the same derivative bounds as an $X$-inert function with $X \ll q^\vep$, in the region $x_i \gg q^{-1/2}$, for all $i$. Similar derivative bounds hold for $V(x)$.
\begin{proof}
Using the approximate functional equation for each $L(1/2, f)$, and the Hecke relations, we obtain
	\begin{multline*}
		L(\tfrac12,f)^3 
		=  \sum_{\substack{e_1, e_2\\ (e_1e_2,q) =1}} \frac{(1 + \vep_f)^3}{e_1} \sum_{\substack{n_1,n_2,n_3\\ e_2 | (n_1n_2,n_3)}} \frac{\lambda_f\left(\frac{n_1n_2n_3}{e_2^2}\right)}{\sqrt{n_1n_2n_3}} 
		V_{1}\left(\frac{n_1e_1}{\sqrt{q}} \right) 
		V_{1}\left(\frac{n_2e_1}{\sqrt{q}}\right) 
		V_{1}\left(\frac{n_3}{\sqrt{q}}\right)\\
		=   \sum_{\substack{e_1, e_2\\ (e_1e_2,q) =1}} \frac{(1 + \vep_f)^3}{e_1 e_2} 
		\sum_{f_1f_2 = e_2}   \sum_{\substack{n_1,n_2,n_3\\ (n_1, f_2) =1}} \frac{\lambda_f(n_1n_2n_3)}{\sqrt{n_1n_2n_3}} V_{1}\left(\frac{n_1e_1f_1}{\sqrt{q}} \right) V_{1}\left(\frac{n_2e_1f_2}{\sqrt{q}}\right) V_{1}\left(\frac{n_3f_1f_2}{\sqrt{q}}\right)  .
	\end{multline*}
	Using M\"obius inversion to detect the coprimality condition with $\sum_{a|(n_1, f_2)} \mu(a)$,  re-ordering the summations, and renaming the summation variables gives
 the more symmetric form 
	\begin{multline*}
		L(\tfrac12,f)^3 = (1+\epsilon_f)^3 \sum_{\substack{a = 1\\ (a,q) = 1}}^{\infty } \frac{\mu(a)}{a^{3/2}} \sum_{\substack{e_1,e_2,e_3\\ (e_1 e_2 e_3,q)=1}}\frac{1}{e_1 e_2 e_3}
		\\
		\times
		\sum_{n_1,n_2,n_3} \frac{\lambda_f(an_1n_2n_3)}{\sqrt{n_1n_2n_3}} V_{1}\left(\frac{an_1e_1 e_2 }{\sqrt{q}} \right) V_{1}\left(\frac{an_2e_1 e_3}{\sqrt{q}}\right) V_{1}\left(\frac{an_3 e_2 e_3}{\sqrt{q}}\right).
	\end{multline*}
This is seen to be equivalent to \eqref{eq:LcubedAFE}.
\end{proof}

Now apply Propositions \ref{prop:AFELsquared} and \ref{prop:AFELcubed} to $\mathcal{M}$.  There is a nice simplification in Proposition \ref{prop:AFELcubed}, whereby we may replace $1+\epsilon_f$ by $2$, because if $\epsilon_f = -1$, then $L(1/2,f)^2 = 0$ anyway.
    Applying the Petersson trace formula then yields
    \[
       \tfrac{1}{16} \mathcal{M} = \mathcal{D} + 2 \pi i^{-\kappa} \mathcal{S},
    \]
    where $\mathcal{D}$ is the diagonal term,
    and 
    \begin{equation}
    \label{eq:Sdef}
	\mathcal{S} = \sum_{(a,q)=1}^\infty \frac{\mu(a)}{a^{3/2}} 
	\sum_{c\equiv 0 \shortmod{q}}
	\sum_{n,m}\frac{\tau_2(m) \tau_3(n,F_{a,\sqrt{q}}) S(m,na;c)}{c\sqrt{mn}  } J_{\kappa-1} \left(\frac{4\pi \sqrt{mna}}{c}\right) V\left(\frac{m}{q} \right).
\end{equation}

    It is easy to bound the diagonal term.
    \begin{mylemma}
        We have 
        \[
            \mathcal{D} \ll_{\vep } q^\vep.
        \]
    \end{mylemma} 
    
    This follows easily from the fact that the functions $V_1(y)$ and $V_2(y)$ decay rapidly as $y \to \infty$, and  using the bound $J_{\kappa - 1}(x) \ll x$ for $\kappa \geq 2$. 

    Proving Theorem \ref{thm:mainthm} is reduced to showing $\mathcal{S} \ll q^{\theta+\varepsilon}$.
We will return to $\mathcal{S}$ in Section \ref{section:dyadic}.  Meanwhile, Sections \ref{section:Kuznetsov} and \ref{section:InertFunctions}, which are self-contained, develop some material necessary for our manipulations of $\mathcal{S}$.

%

\section{First Poisson summation}
\label{section:dyadic}
Now we return to the analysis of $\mathcal{S}$ from \eqref{eq:Sdef}.
We open up the divisor functions using the formulas $\sum_m \tau_2(m) f(m) = \sum_{m_1,  m_2} f(m_1 m_2)$, and  the definition of $\tau_3(n,F_{a, \sqrt{q}})$ (cf. \eqref{eq:tau3Def}).

\subsection{Dyadic partition of unity}
Throughout this paper we will apply dyadic decompositions of important variables.
 Call a number $N$ \emph{dyadic} if 
    $N = 2^{k/2}$, for some integer $k$. 
A dyadic partition of unity is a partition of unity of the form $ \sum_{k \in \mathbb{Z}} \omega(2^{-k/2} x) \equiv 1$, for $x > 0$, where $\omega$ is a fixed smooth function with support on the dyadic interval $[1,2]$.  The family $\omega_N(x) = \omega(x/N)$ forms 
    a $1$-inert family of functions. 
Applying this to $\mathcal{S}$, we have    
\begin{equation}
    \label{eq:initialDyadic}
        \mathcal{S}  = \sum_{M_1, M_2, N_1, N_2, N_3,C \text{ dyadic}} \mathcal{S}_{M_1,M_2, N_1, N_2, N_3,C},
    \end{equation}  
where the dyadic numbers are restricted to be $\geq 2^{-1/2}$, and where
    \begin{multline}
    \label{eq:SAfterDyadicDecomposition}
	\mathcal{S}_{M_1, M_2, N_1, N_2, N_3, C}= \sum_{(a,q)=1}^\infty \frac{\mu(a)}{a^{3/2}} 
	\sum_{\substack{c \equiv 0 \shortmod{q} \\ c \asymp C}} \frac{1}{c}
	\sum_{n_1, n_2, n_3,m_1, m_2 }\frac{  S(m_1 m_2 ,n_1 n_2 n_3 a,c)}{\sqrt{m_1 m_2 n_1 n_2 n_3} } 
\\
	\times	
	J_{\kappa-1} \left(\frac{4\pi \sqrt{m_1 m_2 n_1 n_2 n_3a}}{c}\right) 
	 V\left(\frac{m_1 m_2}{q} \right) 
	 F_a\Big(\frac{n_1}{\sqrt{q}}, \frac{n_2}{\sqrt{q}}, \frac{n_3}{\sqrt{q}}\Big)
	w_T(m_1, m_2, n_1, n_2, n_3, c)
	. 
\end{multline}
    The letter $T$ here and throughout stands for the tuple of dyadic parameters, and we may use $\mathcal{S}_T$ as shorthand for the left hand side of \eqref{eq:SAfterDyadicDecomposition}.  For the main thrust of the argument, the precise form of $w_T$ is not important.  However, when calculating certain potential main terms, we have found it important to re-sum over the partition, in which case one should remember that $w_T$ may be expressed as
    \begin{equation*}
    w_T(m_1, m_2, n_1, n_2, n_3, c) = \omega\Big(\frac{m_1}{M_1}\Big) \dots \omega\Big(\frac{c}{C}\Big).
    \end{equation*}
  Let $M = M_1 M_2$, and $N = N_1 N_2 N_3$.    

    \begin{mylemma} \label{lem:capitalBounds}
        Unless
        \begin{equation}
        \label{eq:NMBound}
            M\ll q^{1+\varepsilon}
            \qquad \text{ and } \qquad 
            N_i \ll \frac{q^{1/2 + \vep}}{a },
        \end{equation}
       for all $i=1,2,3$, then we have
        \[
            \mathcal{S}_{T} \ll_A q^{-A},        
        \]
        for $A>0$ arbitrarily large. 
        Moreover, if $C > q^{3}$, we have
        \begin{equation*}
         \mathcal{S}_{T} \ll q^{\varepsilon}.
        \end{equation*}
    \end{mylemma}
  We will henceforth assume \eqref{eq:NMBound} (which implies $N \ll a^{-3} q^{3/2+\vep}$), and
        \begin{equation}\label{eq:CBound}
        C \leq q^{3}.
        \end{equation}

      \begin{proof}
       The bounds \eqref{eq:NMBound} follow from the rapid decay of the weight functions in the approximate functional equations.  The bound for $C > q^3$ holds using the Weil bound for Kloosterman sums, and $J_{\kappa-1}(x) \ll x$.
       %
     \end{proof}
    
By symmetry (Dirichlet's hyperbola method), we shall assume
    \begin{equation}
        M_1 \leq M_2.
    \end{equation}
Note
$M_1 \ll q^{1/2+\varepsilon}$.

\subsection{Poisson summation}
  Applying Poisson summation in $m_2$ modulo $c$, we obtain
\begin{equation}
\label{eq:SafterFirstPoisson}
	\mathcal{S}_T =   \sum_{(a,q) = 1} \frac{\mu(a)}{a^{3/2}} 
	\sum_{m_1,n_1, n_2, n_3} \frac{1}{\sqrt{m_1 n_1 n_2 n_3}}  \sum_{c \equiv 0 \shortmod{q} }  \frac{1}{c^2} \sum_{k \in \mathbb{Z}} H(k) I(k),
\end{equation}
where (with shorthand $n= n_1 n_2 n_3$)
\begin{equation}
	H(k) = H(k, m_1, na ;c) = \sum_{x \shortmod c}  S(m_1 x, na;c) e\left(\frac{kx}{c}\right),
\end{equation}
and 
\begin{equation}
\label{eq:Idef}
	I(k) =  I(m_1,k,n_1,n_2,n_3,a,c) = \int_{0}^\infty
	e\Big(\frac{-k t}{c}\Big) J_{\kappa-1}\left(\frac{4\pi \sqrt{m_1na t}}{c}\right)  w_{M_2}(t, \cdot ) \frac{\d t}{\sqrt{t}}.
\end{equation}
For the notation $w_{M_2}(t,\cdot)$, recall the convention described in Section \ref{section:InertConvention}. Furthermore note that the definition of the inert function $w_T$ has been altered to include the function $V$ and $F_a$ from the second line of \eqref{eq:SAfterDyadicDecomposition}. 

We now apply a dyadic partition of unity to the $k$-sum, and let $\omega(k/K)$ be a generic such piece.  To ease the notation, we simply add $K$ to the long tuple of parameters already appearing in \eqref{eq:SAfterDyadicDecomposition}; we are already writing $T$ as shorthand for this long tuple, and we shall continue this practice.  Let $I_K = I_{(M_1,M_2,N_1,N_2,N_3,a, C, K)} = \omega(k/K) I$. Then for $k > 0$, $I = \sum_{K\text{ dyadic}} I_K$. By re-defining the inert function in \eqref{eq:Idef} to incorporate $\omega(k/K)$, we may also view $I_K$ as an instance of \eqref{eq:Idef}.

{\bf Remark.} We may without loss of generality assume that $k > 0$. The negative values of $k$ give rise to terms that are complex conjugates of their positive counterparts. Secondly, $H(0) = 0$ unless $c | m_1$, and those terms only contribute $O(q^{-A})$ since $m_1 \ll q^{1/2+\varepsilon}$ and $q|c$.

\subsection{The arithmetic function}

    We now compute the arithmetic sum $H(k)$. Immediately from the definition, we obtain
    \begin{equation}\label{eq:HfirstReduction}
        \frac{1}{c} H(k) 
        = \frac{1}{c} \sums_{u \shortmod{c}} \sum_{x \shortmod{c}} 
        e\lt(\frac{(m_1u + k) x + na\overline{u} }{c}\rt) 
        = \sums_{u \shortmod{c} } \delta_{m_1u \equiv -k \shortmod{c}} 
        e\Big(\frac{na\overline{u}}{c}\Big).
    \end{equation}
    
    One would like to simply substitute $u \equiv -k\overline{m_1}\pmod{c}$,  
    however this is not possible because it is not guaranteed that $m_1$ (or $k$) are coprime to $c$. For this reason, we employ a factorization of $c$ and the Chinese remainder theorem as follows.
    
    Write
    \begin{equation}\label{eq:ckDecomposition}
        c = c_0 c_2, \quad \text{ and } \quad k = k_0 k_1, 
    \end{equation}    
    where the factorizations may be written locally, using the notation $\nu_p(n) = d$ for $p^d || n$, as
    \begin{align*}
     c_0 = 
     \prod_{ \nu_p(c) > \nu_p(k)} p^{\nu_p(c)}, 
     \qquad 
     c_2 = \prod_{1 \leq \nu_p(c) \leq \nu_p(k)} p^{\nu_p(c)}, \\
     k_0 = 
     \prod_{ \nu_p(k) \geq \nu_p(c)} p^{\nu_p(k)}, 
     \qquad 
     k_1 = \prod_{1 \leq \nu_p(k) < \nu_p(c)} p^{\nu_p(k)}.
    \end{align*}
Alternatively,  using the notation $n^* = \prod_{p|n} p$ we have
    \begin{equation}\label{eq:ckOneTwoConditions}
    (c_0, k_0) = 1, \qquad c_2 | k_0, \qquad k_1 k_1^* | c_0 		,
    \end{equation}
    and these conditions characterize $c_0, c_2, k_0, k_1$.  Note that $(c_2, k_1) = 1$ automatically from the other conditions, indeed we also have $(c_2, c_0) = (k_1, k_0) = 1$.

    The congruence condition $m_1u \equiv - k \pmod c$ in \eqref{eq:HfirstReduction} is solvable with $(u,c) = 1$ if and only if $(m_1, c) = (k, c)$.  The conditions \eqref{eq:ckDecomposition}, \eqref{eq:ckOneTwoConditions} give $(k,c) = k _1 c_2$, and so we impose the condition $k_1 c_2 = (m_1, c_0 c_2) = (m_1, \frac{c_0}{k_1} k_1 c_2)$.
    Thus we define
    \begin{equation}\label{eq:m1PrimeDefinition}
        m_1 = k_1 c_2 m_1',
    \end{equation}
  where the new variable $m_1'$ is only subject to the restriction
        $$\Big(m_1', \frac{c_0}{k_1}\Big) = 1 \Leftrightarrow (m_1', c_0) =1,$$
        where we have used that $c_0/k_1$ shares the same prime factors as $c_0$.

\textbf{Remark.}  In $\mathcal{S}$, we have $q | c$.  If $q | k_1 c_2$, this means $q | m_1$, but we have $m_1  \ll  q^{1/2 + \varepsilon}$, so the condition $q|c$ may be freely replaced with $q|c_0$, and we may assume
\begin{equation}
\label{eq:qandk1areCoprime}
(q,k_1) = 1.
\end{equation}

    \begin{myprop} \label{prop:Hreduction}
        Given the notation above, 
        \begin{equation}\label{eq:Hfinalized}
           \frac{1}{c}  H(k, m_1, an;c) = e\Big(- \frac{nam_1'\overline{k_0}}{c_0}\Big) S(na,0;c_2) k_1 \delta_{k_1 | na}\delta_{(m_1', c_0) =1} \delta_{c,k},
        \end{equation}
        where $\overline{k_0}$ indicates multiplicative inverse modulo $c_0$ and where 
       $\delta_{c,k} = 1$ if \eqref{eq:ckDecomposition} and \eqref{eq:ckOneTwoConditions} hold, and $\delta_{c,k} = 0$ otherwise.
    \end{myprop}

    \begin{proof}

        First, we note that $m_1u \equiv - k \pmod{c}$ (that is, $m_1'k_1c_2 u\equiv -k_0 k_1 \pmod{c_0 c_2}$) 
        is equivalent to 
        $
            m_1' u \equiv -  \frac{k_0}{c_2} \pmod{c_0/k_1 }.
       $
        In other words,
        \begin{equation}
        \label{eq:ucongruence}
            \overline{u} \equiv -m_1' \overline{(k_0/c_2)} \pmod{c_0/k_1}.
        \end{equation}
        Here $\overline{(k_0/c_2)}$ can be taken to be the multiplicative inverse modulo $c_0$, since 
        every prime that divides $c_0$ also divides $c_0/k_1$ (via \eqref{eq:ckOneTwoConditions}) 
     
     Now we apply the Chinese remainder theorem to the pair $c_0$ and $c_2$, giving
        \begin{align*}
           \frac{1}{c} H(k, m_1, an;c) 
            =& \sums_{u \shortmod{c} } \delta_{\overline{u} \equiv -m_1' \overline{(k_0/c_2)} \shortmod{\frac{c_0}{k_1}}} e\Big(\frac{na\overline{u}(c_0\overline{c_0} + c_2\overline{c_2}) }{c_0c_2}\Big)\\ 
            =& \sums_{u \shortmod{c_2} } e\Big(\frac{na\overline{u c_0}}{c_2}\Big) 
            \sums_{u \shortmod{c_0}} \delta_{\overline{u} \equiv -m_1' \overline{(k_0/c_2)} \shortmod{\frac{c_0}{k_1}}} e\Big(\frac{na\overline{u}\overline{c_2}}{c_0}\Big). 
        \end{align*}
        The sum modulo $c_2$ is a Ramanujan sum, and for the sum modulo $c_0$ we replace $u$ by $\overline{u}$, giving
        \begin{equation*}
        \frac{1}{c} H(k,m_1,an;c) = S(na,0;c_2) \sums_{u \shortmod{c_0}} \delta_{u \equiv -m_1' \overline{(k_0/c_2)} \shortmod{\frac{c_0}{k_1}}} e\Big(\frac{  na u \overline{c_2}}{c_0}\Big).
        \end{equation*}

        The congruence  restriction  
        on $u$ modulo $c_0$
        may be expressed as
        \begin{equation}\label{eq:congruenceCondition}
            u \equiv - m_1' \overline{(k_0/c_2)} + v \frac{c_0}{k_1} \mymod{c_0}, \quad \text{ with }  \quad v \mymod{k_1}. 
        \end{equation}
        Here $v$ runs over \emph{all} residue classes modulo $k_1$, because as long as $u$ is coprime to $c_0/k_1$
        it is also coprime to $c_0$.
        Thus
        \begin{equation*}
         \begin{split}
            \frac{1}{c} H(k, m_1, an;c)  =& S(na,0;c_2) \sum_{v \shortmod{k_1} } e\Big(\frac{na (-m_1' \overline{(k_0/c_2)} +v\frac{c_0}{k_1}) \overline{c_2}}{c_0}\Big)\\
             =& S(na,0;c_2) e\Big(- \frac{nam_1' \overline{k_0 }}{c_0} \Big) k_1 \delta_{k_1 | na}. \qedhere
         \end{split} 
         \end{equation*}
    \end{proof}

Inserting the conclusion of Proposition \ref{prop:Hreduction} into \eqref{eq:SafterFirstPoisson}, and imposing \eqref{eq:qandk1areCoprime}, we get
\begin{multline}
\label{eq:SafterFirstPoissonEvaluated}
\mathcal{S}_T =   \sum_{(a,q) = 1} \frac{\mu(a)}{a^{3/2}} \sum_{m_1',n_1, n_2, n_3} \frac{1}{\sqrt{ m_1' n_1 n_2 n_3}} \sum_{\substack{(c_0,m_1') = 1\\ c_0\equiv 0 \shortmod{q}}} 
\sum_{(k_0,c_0) = 1} \sum_{\substack{k_1 | \frac{c_0}{k_1^*} \\ (k_1, q) = 1} }  k_1^{1/2} \delta_{k_1 | na} 
\\
\sum_{c_2 | k_0}  \frac{1}{c_0 c_2^{3/2}}
e\Big(- \frac{n_1 n_2 n_3 am_1'\overline{k_0}}{c_0}\Big) S(n_1 n_2 n_3 a,0;c_2) I(m_1'k_1c_2 , k_0k_1,  n_1, n_2, n_3, a,  c_0c_2),
\end{multline}
plus a small error term.

  \subsection{Analysis of integral transform}    
The asymptotic behavior of $I_K$ depends on if $\frac{\sqrt{aMN}}{C} \gg q^{\varepsilon}$ or not, since this dictates  whether the Bessel function is oscillatory or not.
 
   \begin{mylemma}[Pre-Transition]
   \label{lemma:IKpreTransition}
 Let $ I_K(k)$ be defined via \eqref{eq:Idef}.  If
 \begin{equation}
 \label{eq:PreTransition}
  \frac{\sqrt{aM  N}}{C} \ll q^{\varepsilon}, 
 \end{equation}
then
\begin{equation}
\label{eq:IKasymptoticPreTransition}
M_2^{1/2} I_K(k) = \Big(\frac{\sqrt{aMN}}{C}\Big)^{\kappa -1} M_2 w_T(\cdot) ,
\end{equation}
where $w_T(\cdot)$ is an $X$-inert function with $X \ll q^{\varepsilon'}$.  Furthermore,  $I_K$ is very small unless
\begin{equation}
 \frac{K M_2}{C} \ll q^{\varepsilon}.
\end{equation}
\end{mylemma}

\begin{mylemma}[Post-Transition]
\label{lemma:IKpostTransition}
If 
 \begin{equation}
 \label{eq:PostTransition}
  \frac{\sqrt{aM  N}}{C} \gg q^{\varepsilon},
 \end{equation}
then
\begin{equation}
\label{eq:IKpostTransition}
M_2^{1/2} I_K(k) = \frac{C M_2}{(aMN)^{1/2} } e\Big(\frac{m_1 n a}{ck}\Big) w_T(\cdot) + O((kq)^{-A}),
\end{equation}
where $w_T(\cdot)$ is an $X$-inert function with $X \ll q^{\varepsilon'}$.  Furthermore, $I_K$ is very small unless
\begin{equation}
\label{eq:Ksize}
 K \asymp \frac{(aMN)^{1/2}}{M_2}.
\end{equation}
\end{mylemma}

\begin{proof}[Proof of Lemma \ref{lemma:IKpreTransition}]
 Suppose that \eqref{eq:PreTransition} holds.  Then the Bessel function is not oscillatory, and $J_{\kappa-1}(x) = x^{\kappa-1} W(x)$ where $x^j \frac{d^j}{dx^j}W(x) \ll X^{j} $, with $X \ll q^{\varepsilon}$ (note that if $1 \ll x \ll q^{\varepsilon}$, it is still valid to factor out $x^{\kappa -1}$ though there is a small loss of efficiency 
 by doing so).  This is the same derivative bound as an $X$-inert function, so it may be absorbed into the inert function $w_T$.
 Then by the discussion in Sections \ref{section:InertFourier}  and \ref{section:InertConvention}, we have
 \begin{equation}
 M_2^{1/2} I_K(k) = \Big(\frac{\sqrt{aMN}}{C}\Big)^{\kappa -1} M_2 w_T(\cdot), 
 \end{equation}
 and that $I_K(k) \ll (kq)^{-A}$ if $K \gg \frac{C}{M_2} q^{\varepsilon}$.  Here $w_T$ is $X$-inert with $X \ll q^{\varepsilon'}$.
 \end{proof}
 
 \begin{proof}[Proof of Lemma \ref{lemma:IKpostTransition}]
 Now suppose that \eqref{eq:PostTransition} holds.  Then we use that for $x \gg 1$, we have
 $$J_{\kappa - 1}(x) = \sum_{\pm} x^{-1/2} e^{\pm ix} W_{\pm}(x),$$ where  $W_{\pm}$ satisfies the same derivative bounds as a $1$-inert function.  Thus
 \begin{equation*}
\sqrt{M_2} I_K(k) = \sum_{\pm} \frac{C^{1/2}}{(aMN)^{1/4} }  \intR  w_{M_2}(t, \cdot) 
e\Big(\frac{-kt}{c}\Big) e\Big(\frac{\pm 2 \sqrt{t m_1 n a}}{c}\Big) dt,
 \end{equation*}
 where $w_T(t)$ is $q^{\varepsilon}$-inert (in all previously-declared variables), and supported on $t \asymp M_2$.

 Since $k>0$, if the $\pm$ sign is $-$, then Lemma \ref{lemma:exponentialintegral} part 1 shows that the integral is very small.  Therefore, we focus on the case where the $\pm$ sign is $+$, in which case we obtain an oscillatory integral with phase
\begin{equation*}
 \phi(t) = -\frac{kt}{c} + \frac{2 \sqrt{t m_1 n a}}{c}.
\end{equation*}
We have
\begin{equation*}
 \phi'(t) = -\frac{k}{c} + \frac{\sqrt{m_1 na}}{c \sqrt{t}}, 
 \qquad 
 \phi''(t) = -\frac{\sqrt{m_1 n a}}{2 c t^{3/2}}.
\end{equation*}
There is a unique point $t_0$ where $\phi'(t_0) = 0$, namely
\begin{equation*}
 t_0 = \frac{m_1 n a}{k^2}.
\end{equation*}
If it is not the case that $t_0 \asymp M_2$ (with large but absolute implied constants), then we have $|\phi'(t)| \gg \frac{\sqrt{aM N}}{c M_1}$ throughout the support of the weight function, and Lemma \ref{lemma:exponentialintegral} part 1 again shows the integral is small.  If $t_0 \asymp M_2$, then the location of $t_0$ is compatible with the support of $\phi$, and stationary phase (Lemma \ref{lemma:exponentialintegral}) shows that 
\begin{equation*}
 \intR w_{T}(t) e\Big(\frac{-kt}{c}\Big) e\Big(\frac{ 2 \sqrt{t m_1 n a}}{c}\Big) dt = \frac{C^{1/2} M_2}{(aMN)^{1/4}} e\Big(\frac{m_1 n a}{ck}\Big) w_{T}\Big(\frac{m_1 n a}{k^2}, \cdot \Big) + O((kq)^{-A}),
\end{equation*}
where $w_T$ on the right hand side is $q^{\varepsilon}$-inert, and supported on $m_1 n a /k^2 \asymp M_2$.
 \end{proof}   
    
\section{Reciprocity and other arithmetical manipulations}  
\label{section:reciprocity}
Next we reorder the summation $\mathcal{S}_T$ in \eqref{eq:SafterFirstPoissonEvaluated}. We bring the sum over $n = n_1 n_2 n_3$ to the inside, and open up the Ramanujan sum $S(na,0;c_2) = \sum_{d | (na,c_2) } d \mu(c_2/d)$.
This gives
\begin{equation}
\label{eq:SintoS'}
\mathcal{S}_T =   \sum_{(a,q) = 1} \frac{\mu(a)}{a^{3/2}} 
\sum_{c_2} \frac{1}{c_2^{3/2}} 
\sum_{d|c_2} d \mu(c_2/d)
\sum_{(k_1,q) = 1   }  k_1^{1/2}
\sum_{m_1'} \frac{1}{\sqrt{m_1'}} 
\mathcal{S}' + O(q^{-A}),
\end{equation}
where
\begin{multline*}
 \mathcal{S}' = 
 \sum_{\substack{(c_0,m_1') = 1\\ c_0\equiv 0 \shortmod{q k_1 k_1^*} }} 
 \frac{1}{c_0}
\sum_{\substack{(k_0,c_0) = 1 \\ k_0 \equiv 0 \shortmod{c_2}}}   
\\
\times
\sum_{\substack{n_1 n_2 n_3 a \equiv 0 \shortmod{k_1} \\ n_1 n_2 n_3a \equiv 0 \shortmod{d}}}
\frac{e\Big(- \frac{n_1 n_2 n_3am_1'\overline{k_0}}{c_0}\Big) }{\sqrt{ n_1 n_2 n_3}}
I_K(m_1'k_1c_2 , k_0k_1,  n_1, n_2, n_3,a,  c_0c_2).
\end{multline*}
We shall not obtain any significant cancellation in the outer summation variables appearing in $\mathcal{S}$ (except for a ``fake'' main term, in Section \ref{section:TwopiAreZero}) 
, but substantial cancellation is required in $c_0, k_0$, and the $n_i$.

Note that since $d|c_2$, $c_2 | k_0$, $k_1 | c_0$, and $(c_0, k_0) = 1$, we have that $(d, k_1) = 1$.  Then the congruences in the sum over $n = n_1 n_2 n_3$ are equivalent to $an \equiv 0 \pmod{dk_1}$, which in turn is equivalent to $n \equiv 0 \pmod{\delta_1}$, where
\begin{equation}
 \delta_1 = \frac{k_1 d}{(a,k_1 d)}.
\end{equation}
Note that $(\delta_1,q) = 1$, and $(k_0, q) = 1$.

Since $(c_0 , k_0 ) = 1$, we have the reciprocity formula 
\begin{equation*}
-\frac{\overline{k_0 }}{c_0 } \equiv \frac{\overline{c_0 }}{k_0} - \frac{1}{c_0 k_0 } \pmod 1.
\end{equation*}
Define
\begin{equation}
\label{eq:Jdef}
\begin{split}
 J(n_1,n_2,n_3, a, m_1', c_0, k_0, c_2, k_1) &= e\left(-\frac{n  a m_1' }{c_0 k_0 }\right)  I_K(m_1'k_1c_2 , k_0k_1, n_1,n_2,n_3,a , c_0c_2)
 \\
&  = e\left(-\frac{n  a m_1  }{c k }\right) I_K(m_1  , k ,n_1, n_2, n_3, a, c ), 
 \end{split}
 \end{equation}
where in the second line above we have expressed $J$ in terms of the earlier variable names \eqref{eq:ckDecomposition}, \eqref{eq:m1PrimeDefinition}.  This is sometimes convenient for tracking the sizes of certain quantities, for example, $\frac{n am_1'}{c_0 k_0} \asymp \frac{NaM_1}{CK}$.

Our next goal is to apply Poisson summation in the $n$-variables, and to do that we need some preparatory moves.  First, consider a formal sum of the form
\begin{equation}
\label{eq:tau3sum}
	\sum_{\substack{m_1,m_2,m_3 \geq 1\\ m_1m_2m_3 \equiv 0 \shortmod{r}}} J(m_1,m_2,m_3).
\end{equation}
The product $m_1m_2m_3$ runs over integers of the form $mr$ with $m\geq 1$. Now define
\[
	r_1 = (m_1,r), \quad m_1 = m_1' r_1, 
\]
so $(m_1', \frac{r}{r_1}) = 1$.  
Then we have $m_1' m_2 m_3 = \frac{r}{r_1}m$.  Continuing this process, define $r_2 = (m_2, \frac{r}{r_1})$, $m_2 = m_2' r_2$, so we have $m_1' m_2' m_3 = \frac{r}{r_1 r_2} n$ with $(m_2', \frac{r}{r_1 r_2}) = 1$.  Finally, let $r_3 = (m_3, \frac{r}{r_1 r_2})$, and set $m_3= m_3' r_3$, whence $(m_3', \frac{r}{r_1 r_2 r_3}) = 1$.  Now we have $m_1' m_2' m_3' = \frac{r}{r_1 r_2 r_3} m$, and the coprimality conditions mean that 
$(m_1'm_2' m_3', \frac{r}{r_1 r_2 r_3}) = 1$, so $r_1 r_2 r_3 = r$.

 Therefore, translating this discussion into formulas, we have that \eqref{eq:tau3sum} equals
\begin{equation*}
 \sum_{r_1 r_2 r_3 = r} \sum_{(m_1', r_2 r_3) = 1} \sum_{(m_2', r_3) = 1} \sum_{m_3'} J(r_1m_1',r_2m_2', r_3m_3').
\end{equation*}
Using M\"obius inversion, we have that \eqref{eq:tau3sum} equals
\begin{equation}
\label{eq:tau3sumLater}
 \sum_{r_1 r_2 r_3 = r} 
 \sum_{e_1 | r_2 r_3} \mu(e_1) \sum_{e_2 | r_3} \mu(e_2) 
 \sum_{n_1, n_2, n_3 \geq 1 } J(r_1 e_1 n_1 , r_2 e_2 n_2 , r_3 n_3 ).
\end{equation}

Applying this formula to $\mathcal{S}'$, we obtain
\begin{equation}
\label{eq:S''andS'}
 \mathcal{S}' = 
 \sum_{r_1 r_2 r_3 = \delta_1} 
 \sum_{e_1 | r_2 r_3} \sum_{e_2 | r_3} \mu(e_1) \mu(e_2)
 \mathcal{S}'',
\end{equation}
where
\begin{multline*}
\mathcal{S}'' = 
 \sum_{\substack{(c_0, m_1') = 1 \\ c_0 \equiv 0 \shortmod{qk_1 k_1^*}}} 
 \frac{1}{c_0}
\sum_{\substack{( k_0,c_0) = 1 \\ k_0 \equiv 0 \shortmod{c_2} }}   
 \sum_{n_1, n_2, n_3 \geq 1} \frac{ e\Big(\frac{e_1 e_2 \delta_1 am_1' n_1 n_2 n_3
 \overline{c_0}}{k_0 }\Big) }{\sqrt{ \delta_1  e_1 e_2 n_1 n_2 n_3}} 
 \\
 \times
 J( r_1e_1n_1, r_2e_2n_2,r_3n_3,  a, m_1', c_0, k_0 , c_2,k_1).
\end{multline*}
We remark that in doing so we changed variables by
\begin{equation}
\label{eq:niChangesOfVariables}
 n_1 \rightarrow r_1  e_1 n_1, \qquad n_2 \rightarrow  r_2 e_2 n_2 \qquad  n_3 \rightarrow r_3 n_3 .
\end{equation}
With the earlier definition of $n$ as $n_1n_2n_3$, then \eqref{eq:niChangesOfVariables} is equivalent to
	$n \rightarrow e_1e_2\delta_1 n$.

Next define $g_0 = (e_1 e_2 \delta_1 a m_1', k_0)$, and write 
\begin{equation*}
k_0 = g_0 k_0', \qquad \text{and} \qquad \delta_2= \frac{e_1 e_2 \delta_1 am_1'}{g_0}.
\end{equation*}
There are some implicit conditions on the variables that we wish to record explicitly.
Note that since $(k_1, k_0) = 1$, and $d|c_2 | k_0$, we may write $g_0$ as
\begin{equation*}
 \label{eq:g0def}
 g_0 = (e_1 e_2 \frac{ad}{(a,d)} m_1', k_0) = d (e_1 e_2 \frac{a}{(a,d)} m_1', \frac{k_0}{d}),
\end{equation*}
and in particular $d | g_0$, a property that will be important in Section \ref{section:ContinuousSpectrumBound}. Also note that since none of the factors of $\delta_2$ are divisible by $q$ (since $q$ is prime, $(a,q) = 1$, and the original $m_1$ and $n_i$-variables are $\ll q^{1/2+\varepsilon}$), we have
\begin{equation}\label{eq:delta2Coprime}
(\delta_2,q) = 1.
\end{equation}
We may also observe that $(g_0, qk_1) = 1$
since $g_0 | k_0$, $(k_0, c_0) = 1$, and $c_0 \equiv 0 \pmod{qk_1 k_1^*}$.  From $k_1 | a \delta_1$, we also conclude that
\begin{equation}
\label{eq:k1dividesdelta2}
k_1 | \delta_2.
\end{equation}

Therefore,
\begin{equation}
\label{eq:S''intoS'''}
 \mathcal{S}'' = 
 \sum_{\substack{g_0 | e_1 e_2 \delta_1 am_1' \\ g_0\equiv 0 \shortmod{d}}} \mathcal{S}''',
\end{equation}
where
\begin{multline}
\label{eq:S'''prePoisson}
 \mathcal{S}''' = 
 \sum_{\substack{(c_0,g_0 m_1') = 1\\ c_0\equiv 0 \shortmod{q k_1 k_1^*} }} 
 \frac{1}{c_0}
\sum_{\substack{( k_0',\delta_2 c_0) = 1 \\ k_0' \equiv 0 \shortmod{\frac{c_2}{(g_0, c_2)}} }}   
\sum_{\substack{n_1, n_2, n_3  \geq 1 }}
\frac{e\Big(\frac{  \delta_2   n_1 n_2 n_3
 \overline{c_0}}{k_0' }\Big) }{\sqrt{ \delta_1  e_1 e_2 n_1 n_2 n_3 }} 
 \\
 \times
  J(r_1e_1n_2, r_2 e_2 n_2,r_3 n_3 , a, m_1', c_0, g_0 k_0' , c_2,k_1)
.
\end{multline}


\section{Triple Poisson}
\subsection{Poisson summation formula}
Our next step is to apply Poisson summation in $n_1, n_2, n_3$ modulo $k_0'$ to \eqref{eq:S'''prePoisson}, to which end we state the following general version.
\begin{myprop} 
\label{thm:TriplePoisson}
	Let $J$ be any smooth and compactly-supported function on $(0,\infty)^3$, and suppose $(\alpha ,k) = 1$.  
Then
\begin{equation*}
\sum_{n_1,n_2,n_3\geq 1} e\left(\frac{n_1n_2n_3 \alpha}{k}\right)  J(n_1,n_2,n_3) = 
\frac{1}{k ^3} \sum_{p_1,p_2,p_3 \in \mathbb{Z}} A(p_1,p_2,p_3;\alpha;k ) B(p_1,p_2,p_3;k ),
\end{equation*}	
	where
	\begin{equation}
	\label{eq:Adef}
		A(p_1,p_2,p_3; \alpha; k) = \sum_{x_1,x_2,x_3 \shortmod{k} } e\left(\frac{x_1x_2x_3  \alpha - x_1p_1 - x_2p_2 - x_3 p_3}{k}\right),
	\end{equation}
	and
	\begin{equation*}
		B(p_1,p_2,p_3;k) 
	=\int\limits_{0}^\infty \int\limits_{0}^\infty \int\limits_{0}^\infty  J(t_1,t_2,t_3)  e\left(\frac{p_1t_1}{k} +\frac{p_2t_2}{k} + \frac{p_3t_3}{k}\right) \d t_1\d t_2 \d t_3.
	\end{equation*}
	An evaluation for $A$ is given in Lemma \ref{lem:Pnonzero} below (see also Lemma \ref{lem:Pzero}). 
\end{myprop}

In view of \eqref{eq:S'''prePoisson}, we need to evaluate
	$A(p_1,p_2,p_3;\delta_2 \overline{c_0};k_0')$,
and analyze
\begin{multline}
\label{eq:Bdef}
	B(p_1,p_2,p_3;k_0') 
	=\int_{(\mathbb{R}^+)^3} J(r_1e_1t_1, r_2e_2t_2, r_3t_3, a, m_1', c_0,g_0 k_0', c_2, k_1)\\
	e\left(\frac{p_1t_1}{k_0'} +\frac{p_2t_2}{k_0'} + \frac{p_3t_3}{k_0'}\right) \frac{\d t_1\d t_2 \d t_3}{\sqrt{e_1 e_2 \delta_1 t_1t_2t_3}},
\end{multline}
where $r_1 r_2 r_3 = \delta_1$ (see \eqref{eq:S''andS'}).

\subsection{The evaluation of $A$}

\begin{mylemma} \label{lem:Pnonzero}
	 Suppose $(\alpha, k) = 1$. We have
	\[
		A(p_1,p_2,p_3;\alpha;k) = k \sum_{f | (p_{2},p_{3},k)}f S\Big(p_1\overline{\alpha },\frac{ p_2p_3}{f^2}; \frac{k}{f}\Big).
	\]
\end{mylemma}


\begin{proof}
By first evaluating the sum over $x_3$, we derive
	\begin{equation*}
		A(p_1 , p_2 , p_3 ;\alpha;k) = k \sum_{\substack{x_1,x_2 \shortmod{k}\\ x_1x_2 \alpha \equiv p_3 \shortmod{k } }} e\left(\frac{x_1 p_1 + x_2p_2 }{k}\right). 
	\end{equation*}
	At this point we decompose the sum by letting $(x_1, k)= f$ with $f | k$.  Say $x_1 = f y$ with $y$ running over reduced residue classes modulo $k/f$.  
	Note that necessarily $f|(p_3, k)$, and that $x_2 \equiv \overline{\alpha y} \frac{p_3}{f} \pmod{k/f}$.   Therefore, we may write $x_2 = \overline{\alpha y} \frac{p_3}{f} + v \frac{k}{f}$ where $v$ runs modulo $f$.    Hence,
	\begin{equation*}
	A(p_1, p_2, p_3;\alpha;k) = k \sum_{f|(p_3,k)} \thinspace \sumstar_{y \shortmod{k/f}} e\Big(\frac{y p_1}{k/f}\Big) 
	e\Big(\frac{\overline{\alpha y} \frac{p_3}{f} p_2}{k}\Big) 
	\sum_{v \shortmod{f}} 
	e\Big(\frac{p_2 v}{f}\Big).
	\end{equation*}
The sum over $v$ detects $f|p_2$, and 	so the formula follows.
\end{proof}

\subsection{Asymptotics of $B$}
 
 Let us begin by unraveling the definition of $B$.  First we recall its definition from \eqref{eq:Bdef}, \eqref{eq:Jdef}, and \eqref{eq:Idef}.  
 Let us also pull out a factor $N^{-1/2}$ coming from $(e_1 e_2 \delta_1 t_1 t_2 t_3)^{-1/2}$.  Recall that $I_K(k)$ has a built-in inert function.  One may 
 change this inert function appropriately to obtain that $B$ takes a simplified form  
\begin{multline*}
	N^{1/2} B(p_1,p_2,p_3;k_0') 
	=\int_{(\mathbb{R}^+)^3} 
	e\Big(\frac{-e_1 e_2 \delta_1 t_1 t_2 t_3 am_1'}{c_0 k_0}\Big)
	 I^*(m_1, k, r_1e_1t_1, r_2e_2t_2, r_3t_3 a, c) 
	\\
	e\left(\frac{p_1t_1}{k_0'} +\frac{p_2t_2}{k_0'} + \frac{p_3t_3}{k_0'}\right) \d t_1\d t_2 \d t_3,
\end{multline*}
where $I^*$ has the same properties as $I$ given in Lemmas \ref{lemma:IKpreTransition} and \ref{lemma:IKpostTransition} (since all that changed is the definition of the inert function).
Note that the support of the inert function is such that $t_i \asymp N_i'$, say, where
\begin{equation*}
 N_1' = \frac{N_1}{e_1 r_1}, \qquad N_2' = \frac{N_2}{e_2 r_2}, \qquad N_3' = \frac{N_3}{r_3}.
\end{equation*}
Define $N' = N_1' N_2' N_3'$.
In the analytic aspects, it is usually most convenient to work  with the original variable names (we may perform the substitutions later, after analyzing the integral transform).   Let
\begin{equation*}
 \taco = e_1 e_2 r_1 r_2 r_3 = e_1 e_2 \delta_1,
\end{equation*}
and note that
\begin{equation*}
 N' \taco = N.
\end{equation*}

The terms with some $p_i = 0$ will be treated in Section \ref{section:ZeroTerms}, using a more elementary approach than the method used in the analysis of the nonzero terms with $p_1 p_2 p_3 \neq 0$.  For the nonzero terms, we apply dyadic partitions of unity to each $p_i$ variable, both for the positive and negative values separately.  
Let $B_P$ be the same as $B$ but multiplied by one function from this partition of unity with $\pm p_i \asymp P_i$, $i=1,2,3$; we suppress the signs in the notation of $B_P$.
As a convention, we may incorporate the case $p_i = 0$ by setting $P_i = 0$.

\begin{mylemma}[Post-Transition]
\label{lemma:BpostTransition}
Suppose \eqref{eq:PostTransition}  holds.  Then 
\begin{equation}
\label{eq:BasymptoticPostTransition}
M_2^{1/2} N^{1/2} B_P(p_1, p_2, p_3;k_0') = \frac{C }{(aMN)^{1/2}} M_2 N' w_T(\cdot), 
\end{equation}
where $w_T$ is $q^{\varepsilon}$-inert, $N' = N_1' N_2' N_3'$, and where $B_P(p_1, p_2, p_3) $
is very small 
unless
\begin{equation}
\label{eq:PiboundPostTransition}
 P_i \ll \frac{k_0'}{N_i'} q^{\varepsilon}   \qquad \text{ and } \qquad  K \asymp \frac{(aMN)^\hf}{M_2}. 
\end{equation}
\end{mylemma}
\begin{proof}[Proof of Lemma \ref{lemma:BpostTransition}]
The main observation is that the exponential factor appearing in \eqref{eq:IKpostTransition} cancels the exponential factor in the definition of $J$ in \eqref{eq:Jdef}.  
Therefore, $B$ is a Fourier transform of a $q^{\varepsilon}$-inert function supported on $t_i \asymp N_i'$, and hence by the discussion in Section \ref{section:InertFourier},
\eqref{eq:BasymptoticPostTransition} follows.
\end{proof}

 \begin{mylemma}[Pre-Transition, non-oscillatory]
\label{lemma:CasePreNonOsc}
Suppose \eqref{eq:PreTransition} holds.  If in addition 
\begin{equation}
\label{eq:CasePreNonOsc}
\frac{NaM_1}{CK} \ll q^{\varepsilon},
\end{equation}
then 
\begin{equation}
\label{eq:BasymptoticCasePreNonOsc}
M_2^{1/2} N^{1/2} B_P(p_1, p_2, p_3;\mattletter) = \Big(\frac{\sqrt{aMN}}{C}\Big)^{\kappa -1} M_2 N' w_T(\cdot), 
\end{equation}
where the inert function is $X$-inert with $X \ll q^{\varepsilon}$, and where $w_T$ is very small unless
\begin{equation*}
 \frac{N_i' P_i}{\mattletter} \ll q^{\varepsilon}, \quad i=1,2,3, \qquad \text{and} \qquad \frac{KM_2}{C} \ll q^{\varepsilon}.
\end{equation*}
\end{mylemma}

\begin{proof}
In this case, the exponential factor in the definition of $B_P$ is essentially not oscillatory, because of the condition \eqref{eq:CasePreNonOsc}.  For this, it is again helpful to remember that
\begin{equation*}
\frac{e_1 e_2 \delta_1 t_1 t_2 t_3 am_1'}{c_0 k_0}  \asymp \frac{N aM_1}{CK}.
\end{equation*}
Since \eqref{eq:CasePreNonOsc} holds, the exponential factor may be included into the definition of the inert weight function, which is $X$-inert with $X \ll q^{\varepsilon}$.  As in the case of Lemma \ref{lemma:BpostTransition}, we again obtain a Fourier transform of an $X$-inert function, and hence we obtain the claimed estimates.
\end{proof}

We record for later use that under the conditions of Lemmas \ref{lemma:BpostTransition} and \ref{lemma:CasePreNonOsc}, we have
\begin{equation}
\label{eq:Pbound}
P_1 P_2 P_3 \ll q^{\varepsilon} \frac{h}{N} \Big(\frac{k}{k_1 g_0}\Big)^3 \asymp  q^{\varepsilon} \frac{K^3}{N} \frac{h}{(k_1 g_0)^3}.
\end{equation}

\begin{mylemma}[Pre-Transition, oscillatory]
\label{lemma:CasePreOsc}
Suppose \eqref{eq:PreTransition} holds.  If in addition
\begin{equation}
\label{eq:CasePreOsc}
\frac{NaM_1}{CK} \gg q^{\varepsilon},
\end{equation}
then with $w_T(\cdot)$ denoting an $X$-inert function with $X \ll q^{\varepsilon}$, we have
\begin{multline}
\label{eq:BasymptoticCasePreosc}
M_2^{1/2} N^{1/2}
B_P(p_1, p_2, p_3;k_0')  = 
O(q^{-A} \prod_{i=1}^{3} (1 + |p_i|)^{-A})
\\
+
\Big(\frac{\sqrt{aMN}}{C}\Big)^{\kappa -1} M_2 N' \Big(\frac{CK}{a M_1 N}  \Big)^{3/2} e\Big(\frac{2 (p_1 p_2 p_3 ck)^{1/2}}{(a m_1 \taco \mattletter^3)^{1/2} } \Big) w_T(\cdot).
\end{multline}
 Moreover, 
$B_P(p_1, p_2, p_3;\mattletter)$ is very small unless each $p_i > 0$ and 
\begin{equation}
\label{eq:PiSizeOscillatory}
P_i \asymp \frac{N a M_1 \mattletter}{C K N_i'}, \quad i=1,2,3   \qquad\text{and} \qquad \frac{KM_2}{C} \ll q^\vep.
\end{equation}
\end{mylemma}
Observe the identity
\begin{equation*}
 \frac{2 (p_1 p_2 p_3 ck)^{1/2}}{(a m_1 \taco \mattletter^{3})^{1/2} } = \frac{2 (p_1 p_2 p_3 c_0)^{1/2}}{k_0' (\frac{ahm_1'}{g_0})^{1/2}} .
\end{equation*}

\begin{proof}
In this case, the phase arising from reciprocity is  oscillatory, and is not cancelled by a corresponding phase from the kernel function $I_K$.
By \eqref{eq:IKasymptoticPreTransition} and \eqref{eq:Jdef}, we have
 \begin{multline*}
M_2^{1/2} N^{1/2} B_P = \\
 \Big(\frac{\sqrt{aMN}}{C}\Big)^{\kappa -1} M_2 \int_{\mathbb{R}^3}    w_T(t_1, t_2, t_3, \cdot) e\Big(\frac{-t_1 t_2 t_3 a m_1 \taco}{ck} \Big) e\Big(\frac{t_1 p_1 + t_2 p_2 + t_3 p_3}{\mattletter} \Big) dt_1 dt_2 dt_3.
\end{multline*}
The behavior of this oscillatory integral is derived as an example in \cite{InertFunctionPaper}.
\end{proof}
 
 \subsection{Mellin transform of $B$}
 \label{section:MellinTransformOfB}
 For many of our later purposes, we prefer to work with the Mellin transform of $B_P$ instead of $B_P$ itself.  Of course, $B_P$ depends on a number of variables, and  what is meant here is the Mellin transform \emph{in terms of $k_0'$}.
Define
\begin{equation}
\label{eq:BtildeDef}
 \widetilde{B_P}(s) := \int_0^{\infty} B_P(p_1, p_2, p_3; x) x^s \frac{dx}{x},
\end{equation}
which is the Mellin transform of $B_P$ in $k_0'$.  Recalling $k = g_0 k_1 k_0'$, note that 
\begin{equation*}
 x \asymp \frac{K}{g_0 k_1}.
\end{equation*}

Let us combine the results from Lemmas \ref{lemma:BpostTransition} and \ref{lemma:CasePreNonOsc}.  In these two cases, we have
\begin{equation}
\label{eq:BasymptoticNonOsc}
M_2^{1/2} N^{1/2}  B_P(p_1, p_2, p_3;k_0') = \Big(\frac{\sqrt{aMN}}{C}\Big)^{\exponent} M_2 N' w_T(\cdot),
\end{equation}
where $\exponent = -1 $ in Lemma \ref{lemma:BpostTransition}, and $\exponent = \kappa - 1$ in Lemma \ref{lemma:CasePreNonOsc}.  In both cases, $p_i$ are supported on $|p_i| \asymp P_i \ll \frac{k_0'}{N_i'} q^{\varepsilon}$, but there are different constraints on the parameters.  In any event, in terms of $k_0'$, $B_P$ is $X$-inert, so we group these two cases together under the heading of ``Non-oscillatory''.
Lemma
 \ref{lemma:MellinTransformInertSection} then leads to the following.
\begin{mylemma}[Non-Oscillatory]
\label{lemma:BmellinNonOscillatory}
 Suppose the conditions of Lemma \ref{lemma:BpostTransition} or Lemma \ref{lemma:CasePreNonOsc} hold,  and put $\exponent = -1$ or $\exponent = \kappa - 1$ in the respective cases.   Then
 \begin{equation*}
 M_2^{1/2} N^{1/2} \widetilde{B_P}(s) = \Big(\frac{\sqrt{aMN}}{C}\Big)^{\exponent} M_2 N' \Big(\frac{K}{g_0 k_1}\Big)^s w_T(s,\cdot),
 \end{equation*}
where $w_T$ is $q^{\varepsilon}$-inert in all the variables except for $s$, and entire in terms of $s$.  Moreover, $w_T(\cdot;\sigma + it)$ is very small unless $|t| \ll_{\sigma} q^{\varepsilon}$.
\end{mylemma}
In the case that $B_P$ is oscillatory, it turns out to be easier to use the Bessel integral representation in the Bruggeman-Kuznetsov formula, and so we may avoid the Mellin transform analysis of $B_P$.  See the introductory paragraphs of Section \ref{section:WfunctionOscillatoryCase} for more explanation.

\section{Application of Bruggeman-Kuznetsov}
Write $\mathcal{T}_{P}$ for the terms from $\mathcal{S}'''$ with $B$ replaced by $B_P$ (in particular, $p_1 p_2 p_3 \neq 0$).  Therefore,
\begin{multline*}
\mathcal{T}_{P} = \sum_{\substack{(c_0,g_0 m_1') = 1\\ c_0\equiv 0 \shortmod{q k_1 k_1^*} }} 
 \frac{1}{c_0}
\sum_{\substack{( k_0',\delta_2 c_0) = 1 \\ k_0' \equiv 0 \shortmod{\frac{c_2}{(g_0, c_2)}} }}   \frac{1}{k_0'^3}
\sum_{\substack{p_1, p_2, p_3  \neq 0 }}
A(p_1, p_2, p_3;\delta_2 \overline{c_0};k_0')
B_P(p_1, p_2, p_3; k_0').
\end{multline*}
Applying Lemma \ref{lem:Pnonzero}, and moving the sum over $k_0'$ to the inside, we obtain
\begin{multline*}
\mathcal{T}_P = \sum_{\substack{(c_0,g_0 m_1') = 1\\ c_0\equiv 0 \shortmod{q k_1 k_1^*} }} 
 \frac{1}{c_0}
\sum_{\substack{p_1, p_2, p_3  \neq 0 }}
\sum_{f|(p_2, p_3)} 
\sum_{\substack{( k_0',\delta_2 c_0) = 1 \\ k_0' \equiv 0 \shortmod{\frac{c_2}{(g_0, c_2)}} \\  k_0' \equiv 0 \shortmod{f} }}   
\\
\frac{f}{k_0'^2}
S\Big(p_1 c_0 \overline{\delta_2}, \frac{p_2 p_3}{f^2};\frac{k_0'}{f}\Big)
B_P(p_1, p_2, p_3; k_0').
\end{multline*}
We absorb $k_0'^{-2} \asymp \frac{(g_0 k_1)^2}{K^2}$ into the inert function which changes the definition of $B_P$ (call the new function $B_{P,*}$), but not any of the analytic properties it satisfies (cf. Section \ref{section:InertConvention}), giving
\begin{multline*}
\mathcal{T}_P = \frac{(g_0 k_1)^2}{K^2} \sum_{\substack{(c_0,g_0 m_1') = 1\\ c_0\equiv 0 \shortmod{q k_1 k_1^*} }} 
 \frac{1}{c_0}
\sum_{\substack{p_1, p_2, p_3  \neq 0 }}
\sum_{f|(p_2, p_3)} f
\sum_{\substack{( k_0',\delta_2 c_0) = 1 \\ k_0' \equiv 0 \shortmod{\frac{c_2}{(g_0, c_2)}} \\  k_0' \equiv 0 \shortmod{f} }} 
\\
S\Big(p_1 c_0 \overline{\delta_2}, \frac{p_2 p_3}{f^2};\frac{k_0'}{f}\Big)
B_{P,*}(p_1, p_2, p_3; k_0').
\end{multline*}

 Let $k_0' = f k_0''$, so that the inner sum over $k_0'$ becomes
 \begin{equation*}
  \delta_{(f,\delta_2 c_0) = 1} \sum_{\substack{(k_0'',\delta_2 c_0) = 1 \\ k_0'' \equiv 0 \shortmod{\delta_3} }} 
S\Big(p_1 c_0 \overline{\delta_2}, \frac{p_2 p_3}{f^2};k_0''\Big)
B_{P,*}(p_1, p_2, p_3, f k_0''),
 \end{equation*}
where we have defined
\begin{equation}
\label{eq:delta3def}
 \delta_3 = \frac{ \frac{c_2}{(g_0, c_2)} }{(f, \frac{c_2}{(g_0, c_2)})}.
\end{equation}

Finally, to ease a later summation over $c_0$, we detect the condition $(k_0'', c_0) = 1$ with M\"{o}bius inversion, say over the variable $\delta_4$.  Then we reverse the orders of summation, and define
\begin{equation}
\label{eq:delta4anddelta5definitions}
 c_0 = \delta_4 c_0', \qquad \delta_5 = [\delta_3, \delta_4].
\end{equation}
We record that the summation conditions in the sum over $k_0''$ are empty unless
\begin{equation*}
 (\delta_2, \delta_5) = 1 \Leftrightarrow (\delta_2, \delta_3 \delta_4) = 1.
\end{equation*}
For later use, we also record that
\begin{equation}
\label{eq:delta4iscoprimetok1}
 (\delta_4, k_1) = 1,
\end{equation}
since $k_1 | \delta_2$, and $(\delta_2, \delta_4) = 1$. 
Using this, and moving the sum over $f$ to the outside, with the definitions
\begin{equation*}
 p_2 = f p_2', \qquad p_3 = f p_3',
\end{equation*}
we obtain
\begin{equation}
\label{eq:Tformula}
\mathcal{T}_P = \frac{(g_0 k_1)^2}{K^2}
\sum_{\delta_3 | \frac{c_2}{(g_0, c_2)}}
\sum_{(\delta_4, \delta_2 g_0 m_1' ) = 1} \frac{\mu(\delta_4)}{\delta_4}
\sum_{\substack{(f,\delta_2 \delta_4) = 1 \\ \eqref{eq:delta3def} \text{ is true}}} f
\sum_{\substack{(c_0',fg_0 m_1') = 1\\ \delta_4 c_0'\equiv 0 \shortmod{q k_1 k_1^*}  }} 
 \frac{1}{c_0'}
\sum_{\substack{p_1, p_2', p_3'  \neq 0 }}
\mathcal{K},
\end{equation}
where
\begin{equation}
 \mathcal{K} = \sum_{\substack{( k_0'',\delta_2) = 1 \\ k_0'' \equiv 0 \shortmod{\delta_5}  
}} 
S(p_1 \delta_4 c_0' \overline{\delta_2},p_2' p_3';k_0'')
B_{P,*}(p_1, f p_2'  , f p_3' ; f k_0'').
\end{equation}

Consulting Proposition \ref{prop:KloostermanSumsAtInfinityAndAtkinLehnerCusp}, we may now realize the Kloosterman sum in question as one belonging to the group $\Gamma = \Gamma_0(\delta_2\delta_5)$ with the pair of cusps $\infty, \frac{1}{\delta_5}$ (note that these are Atkin-Lehner cusps, since $(\delta_2, \delta_5) = 1$).  Hence
\begin{equation*}
\mathcal{K} = \sum_{k_0''\sqrt{\delta_2} \in \mathcal{C}_{\infty, \frac{1}{\delta_5}} } 
S_{\infty, \frac{1}{\delta_5}}^{} (p_1 \delta_4 c_0', p_2' p_3';k_0''\sqrt{\delta_2})
B_{P,*}(p_1, fp_2', fp_3'; f k_0'').
\end{equation*}
According to Theorem \ref{thm:KuznetsovTraceFormula}, write $\mathcal{K} = \mathcal{K}_d + \mathcal{K}_c + \mathcal{K}_h$, and accordingly write
$\mathcal{T}_{P} = \mathcal{T}_d + \mathcal{T}_c + \mathcal{T}_h$.  
We furthermore decompose $\mathcal{K} = \sum_{\epsilon_1, \epsilon_2, \epsilon_3 \in \{-1, 1\}} \mathcal{K}_{\epsilon_1, \epsilon_2, \epsilon_3}$, where the meaning is $\epsilon_i p_i \geq 1$ for $i=1,2,3$, and likewise decompose $\mathcal{K}_d$, etc.  To help ease the notation, let $\mathcal{K}^+$ denote the terms with $p_i \geq 1$ for all $i$, and $\mathcal{K}^{-}$ denote the terms with $p_1 \leq -1$ and $p_2, p_3 \geq 1$.   
For instance,
\begin{equation}
\label{eq:KformulaDiscreteSpectrum}
 \mathcal{K}_d^{\pm} = \sum_{t_j \text{ level }  \delta_2 \delta_5} \nu_{\infty,j}(p_1 \delta_4 c_0') \overline{\nu}_{\frac{1}{\delta_5},j}(p_2' p_3') W_{\pm}(p_1, fp_2', fp_3'; t_j),
\end{equation}
where
\begin{equation*}
 W_{\pm}(p_1, fp_2', fp_3'; t_j) = \int_{(2\theta + \varepsilon)}  h_{\pm}(s,t_j) 
 \Big(4 \pi \sqrt{\delta_4 c_0' |p_1| p_2' p_3'}\Big)^{-s} \widetilde{\widetilde{B}}_{P,*}(p_1, f p_2', fp_3'; s+1) ds,
\end{equation*}
(recall $h_{\pm}$ was defined by \eqref{eq:hplusminusDefinition}),
and where 
\begin{equation*}
 \widetilde{\widetilde{B}}_{P,*}(p_1, fp_2', fp_3'; s+1)
 := \int_0^{\infty} B_{P,*}\Big(p_1, fp_2', fp_3'; \frac{f y}{  \sqrt{\delta_2}}\Big) y^{s+1} \frac{dy}{y}.
\end{equation*}
Here the ``double tilde'' notation for $B$ is meant to indicate the Mellin transform of $B$ with respect to $\gamma = k_0'' \sqrt{\delta_2}$ (where $\gamma \in \mathcal{C}_{\a \b}$ as in \eqref{eq:GeneralSumOfKloostermanSum}), because we have already reserved the meaning of $\widetilde{B}$ for the Mellin transform in the $k_0'$-variable (as in Section \ref{section:MellinTransformOfB}).  
 The relationship between these two transforms is given by
 \begin{equation*}
  \widetilde{\widetilde{B}}_{P,*}(s+1) = \Big(\frac{  \sqrt{\delta_2}}{f}\Big)^{s+1} \widetilde{B}_{P,*}(s+1).
 \end{equation*}
Simplifying, we obtain
\begin{equation}
\label{eq:WformulaWithBtildeUnspecified}
 W_{\pm}(p_1, p_2, p_3;t_j) = \frac{ \sqrt{\delta_2}}{f} \int_{(2\theta+\varepsilon)}   h_{\pm}(s,t_j) 
 \Big(  \frac{ \sqrt{\delta_2} }{\sqrt{\delta_4 c_0' |p_1| p_2 p_3}}\Big)^{s} \widetilde{B}_{P,*}(s+1) ds.
\end{equation}
The holomorphic case is similar, but with a different integral kernel than $h_{\pm}(s,t_j)$.

We may also prefer to use the Bessel integral representation for $W$, which we do in case $B_{P,*}$ is oscillatory.  For instance, we have
\begin{equation*}
 W_h(p_1, fp_2', fp_3'; \ell) =  \int_0^{\infty} J_{\ell-1}\Big(\frac{4\pi \sqrt{p_1 \delta_4 c_0' p_2' p_3'} }{y} \Big) 
 B_{P,*}\Big(p_1, fp_2', fp_3'; \frac{fy}{ \sqrt{\delta_2}}\Big) \d y.
\end{equation*}
Changing variables, we obtain
\begin{equation}
\label{eq:WhFunctionBesselFormula}
 W_h(p_1, fp_2', fp_3'; \ell) = \frac{ \sqrt{\delta_2}}{f}  \int_0^{\infty} J_{\ell-1}\Big(\frac{4\pi  \sqrt{f^2 p_1 p_2' p_3' \delta_4 c_0'   } }{ \sqrt{\delta_2} y} \Big) 
 B_{P,*}(p_1, fp_2', fp_3'; y) \d y.
\end{equation}
Note that, in terms of older variables names, we have
\begin{equation*}
 \frac{f^2 p_1 p_2' p_3' \delta_4 c_0'}{\delta_2} = \frac{p_1 p_2 p_3 c_0}{(ahm_1')/g_0}.
\end{equation*}
Similarly, in the $+$ Maass case, we have
\begin{equation}
\label{eq:WplusFunctionBesselFormula}
 W_{+}(p_1, fp_2', fp_3';t_j) = \frac{ \sqrt{\delta_2}}{f}  \int_0^{\infty} B_{2it_j}^{+}\Big(\frac{4\pi  \sqrt{f^2 p_1 p_2' p_3' \delta_4 c_0'   } }{ \sqrt{\delta_2} y} \Big) 
 B_{P,*}(p_1, fp_2', fp_3'; y) \d y.
\end{equation}

 \section{Asymptotics of $W$}
Here we analyze the various $W$-functions appearing in the Bruggeman-Kuznetsov formula.

\subsection{Non-oscillatory cases}
First suppose the 
conditions of Lemma \ref{lemma:BpostTransition} or Lemma \ref{lemma:CasePreNonOsc} hold, so that Lemma \ref{lemma:BmellinNonOscillatory} gives the behavior of $\widetilde{B}$.  Continuing from \eqref{eq:WformulaWithBtildeUnspecified}, we have
\begin{equation}
\label{eq:WformulaWithBtildeSpecifiedNonOscillatoryCase}
 W_{\pm}(p_1, p_2, p_3; t_j) = \frac{\Big(\frac{\sqrt{aMN}}{C}\Big)^{\exponent} M_2 N' K}{M_2^{1/2} N^{1/2}} \frac{ \sqrt{\delta_2}}{f g_0 k_1} \int_{(2\theta+\varepsilon)}   h_{\pm}(s,t_j) 
 \Big(  \frac{ \sqrt{\delta_2} K }{g_0 k_1 \sqrt{\delta_4 c_0' |p_1| p_2 p_3}}\Big)^{s} w_T(s, \cdot) ds.
\end{equation}
Here $w_T$ is $q^{\varepsilon}$-inert in all variables except $s$.  It is entire in $s$, with rapid decay for $|\text{Im}(s)| \gg q^{\varepsilon}$.

As shorthand, let
\begin{equation}
\label{eq:Ydef}
 Y = \frac{g_0 k_1 \sqrt{C P_1 P_2 P_3}}{ \sqrt{\delta_2 c_2} K} \asymp \Big(\frac{ \sqrt{\delta_2} K }{g_0 k_1 \sqrt{\delta_4 c_0' |p_1| p_2 p_3}} \Big)^{-1}. 
\end{equation}

Our goal now is to show
\begin{mylemma}[Non-Oscillatory]
\label{lemma:spectraltrunctationNonOscillatory}
Suppose the conditions of Lemma \ref{lemma:BpostTransition} or Lemma \ref{lemma:CasePreNonOsc} hold.  If $|t_j| \gg (1+Y) q^{\varepsilon}$, then $W_{\pm}$ is very small.  Similarly, if $k \gg (1+Y) q^{\varepsilon}$, then $W_h$ is very small.
\end{mylemma}
\begin{proof}
 If $s = \sigma + it$, and $|t| \gg (|t_j| q)^{\varepsilon}$, then by the rapid decay of $w_T$, we conclude that this part of the integral is bounded in a satisfactory manner.  In the complementary region, we then have from Stirling that
 \begin{equation*}
  h_{\pm}(\sigma + it, t_j)  \ll_\sigma  q^{\varepsilon} (q^{\varepsilon} + |t_j|)^{\sigma-1}.
 \end{equation*}
Side remark: The exponential factor implicitly appearing in Stirling's bound on $h_{\pm}(s,t_j)$ is $\ll 1$, and one cannot do better than this in general, because in one of the two cases of $\pm$ sign, the exponential factor is exactly $1$.

Now if $|t_j| \gg (1+Y) q^{\varepsilon}$, we shift the contour far to the left and bound it trivially.  In doing so, one encounters poles at $\frac{s}{2} \pm i t_j = 0, -1, -2, \dots$.  However, these all have large imaginary part and $w_T$ is very small here, so these residues are bounded in a satisfactory manner.  The integral on the new line is very small since $|t_j|/Y \gg q^{\varepsilon}$.

Next consider $W_h(k)$. 
The analysis is similar, except one replaces $h_{\pm}(s,t_j)$ by
\begin{equation*}
h(s,k) := 2^{s-1}  \frac{\Gamma(\frac{s+ k-1}{2})}{\Gamma(\frac{-s+ k+1}{2})}.
\end{equation*}
Stirling's formula gives, for $\sigma \ll \sqrt{|k+it|}$, that
\begin{equation*}
 |h(\sigma + it, k)| \ll (\max(k, |t|))^{\sigma-1}.
\end{equation*}
As before, if $k \gg (1+ Y)q^{\varepsilon}$, then we may move the contour far to the left (some large constant not growing with $q$).  Then we get that $W_h$ is small, by the exact same type of reasoning as in the Maass case.
\end{proof}

Now we reap the reward of the language of inert functions.  Since $w_T$ is inert in all variables, we may apply the Mellin inversion formula together with Lemma \ref{lemma:MellinTransformInertSection}, giving 
\begin{multline}
\label{eq:WMellinformulaNonOscillatoryCase}
 W_{\pm}(p_1, p_2, p_3; t_j) = \Big(\frac{\sqrt{aMN}}{C}\Big)^{\kappa - 1}
(M_2 N)^{1/2} K  
  \frac{ \sqrt{\delta_2}}{hf g_0 k_1} 
 \int_{(2\theta+\varepsilon)} h_{\pm}(s,t_j) 
 \\
 \times
   \int
 \Big(  \frac{ \sqrt{\delta_2} K }{g_0 k_1 \sqrt{\delta_4 c_0' |p_1| p_2 p_3}}\Big)^{s} \widetilde{w_T}(s,{\bf u}, \cdot) 
 \Big(\frac{P_1}{|p_1|}\Big)^{u_1}  
 \Big(\frac{P_2}{p_2}\Big)^{u_2}  
  \Big(\frac{P_3}{p_3}\Big)^{u_3}
   \Big(\frac{C}{\delta_4 c_0' c_2}\Big)^{u_4} 
  d{\bf u} ds,
\end{multline}
plus a small error term. 
Here $\widetilde{w_T}$ is very small except if the imaginary parts of all the variables are $\ll q^{\varepsilon}$.

\subsection{Oscillatory Case}
\label{section:WfunctionOscillatoryCase}
Now we consider $W_{\pm}$ and $W_h$ when $B$ is given by Lemma \ref{lemma:CasePreOsc}.
The first significant point is that $W_{-}$ is small, because this corresponds to the case where $p_1 p_2 p_3 < 0$, which means some $p_i < 0$, in which case $B$ is small.  Indeed, $B$ is small unless $p_i > 0$ for all $i$, and so the only relevant functions are $W_{+}$ and $W_{h}$.

It is inconvenient to use \eqref{eq:WformulaWithBtildeUnspecified} in the oscillatory case.  The problem is that the oscillatory nature of $B$ means that we may no longer restrict $|\text{Im}(s)|$ to be $O(q^{\varepsilon})$, which in turn has an effect on the behavior of $h_+(s,r)$ and $h(s,k)$.  Namely, it is no longer true that $h_+(s,r)$ and $h(s,k)$ satisfy analogous asymptotic formulas (due to the use of Stirling with $ir$ large vs. $k$ large), and so it appears difficult to unify these two cases. 
In addition, one is forced to confront some tricky oscillatory integrals.
 To sidestep these problems entirely, we shall use the Bessel integral formula for $W$ instead.  The oscillatory behavior of $B$ is actually beneficial and causes $W$ to be essentially inert (in both the Maass and holomorphic cases).

Let us begin with $W_h$.  We have
\begin{equation*}
 W_h(p_1, p_2, p_3; \ell) = 
 \Big(\frac{\sqrt{aMN}}{C}\Big)^{\exponent} \sqrt{M_2 N} \Big(\frac{CK}{a M_1 N}  \Big)^{3/2}
 \frac{ \sqrt{\delta_2}}{f h} Z,
\end{equation*}
plus a small error term,
where $Z$ is shorthand for
\begin{equation}
 Z = \int_0^{\infty} 
 J_{\ell-1}\Big(\frac{4\pi  \sqrt{p_1  p_2 p_3\delta_4 c_0'} }
 				{\sqrt{\delta_2} y} \Big) 
 e\Big(\frac{2  \sqrt{p_1  p_2 p_3 \delta_4 c_0'} }{ \sqrt{\delta_2} y} \Big)  w_T(y, \cdot) 
 dy.
\end{equation}
Here we recall that $w_T$ has support on $y \asymp \frac{K}{g_0 k_1}$.
The fact that the phases match is pleasant.

Recall the integral representation
\begin{equation}
\label{eq:JBesselIntegralRepresentation}
 J_{\ell-1}(x) = \sum_{\pm } c_{\ell,\pm}
 \int_{0 }^{\pi/2} 
 \cos((\ell-1) \theta) 
 e^{ \pm ix \cos(\theta )} d \theta, \qquad c_{\ell,\pm} = \frac{e^{\mp i (\ell -1)\frac{\pi}{2}}}{ \pi}.
\end{equation}
This gives
\begin{equation*}
 Z = \sum_{\pm} \int_0^{\pi/2} c_{\ell, \pm} \cos((\ell-1) \theta) 
 \int_0^{\infty} e\Big(\frac{z(1\pm \cos \theta)}{y}\Big) w_T(y, \cdot) dy d \theta,
\end{equation*}
with
\begin{equation*}
 z = \frac{2  \sqrt{p_1  p_2 p_3 \delta_4 c_0'} }{ \sqrt{\delta_2} }.
\end{equation*}
Changing variables $y = \frac{K}{g_0 k_1 x}$ now gives $x \asymp 1$, and the inner integral is a Fourier transform of an inert function.  Hence
\begin{equation*}
 Z = \frac{K}{g_0 k_1} 
 \sum_{\pm} \int_0^{\pi/2} c_{\ell,\pm} \cos((\ell-1) \theta)
 \widehat{w_T}\Big(\frac{z g_0 k_1}{K} (1 \pm \cos \theta)\Big) d \theta,
\end{equation*}
where we have re-defined $w_T$ (see Section \ref{section:InertConvention}).
Using $\delta_4 c_0' = c_0 = \frac{c}{c_2}$, $\delta_2 = \frac{h am_1'}{g_0}$, $m_1' = \frac{m_1}{k_1 c_2}$, \eqref{eq:PiSizeOscillatory}, and $k_0' = \frac{k}{k_1 g_0}$, we 
check the size of
\begin{equation*}
 \frac{zg_0 k_1}{K} \asymp \frac{\sqrt{P_1 P_2 P_3 C k_1^3 g_0^3}}{K\sqrt{ha M_1}} \asymp \frac{Na M_1}{CK },
\end{equation*}
which is $\gg q^{\varepsilon}$ because we are operating under the conditions of Lemma \ref{lemma:CasePreOsc}.  

Now we observe that the integrand is very small unless
\begin{equation*}
\frac{Na M_1}{CK}  |1\pm \cos \theta| \ll q^{\varepsilon}.
\end{equation*}
Hence, the sign must be $-$, and we must have
\begin{equation*}
 \theta \ll \Big(\frac{CK}{N aM_1}\Big)^{ 1/2} q^{\varepsilon},
\end{equation*}
(which is $O(q^{-\delta})$, for some $\delta > 0$).  
This means that by using a Taylor expansion, we may develop the $\widehat{w_T}$ part into an asymptotic expansion with leading term given by the substitution $1-\cos \theta \rightarrow \theta^2/2$.  Therefore,
\begin{equation*}
 Z = \frac{K}{g_0 k_1} \int_{-\infty}^{\infty} \cos((\ell-1)\theta) 
 (\widehat{w_T}\Big(\frac{z g_0 k_1}{K} \theta^2\Big) + \dots ) d\theta,
\end{equation*}
where we were able to extend the integral to $+\infty$ since $\widehat{w_T}$ is small otherwise, and also extend to $-\infty$ by symmetry (we have also re-defined the inert function to absorb constants).

As another shorthand, let
\begin{equation*}
 Q = \frac{z g_0 k_1}{K} \asymp \frac{NaM_1}{CK}.
\end{equation*}
Then $Z$ takes the form
\begin{equation*}
 Z = \frac{K}{g_0 k_1 \sqrt{Q}} \intR \exp\Big(i \frac{(\ell-1)}{\sqrt{Q}} \theta\Big) \widehat{w_T}(\theta^2) d \theta + \dots.
\end{equation*}
If we let $g(\theta) = \widehat{w_T}(\theta^2)$, then $g^{(j)}(\theta) \ll_{j,A} X^j (1 + \theta)^{-A}$, for arbitrary $j,A$, where $X \ll q^{\varepsilon}$.  Therefore, this is another Fourier transform of a function with controlled derivatives, and so
by the discussion in Section \ref{section:InertFourier},
 it takes the form
\begin{equation*}
 \frac{K}{g_0 k_1 \sqrt{Q}}  G\Big(\frac{\ell-1}{\sqrt{Q}}, \cdot \Big), 
\end{equation*}
plus a very small error term, 
where $G$ would be $q^{\varepsilon}$-inert (in $\ell$) if it had dyadic support.  It is $q^{\varepsilon}$-inert in all the other variables, however.

Re-grouping, we have that
\begin{equation*}
 W_h  =  \Big(\frac{\sqrt{aMN}}{C}\Big)^{\kappa -1 } \sqrt{M_2 N} \Big(\frac{CK}{a M_1 N}  \Big)^{2} K
 \frac{ \sqrt{\delta_2}}{f g_0 k_1 h} G\Big(\frac{\ell-1}{\sqrt{Q}}, \cdot \Big),
\end{equation*}
plus a very small error term, 
where $G$ is very small unless
\begin{equation*}
 \ell \ll \Big(\frac{M_1 a N}{CK}\Big)^{1/2} q^{\varepsilon}.
\end{equation*}
Then we may take the Mellin transform in $p_1, p_2, p_3, c_0$, giving
\begin{multline}
\label{eq:WhOscillatoryCase}
 W_h(p_1, p_2, p_3; c_0'; \ell) = 
\Big(\frac{\sqrt{aMN}}{C}\Big)^{\kappa - 1}  \sqrt{M_2 N} \Big(\frac{CK}{a M_1 N}  \Big)^{2} K
 \frac{ \sqrt{\delta_2}}{f g_0 k_1 h}
 \\
 \int_{(\sigma)} \widetilde{w_T}({\bf u}, \ell, \cdot)
 \Big(\frac{P_1}{|p_1|}\Big)^{u_1}  
 \Big(\frac{P_2}{p_2}\Big)^{u_2}  
  \Big(\frac{P_3}{p_3}\Big)^{u_3}
   \Big(\frac{C}{\delta_4 c_0' c_2}\Big)^{u_4} d {\bf u},
\end{multline}
plus a small error term.

Now we turn to $W_{+}$.  Since the details are similar to the previous case, the exposition is brief.
We follow through the steps used for $W_h$, where the alteration in the first step is replacing $J_{\ell-1}$ by $B_{2ir}^{+}(x)$.  In place of \eqref{eq:JBesselIntegralRepresentation}, we have instead
\begin{equation*}
 \frac{J_{2ir}(x) - J_{-2ir}(x)}{\sinh( \pi r)} = \frac{2}{\pi i}\intR \cos(x \cosh v) e\Big(\frac{rv}{\pi}\Big) dv.
\end{equation*}
We shall use this for real values of $r$.  Although this integral  does not converge absolutely, we have
\begin{equation}
\label{eq:JBesselIntegralRepTailBound}
\Big| \int_{|v| \geq V} \cos(x \cosh v) e\Big(\frac{rv}{\pi}\Big) dv \Big| \ll \frac{1+ |r|}{x \sinh{V}},
\end{equation}
from integration by parts. 

Forming the analog of $Z$ from $W_h$, and keeping the same definition of $z$, we have (absorbing the absolute constant into the inert function)
\begin{equation*}
Z =  \intR e\Big(\frac{rv}{\pi}\Big) \intR e\Big(\frac{z}{y}\Big) \cos\Big(2\pi \frac{z}{y} \cosh(v)\Big) w_T(y,\cdot) dy dv,
\end{equation*}
using \eqref{eq:JBesselIntegralRepTailBound} to reverse the orders of integration.
Next write $\cos(u) = \frac12 e^{iu} + \frac12 e^{-iu}$; the part with $e^{iu}$ is very small as in the $W_h$ case.  From this point on, the analysis is nearly identical to that of $W_h$, and the conclusion is that $W_{+}$ is very small unless
\begin{equation*}
|t_j| \ll \Big(\frac{M_1 a N}{CK}\Big)^{1/2} q^{\varepsilon},
\end{equation*}
and $W_{+}$ satisfies a formula identical to that in \eqref{eq:WhOscillatoryCase}.

In the exceptional eigenvalue case where $ir \in \mathbb{R}$, then the final shape of the formula for $W_{+}$ is the same as \eqref{eq:WhOscillatoryCase}, but the above arguments would need modification since $e(rv/\pi)$ is no longer bounded.  There is a more direct route, however.  We have the asymptotic expansion (see \cite[(8.451.1)]{GR})
\begin{equation*}
 \frac{J_{2ir}(x) - J_{-2ir}(x)}{\sinh( \pi r)} \sim 
 \sum_{\pm} e^{\pm ix} \sum_{k} \frac{P_{\pm}(r,k)}{x^{\frac12 + k}} 
 ,
\end{equation*}
where $P_{\pm}(r,k)$ is a polynomial in $r$ and $k$.  
This is certainly valid for $r = O(1)$, and $x \gg 1$ (in the present context, $x \gg q^{\varepsilon}$).  With this, it is easy to estimate $Z$ directly, showing that it is of the form $\frac{K}{g_0 k_1 \sqrt{Q}}$ times an $X$-inert function, plus a small error term.  Therefore, applying Mellin inversion in the appropriate variables, we obtain an expression of the same form as \eqref{eq:WhOscillatoryCase}.

\section{Regrouping after Bruggeman-Kuznetsov}
\subsection{Non-oscillatory, Maass cases}
\label{section:TdNonOscillatory}
Here we consider the contribution to $\mathcal{T}_d$ from the parameters where $B$ is non-oscillatory.
By \eqref{eq:Tformula}, \eqref{eq:KformulaDiscreteSpectrum}, and 
\eqref{eq:WMellinformulaNonOscillatoryCase}, we obtain
\begin{multline}
\label{eq:Tdformula}
 \mathcal{T}_d^{\pm} = 
 \frac{g_0 k_1}{K}
 \sum_{\delta_3 | \frac{c_2}{(g_0, c_2)}}
\sum_{(\delta_4, \delta_2 g_0 m_1' ) = 1} \frac{\mu(\delta_4)}{\delta_4}
\sum_{\substack{(f,\delta_2 \delta_4 ) = 1 \\ \eqref{eq:delta3def} \text{ is true}}}  
\thinspace
\sum_{\substack{(c_0',fg_0 m_1') = 1\\ \delta_4 c_0'\equiv 0 \shortmod{q k_1 k_1^*} }} 
 \frac{1}{c_0'}
\sum_{\substack{\pm p_1, p_2', p_3'  \geq 1 }} 
\\
\sum_{t_j \text{ level }  \delta_2 \delta_5} \nu_{\infty,j}(p_1 \delta_4 c_0') \overline{\nu}_{\frac{1}{\delta_5},j}(p_2' p_3') 
\Big(\frac{\sqrt{aMN}}{C}\Big)^{\exponent}
(M_2 N)^{1/2}    \frac{\sqrt{\delta_2} }{h}
 \int_{(2\theta+\varepsilon)} h_{\pm}(s,t_j) 
 \\
   \int_{(\sigma)}
 \Big(  \frac{ \sqrt{\delta_2} K }{f g_0 k_1 \sqrt{\delta_4 c_0' |p_1| p_2' p_3'}}\Big)^{s} \widetilde{w_T}(s,{\bf u}, \cdot)
 \Big(\frac{P_1}{|p_1|}\Big)^{u_1}  
 \Big(\frac{P_2}{fp_2'}\Big)^{u_2}  
  \Big(\frac{P_3}{fp_3'}\Big)^{u_3}
   \Big(\frac{C}{\delta_4 c_0' c_2}\Big)^{u_4} 
  d{\bf u} ds,
\end{multline}
plus a very small error term.
In the above expression, we could take $\text{Re}(s) > 2 \theta$ without crossing any poles coming from exceptional Laplace eigenvalues (recall \eqref{eq:hplusminusDefinition} for the definition of $h_{\pm}$).
By Lemma \ref{lemma:spectraltrunctationNonOscillatory}, we may truncate at $|t_j| \ll (1+Y) q^{\varepsilon}$ with a small error term.  Now we move the sums over $p_1, p_2', p_3'$, and $c_0$ to the inside, change variables $u_i \rightarrow u_i - \frac{s}{2}$, and bound everything at that point with absolute values.  In this way, we obtain
\begin{multline}
\label{eq:TdExpressionSpectralMoment}
 \mathcal{T}_d^{\pm} \ll q^{\varepsilon}
 \frac{g_0 k_1 c_2}{KC} \Big(\frac{\sqrt{aMN}}{C}\Big)^{\exponent} 
 (M_2 N)^{1/2}      
 \frac{\sqrt{\delta_2} }{h} 
\sum_{\delta_3 | \frac{c_2}{(g_0, c_2)}}
 \sum_{(\delta_4, \delta_2 g_0 m_1' ) = 1}   
\sum_{\substack{(f,\delta_2 \delta_4 ) = 1 \\ \eqref{eq:delta3def} \text{ is true}}}  
\sum_{\substack{t_j \text{ level }  \delta_2 \delta_5 \\ |t_j| \ll (1+Y)q^{\varepsilon}}}  
\\
\frac{1}{1+|t_j|}
 \int_{(2\theta+\varepsilon)} 
 \int_{(\sigma)}
 \Big(  \frac{  t_j}{Y}\Big)^{2\theta+\varepsilon} |\widetilde{w_T}(s, {\bf u} - \tfrac{s}{2}, \cdot)|
 \Big| P_1^{u_1}  
 \Big(\frac{P_2}{f }\Big)^{u_2}  
  \Big(\frac{P_3}{f }\Big)^{u_3}
   \Big(\frac{C}{\delta_4  c_2}\Big)^{u_4} \Big|
   |Z_j({\bf u})|
  d{\bf u} ds,
\end{multline}
plus a very small error term,
where
\begin{equation*}
 Z_j({\bf u}) = \sum_{\substack{(c_0',fg_0 m_1') = 1\\  \delta_4 c_0' \equiv 0 \shortmod{qk_1 k_1^*}}} 
\sum_{\substack{p_1, p_2', p_3'  \geq 1 }}  \frac{ \nu_{\infty,j}(p_1 \delta_4 c_0') \overline{\nu}_{\frac{1}{\delta_5},j}(p_2' p_3') }{p_1^{u_1} p_2'^{u_2} p_3'^{u_3} c_0'^{u_4}}.
\end{equation*}
Our plan is to relate $Z_j({\bf u})$ to $L$-functions, and use a large sieve inequality to bound it on average over $t_j$.  
\subsection{Non-oscillatory, Holomorphic cases}
These cases are nearly identical to those in Section \ref{section:TdNonOscillatory}, but the bounds will turn out to be even better due to the applicability of Deligne's bound.  The key point is that for $k \ll (1+Y)q^{\varepsilon}$, we may claim the bound
\begin{equation*}
 |h(s,k)| \ll k^{\sigma-1},
\end{equation*}
which is entirely analogous to $|h(s,t_j)| \ll t_j^{\sigma-1}$.  We omit the details for brevity.

\subsection{Oscillatory, Maass cases}
As in Section \ref{section:TdNonOscillatory}, we use \eqref{eq:Tformula} and \eqref{eq:KformulaDiscreteSpectrum}, but instead of \eqref{eq:WMellinformulaNonOscillatoryCase} we use a variant on \eqref{eq:WhOscillatoryCase}.   Also recall that only the $+$ sign enters the picture in the oscillatory case.
Thus we obtain
\begin{multline}
\label{eq:TdSpectralBoundOscillatoryCase}
 \mathcal{T}_d^{+} \ll 
 \frac{g_0 k_1 c_2 }{KC} \Big(\frac{\sqrt{aMN}}{C}\Big)^{\exponent} (M_2 N)^{1/2}   \frac{ \sqrt{\delta_2}  }{h}
\sum_{\delta_3 | \frac{c_2}{(g_0, c_2)}}
 \sum_{(\delta_4, \delta_2 g_0 m_1' ) = 1}   
\sum_{\substack{(f,\delta_2 \delta_4 c_0') = 1 \\ \eqref{eq:delta3def} \text{ is true}}}  
\\
\sum_{\substack{t_j \text{ level }  \delta_2 \delta_5 \\ |t_j| \ll Y'q^{\varepsilon}}} \Big(\frac{CK}{a M_1 N}\Big)^2 
 \int_{(\sigma)}
 |\widetilde{w_T}(t_j, {\bf u}, \cdot)|
 \Big| P_1^{u_1}  
 \Big(\frac{P_2}{f }\Big)^{u_2}  
  \Big(\frac{P_3}{f }\Big)^{u_3}
   \Big(\frac{C}{\delta_4  c_2}\Big)^{u_4} \Big|
   |Z_j({\bf u})|
  d{\bf u},
\end{multline}
where
\begin{equation*}
 Y' = \Big(\frac{M_1 aN}{CK}\Big)^{1/2} q^{\varepsilon}.
\end{equation*}

\subsection{Oscillatory, Holomorphic cases}
These are similar to (but easier than) the Oscillatory, Maass cases, and so we omit them.

\subsection{Continuous spectrum}\label{subsec:Cts}
First consider the non-oscillatory cases.  Then analogously to \eqref{eq:Tdformula}, we have
\begin{multline}
\label{eq:TcFormulaSpectralPreAbsoluteValues}
\mathcal{T}_{c}^{\pm} = 
 \frac{ g_0 k_1 }{K }
 \sum_{\delta_3 | \frac{c_2}{(g_0, c_2)}}
\sum_{(\delta_4, \delta_2 g_0 m_1' ) = 1} \frac{\mu(\delta_4)}{\delta_4}
\sum_{\substack{(f,\delta_2 \delta_4 ) = 1 \\ \eqref{eq:delta3def} \text{ is true}}}  
\sum_{\substack{(c_0',fg_0 m_1') = 1\\ \delta_4 c_0'\equiv 0 \shortmod{q k_1 k_1^*} }} 
 \frac{1}{c_0'}
\sum_{\substack{\pm p_1, p_2', p_3'  \geq 1}}
\\
\sum_{\mathfrak{c}} \int_{t }  
\nu_{\infty, \mathfrak{c}}(p_1 \delta_4 c_0',\tfrac12 + it) 
\overline{\nu}_{\frac{1}{\delta_5}, \mathfrak{c}}(p_2' p_3',\tfrac12 +it)
\Big(\frac{\sqrt{aMN}}{C}\Big)^{\exponent} 
(M_2 N)^{1/2} \frac{ \sqrt{\delta_2}}{h}  
 \int_{(\varepsilon)} h_{\pm}(s,t ) 
 \\
   \int_{(1+\varepsilon)}
 \Big(  \frac{ \sqrt{\delta_2} K }{f g_0 k_1 \sqrt{\delta_4 c_0' |p_1| p_2' p_3'}}\Big)^{s} \widetilde{w_T}(s, {\bf u}, \cdot)
 \Big(\frac{P_1}{|p_1|}\Big)^{u_1}  
 \Big(\frac{P_2}{fp_2'}\Big)^{u_2}  
  \Big(\frac{P_3}{fp_3'}\Big)^{u_3}
   \Big(\frac{C}{\delta_4 c_0' c_2}\Big)^{u_4}
  d{\bf u} ds dt.
\end{multline}
Now we move the sums to the inside, getting
\begin{multline}
\label{eq:TcExpressionSpectralMoment}
 \mathcal{T}_c^{\pm} \ll q^{\varepsilon}
 \frac{g_0 k_1 c_2}{KC} \Big(\frac{\sqrt{aMN}}{C}\Big)^{\exponent} 
 (M_2 N)^{1/2} \frac{\sqrt{\delta_2}}{h}
\sum_{\delta_3 | \frac{c_2}{(g_0, c_2)}}
 \sum_{(\delta_4, \delta_2 g_0 m_1' ) = 1}   
\sum_{\substack{(f,\delta_2 \delta_4 ) = 1 \\ \eqref{eq:delta3def} \text{ is true}}}  
\int_{\substack{ |t| \ll (1+Y)q^{\varepsilon}}}  
\\
\frac{1}{1+|t|}
 \int_{(\varepsilon)} 
 \int_{(1+\varepsilon)}
 |\widetilde{w_T}(s, {\bf u} -\tfrac{s}{2}, \cdot)|
 \Big| P_1^{u_1}  
 \Big(\frac{P_2}{f }\Big)^{u_2}  
  \Big(\frac{P_3}{f }\Big)^{u_3}
   \Big(\frac{C}{\delta_4  c_2}\Big)^{u_4} \Big|
 \sum_{\mathfrak{c}}  |Z_{\mathfrak{c},t}({\bf u})|
 d{\bf u} ds,
\end{multline}
where
\begin{equation}\label{eq:ZctDef}
Z_{\mathfrak{c},t}({\bf u})
= 
\sum_{\substack{(c_0',fg_0 m_1') = 1\\  \delta_4 c_0' \equiv 0 \shortmod{qk_1 k_1^*}}} 
\sum_{\substack{p_1, p_2', p_3'  \geq 1 }}  
\frac{ 
\nu_{\infty,\mathfrak{c}}(p_1 \delta_4 c_0',\tfrac12 +it) 
\overline{\nu}_{\frac{1}{\delta_5},\mathfrak{c}}(p_2' p_3', \tfrac12 + it) }
{p_1^{u_1} p_2'^{u_2} p_3'^{u_3} c_0'^{u_4}}.
\end{equation}

The {\bf oscillatory case} is similar, leading to 
\begin{multline}
\label{eq:TcSpectralBoundOscillatoryCase}
 \mathcal{T}_c^{+} \ll 
 \frac{g_0 k_1 c_2}{KC} \Big(\frac{\sqrt{aMN}}{C}\Big)^{\kappa - 1} (M_2 N)^{1/2}    \frac{\sqrt{\delta_2} }{h}
\sum_{\delta_3 | \frac{c_2}{(g_0, c_2)}}
 \sum_{(\delta_4, \delta_2 g_0 m_1' ) = 1}   
\sum_{\substack{(f,\delta_2 \delta_4 c_0') = 1 \\ \eqref{eq:delta3def} \text{ is true}}}  
\int_{\substack{ |t| \ll Y'q^{\varepsilon}}} 
\\
\Big(\frac{CK}{a M_1 N}\Big)^2 
 \int_{(1+\varepsilon)} 
 |\widetilde{w_T}(t,{\bf u }, \cdot)|
 \Big| P_1^{u_1}  
 \Big(\frac{P_2}{f }\Big)^{u_2}  
  \Big(\frac{P_3}{f }\Big)^{u_3}
   \Big(\frac{C}{\delta_4  c_2}\Big)^{u_4} \Big|
  \sum_{\mathfrak{c}} |Z_{\mathfrak{c},t}({\bf u})|
  d{\bf u} dt.
\end{multline}

\subsection{Claiming bounds on $Z_j$, and estimating $\mathcal{T}$}
\label{section:EstimatingT}
In Section \ref{section:DirichletSeries}, we will show the following
\begin{mylemma}
\label{lemma:spectralmomentbound}
The function $Z_j({\bf u})$ has analytic continuation to $\text{Re}({\bf u}) \geq \sigma >  1/2$.  In this region 
it satisfies the bound
 \begin{equation}
 \label{eq:ZjuBound}
 \sum_{\substack{t_j \text{ level } \delta_2 \delta_5 \\ |t_j| \leq T}} |Z_j({\bf u})| \ll_{\sigma,\varepsilon} q^{\theta-\frac12} 
 \frac{(\delta_4, q )^{1/2}}{(k_1 k_1^*)^{\frac12} \delta_4^{1/2}} 
 T^{2 +\varepsilon}
 q^{\varepsilon} \operatorname{Poly}(|\bf{u}|),
\end{equation}
where $\operatorname{Poly}(|\bf{u}|)$ is some fixed polynomial in the absolute values of the coordinates of ${\bf u}$.
\end{mylemma}
The key feature is that this bound saves a factor $\delta_4^{1/2}$ which ultimately arises from \eqref{eq:smallHecke}.

Now we use Lemma \ref{lemma:spectralmomentbound} to estimate $\mathcal{T}_d^{\pm}$, and eventually $\mathcal{S}$.  
We do not require the factor $(k_1 k_1^*)^{-1/2}$ appearing in \eqref{eq:ZjuBound}, and in order to unify the treatment with the continuous spectrum, we shall only use a weaker bound with this factor omitted.

First consider the {\bf Non-oscillatory, Maass cases}.  Inserting the bound from Lemma \ref{lemma:spectralmomentbound} into \eqref{eq:TdExpressionSpectralMoment} (taking $\sigma = 1/2 + \varepsilon$), summing over $\delta_4$ (here is where the savings of $\delta_4^{1/2}$ is important), $f$, and $\delta_3$, and integrating over $s$ and ${\bf u}$, we obtain
\begin{equation}
\label{eq:TdBound}
 \mathcal{T}_d^{\pm} \ll q^{\varepsilon}
 \frac{g_0 k_1 c_2}{KC} \Big(\frac{\sqrt{aMN}}{C}\Big)^{\exponent} (M_2 N)^{1/2}    \frac{\sqrt{\delta_2}}{h}  
(Y^{-2\theta} + Y)
P^{1/2} 
   \Big(\frac{C}{  c_2}\Big)^{1/2}
  q^{\theta-\frac12} .
\end{equation}
Let us call $\mathcal{S}_d^{\pm}$ for the contribution to $\mathcal{S}$ from this part.  
Applying the additional summations that led from $\mathcal{S}$ to $\mathcal{S}'''$ (see \eqref{eq:S''intoS'''}, \eqref{eq:S''andS'}, \eqref{eq:SintoS'}), we obtain
\begin{multline*}
 \mathcal{S}_d^{\pm} \ll q^{\varepsilon}
 \sum_{a} \frac{1}{a^{3/2}} 
\sum_{c_2} \frac{1}{c_2^{3/2}} 
\sum_{d|c_2} d 
\sum_{\substack{k_1  } }  k_1^{1/2}
\sum_{m_1'} \frac{1}{\sqrt{m_1'}}
\sum_{r_1 r_2 r_3 = \delta_1} \sum_{e_1 | r_2 r_3} \sum_{e_2 | r_3} 
\sum_{\substack{g_0 | e_1 e_2 \delta_1 am_1'\\ g_0 \equiv 0 \shortmod{d}}}
\\
\times
 \frac{ g_0 k_1 c_2}{KC } \Big(\frac{\sqrt{aMN}}{C}\Big)^{\exponent} (M_2 N)^{1/2}    \frac{\sqrt{\delta_2}}{h}  
(Y^{-2\theta} + Y^{})
P^{1/2} 
   \Big(\frac{C}{  c_2}\Big)^{1/2}
  q^{\theta-\frac12}.
\end{multline*}

{\bf Convention}. Here and below, we have not written the truncation points for these outer summation variables.  In almost all cases, all that is necessary is to recall that all the variables may be bounded by some fixed power of $q$.  The only exception is that for some estimates we need to use that $m_1' \ll \frac{M_1}{k_1 c_2}$.

For convenience, we gather some of the previous definitions:
\begin{equation}
\label{eq:variablenames}
\begin{split}
h = e_1 e_2 r_1 r_2 r_3, \quad 
 \delta_1 &= r_1 r_2 r_3 = \frac{k_1 d}{(a,k_1d)}, \quad
\delta_2 = \frac{e_1 e_2 \delta_1 a m_1'}{g_0} = \frac{ham_1'}{g_0},
\\
N' h = N, 
\quad
m_1 &= k_1 c_2 m_1', \quad Y = \frac{g_0 k_1 \sqrt{CP}}{\sqrt{\delta_2 c_2} K}.
\end{split}
\end{equation}
With these substitutions, we obtain
\begin{equation*}
  \sqrt{\delta_2} (Y^{-2 \theta} + Y) 
 P^{1/2} 
   \Big(\frac{C}{  c_2}\Big)^{1/2} 
   = \frac{\delta_2 K}{g_0 k_1} (Y^{1-2\theta} + Y^2),
\end{equation*}
and hence
\begin{multline}
\label{eq:sweater}
 \mathcal{S}_d^{\pm} \ll q^{\varepsilon}
 \sum_{a} \frac{1}{a^{3/2}} 
\sum_{c_2} \frac{1}{c_2^{3/2}} 
\sum_{d|c_2} d 
\sum_{\substack{k_1  } }  k_1^{1/2}
\sum_{m_1'} \frac{1}{\sqrt{m_1'}}
\sum_{r_1 r_2 r_3 = \delta_1} \sum_{e_1 | r_2 r_3} \sum_{e_2 | r_3} 
\sum_{\substack{g_0 | e_1 e_2 \delta_1 am_1'\\g_0 \equiv 0 \shortmod{d} }}
\\
\times
  \Big(\frac{\sqrt{aMN}}{C}\Big)^{\exponent} M_2^{1/2} N^{1/2} \frac{  c_2}{C}   
 \frac{ a m_1' }{g_0 } (Y^{1-2\theta} + Y^2)
  q^{\theta-\frac12} .
\end{multline}

Now we note that in this non-oscillatory case, we have from \eqref{eq:Pbound} that
\begin{equation*}
  \frac{P (g_0 k_1)^2}{\delta_2 c_2} \ll q^{\varepsilon} \frac{K^3}{Na M_1},
\end{equation*}
which in particular means that $Y  \ll (\frac{CK}{NaM_1})^{1/2} q^{\varepsilon}$, which is independent of $g_0, k_1, c_2,$ etc.
Now it is evident that the sums over $g_0, e_1, e_2, r_1, r_2, r_3$ contribute at most $O(q^{\varepsilon})$, and the fact that $d|g_0$ will cancel the other visible factor of $d$ in \eqref{eq:sweater}.
With this observation, and performing minor simplifications, we have
\begin{multline*}
 \mathcal{S}_d^{\pm} \ll q^{\varepsilon}
 \sum_{a} \frac{1}{a^{1/2}} 
\sum_{c_2} \frac{1}{c_2^{1/2}} 
\sum_{d|c_2} 
\sum_{\substack{k_1  } } k_1^{1/2} 
\sum_{m_1'} \sqrt{m_1'}
\\
\times
 \Big(\frac{\sqrt{aMN}}{C}\Big)^{\exponent}  \frac{M_2^{1/2} N^{1/2}  }{C}   
q^{\theta-\frac12}    
     \Big(\Big(\frac{\sqrt{CK}}{\sqrt{NaM_1}}\Big)^{1-2\theta} + \Big(\frac{\sqrt{CK}}{\sqrt{NaM_1}}\Big)^2 \Big)
  .
\end{multline*}
 Trivially summing over $m_1'$ (recall $m_1' \ll \frac{M_1}{k_1 c_2}$), $k_1, d, c_2$, and finally $a$, we derive
\begin{equation*}
 \mathcal{S}_d^{\pm} \ll q^{\varepsilon}
\max_a 
 M_1^{3/2}
 \Big(\frac{\sqrt{aMN}}{C}\Big)^{\exponent}  \frac{M_2^{1/2} N^{1/2}  a^{1/2}}{C}   q^{\theta-\frac12}
    \Big(\Big(\frac{\sqrt{CK}}{\sqrt{NaM_1}}\Big)^{1-2\theta} + \Big(\frac{\sqrt{CK}}{\sqrt{NaM_1}}\Big)^2 \Big)
   .
\end{equation*}

Now we split it up into the cases from Lemmas \ref{lemma:BpostTransition} and \ref{lemma:CasePreNonOsc}.  In the {\bf case of Lemma \ref{lemma:BpostTransition}}, we have $\exponent = -1$, and
\begin{equation*}
 M_1^{3/2}
 \Big(\frac{\sqrt{aMN}}{C}\Big)^{\exponent}  \frac{M_2^{1/2} N^{1/2}  a^{1/2}}{C}   q^{\theta-\frac12} = \frac{M_1}{q^{\frac12-\theta}}.
\end{equation*}
Meanwhile, using $K \asymp M_2^{-1} \sqrt{aMN}$ (see \eqref{eq:Ksize}), we have
\begin{equation*}
 \frac{CK}{NaM_1} \asymp \frac{C}{\sqrt{aMN}} \ll q^{\varepsilon}.
\end{equation*}
via \eqref{eq:PostTransition}.
Therefore, in this case we have
\begin{equation}
\label{eq:SdFinalBound}
 \mathcal{S}_d^{\pm} \ll \frac{M_1}{q^{1/2}} q^{\theta+\varepsilon} \ll q^{\theta + \varepsilon}.
\end{equation}

In the {\bf case of Lemma \ref{lemma:CasePreNonOsc}}, we have $\frac{CK}{NaM_1} \gg q^{-\varepsilon}$, and $\exponent = \kappa - 1 \geq 1$, so with easy simplifications, we derive
\begin{equation*}
  \mathcal{S}_d^{\pm} \ll q^{\varepsilon}
\max_a 
 \frac{ M_2 K}{C}    
  \frac{M_1 q^{\theta}}{q^{\frac12} }.
\end{equation*}
Since $\frac{KM_2}{C} \ll q^{\varepsilon}$ in this case, we obtain the same bound as \eqref{eq:SdFinalBound}.

{\bf The Non-oscillatory, holomorphic cases} are nearly identical, so we omit the proofs.

Now consider the {\bf oscillatory, Maass case}, where we treat \eqref{eq:TdSpectralBoundOscillatoryCase}.  Following the same steps as the non-oscillatory cases, we obtain
\begin{equation*}
 \mathcal{T}_d^{+} \ll 
 \frac{g_0 k_1 }{K} \Big(\frac{\sqrt{aMN}}{C}\Big)^{\exponent} (M_2 N)^{1/2}    \frac{\sqrt{\delta_2} }{h}
 \Big(\frac{CK}{a M_1 N}\Big) P^{1/2}
   \Big(\frac{  c_2}{C}\Big)^{1/2} 
   \frac{q^{\theta}}{q^{1/2}}.
\end{equation*}
After some simplifications, we have
\begin{multline} \label{eq:SdOsc}
 \mathcal{S}_d^{+} \ll q^{\varepsilon}
 \sum_{a} \frac{1}{a^{3/2}} 
\sum_{c_2} \frac{1}{c_2^{3/2}} 
\sum_{d|c_2} d 
\sum_{\substack{k_1  } }  k_1^{1/2}
\sum_{m_1'} \frac{1}{\sqrt{m_1'}}
\sum_{r_1 r_2 r_3 = \delta_1} \sum_{e_1 | r_2 r_3} \sum_{e_2 | r_3} 
\sum_{\substack{g_0 | e_1 e_2 \delta_1 am_1'\\ g_0 \equiv 0 \shortmod{d}}}
\\
\frac{g_0 k_1}{K}
\Big(\frac{\sqrt{aMN}}{C}\Big)^{\exponent}
  (M_2 N)^{1/2}  \Big(\frac{  c_2}{C}\Big)^{1/2}   \frac{\sqrt{\delta_2} }{h} 
 \Big(\frac{CK}{a M_1 N}\Big) P^{1/2} 
   \frac{q^{\theta}}{q^{1/2}} .
\end{multline}
We need to remember  the origins of these variables.  We have
\begin{equation}
\label{eq:PsizeOscillatorySimplified}
 P = P_1 P_2 P_3 \asymp \frac{(NaM_1)^3 k_0'^3}{C^3 K^3 N'} \asymp \Big(\frac{N a  M_1  }{C K}\Big)^3 \frac{K^3}{N} \frac{h}{(g_0 k_1)^3}.
\end{equation}
Thus the bound becomes
\begin{multline*}
 \mathcal{S}_d^{+} \ll q^{\varepsilon}
 \sum_{a} \frac{1}{a^{3/2}} 
\sum_{c_2} \frac{1}{c_2^{3/2}} 
\sum_{d|c_2} d 
\sum_{\substack{k_1  } }  k_1^{1/2}
\sum_{m_1'} \frac{1}{\sqrt{m_1'}}
\sum_{r_1 r_2 r_3 = \delta_1} \sum_{e_1 | r_2 r_3} \sum_{e_2 | r_3} 
\sum_{\substack{g_0 | e_1 e_2 \delta_1 am_1'\\ g_0\equiv 0 \shortmod{d}}}
\\
\frac{g_0 k_1 }{K} \Big(\frac{\sqrt{aMN}}{C}\Big)^{\exponent} (M_2 N)^{1/2} \Big(\frac{  c_2}{C}\Big)^{1/2}   \sqrt{\frac{h am_1'}{g_0}}  
 \Big(\frac{N a  M_1  }{C K}\Big)^{1/2} 
 \Big( \frac{K^3}{N} \frac{h}{(g_0 k_1)^3} \Big)^{1/2} 
   \frac{q^{\theta}}{h q^{1/2}} .
\end{multline*}
We see that the sum over $g_0$ gives $O(d^{-1}q^{\varepsilon})$, and the $h$-dependence cancels out entirely, so that the $\delta_1$-dependence is also essentially gone.  Thus, we obtain
\begin{multline*}
 \mathcal{S}_d^{+} \ll q^{\varepsilon} \frac{q^{\theta}}{q^{1/2}}
 \sum_{a} \frac{1}{a^{ }} 
\sum_{c_2} \frac{1}{c_2^{ }} 
\sum_{d|c_2} 
\sum_{\substack{k_1  } }  (k_1 k_1^*)^{-1/2} 
\sum_{m_1'} 
\\
\frac{(M_2 N)^{1/2}}{K C^{1/2}} \Big(\frac{\sqrt{aMN}}{C}\Big)^{\exponent} 
 \Big(\frac{N a  M_1  }{C K}\Big)^{1/2} 
 \Big( \frac{K^3}{N}\Big)^{1/2} .
\end{multline*}
Now we sum over all the remaining variables, giving in all
\begin{equation*}
 \mathcal{S}_d^{+} \ll q^{\varepsilon} \frac{M_1 q^{\theta}}{q^{1/2}}
\frac{(M_2 N)^{1/2}}{K C^{1/2}} \Big(\frac{\sqrt{aMN}}{C}\Big)^{\exponent}  
 \Big(\frac{N a  M_1  }{C K}\Big)^{1/2} 
 \Big( \frac{K^3}{N}\Big)^{1/2} .
\end{equation*}
Simplifying (in particular, $\exponent = \kappa -1$ here), we obtain
\begin{equation*}
 \mathcal{S}_d^{+} \ll q^{\varepsilon} \frac{M_1 q^{\theta}}{q^{1/2}} 
\frac{ M a N}{C^2} 
  .
\end{equation*}
Since $\sqrt{MaN} \ll C q^{\varepsilon}$ (see \eqref{eq:PreTransition}), we obtain the same bound as 
\eqref{eq:SdFinalBound}. The {\bf oscillatory, holomorphic case} is similar, but even simpler.

In summary, this shows the desired bound for the Maass forms and holomorphic forms.

\subsection{Claiming bounds on  $Z_{\mathfrak{c},t}$, and estimating $\mathcal{T}_c$}
\label{section:ContinuousSpectrumBound}
Recall the definition \eqref{eq:ZctDef}.  Define $\flooroot$ to be the multiplicative function defined on prime powers by
\begin{equation}
\label{eq:floorootDefinition}
\flooroot(p^{\alpha}) = p^{\lfloor \alpha/2 \rfloor}.
\end{equation}

\begin{mylemma}
\label{lemma:ZtcBound}
 The function $Z_{\c, t}({\bf u})$ has a decomposition of the following form.  We have
 \begin{equation*}
  Z_{\c, t}(u_1, u_2, u_3, u_4) = (Z_1^0(u_2, u_3) + Z_1^*(u_2, u_3))  (Z_2^0(u_1, u_4) + Z_2^*(u_1, u_4)),
 \end{equation*}
where for $i=1,2$, $Z_i^*(\alpha,\beta)$ has analytic continuation to $\text{Re}(\alpha, \beta) \geq \sigma > 1/2$, and $Z_i^0(\alpha,\beta)$ is analytic for $\text{Re}(\alpha, \beta) \geq \sigma > 1$.  
For $\text{Re}({\bf u}) \geq \sigma > 1/2$, we have
\begin{equation}
\label{eq:Z1*Z2*bound}
\int_{|t| \leq T} \sum_{\c} |Z_1^*(u_2, u_3) Z_2^*(u_1, u_4)| dt \ll_{{\bf u}, \varepsilon} q^{\varepsilon} T^{2+\varepsilon} \Big(\frac{(\delta_4, q)}{ q  }\Big)^{1/2}
\frac{\flooroot(\delta_2) \flooroot(\delta_3)^{3/2}}{\sqrt{\delta_2 \delta_5 }}
.
\end{equation}

For $\text{Re}({\bf u}) \geq \sigma > 1$, we have
\begin{equation}
\label{eq:Z10Z20bound}
 \sum_{\c} |Z_1^0(u_2, u_3) Z_2^0(u_1, u_4)|  \ll_{{\bf u}, \varepsilon} 
 (q (1+ |t|))^{\varepsilon} \frac{(\delta_4, q)}{q \sqrt{k_1 k_1^*}} \frac{1 }{\delta_2 \delta_5 }
 .
\end{equation}

For $\text{Re}(u_1, u_4) \geq \sigma > 1$ and $\text{Re}(u_2, u_3) \geq \sigma' > 1/2$, we have
\begin{equation}
\label{eq:Z1*Z20bound}
\int_{|t| \leq T} \sum_{\c} |Z_1^*(u_2, u_3) Z_2^0(u_1, u_4)| dt 
\ll_{{\bf u}, \varepsilon} 
q^{\varepsilon} T^{1+\varepsilon}
\frac{(\delta_4, q)}{q \sqrt{k_1 k_1^*}} 
  \frac{\flooroot(\delta_2) \sqrt{\flooroot(\delta_3)}}{\delta_2 \sqrt{\delta_5}}
  .
\end{equation}

For $\text{Re}(u_2, u_3) \geq \sigma > 1$ and $\text{Re}(u_1, u_4) \geq \sigma' > 1/2$, we have
\begin{equation}
\label{eq:Z10Z2*bound}
\int_{|t| \leq T} \sum_{\c} |Z_1^0(u_2, u_3) Z_2^*(u_1, u_4)| dt 
\ll_{{\bf u}, \varepsilon}
q^{\varepsilon} T^{1+\varepsilon}
\Big(\frac{(\delta_4, q)}{q} \Big)^{1/2} 
  \frac{\flooroot(\delta_2) \flooroot(\delta_3) }{ \delta_2 \delta_5}
  .
\end{equation}
The implied dependence on ${\bf u}$ is at most polynomial, as in Lemma \ref{lemma:spectralmomentbound}.
\end{mylemma}
We remark on some important features of the above bounds.  In \eqref{eq:Z1*Z2*bound} and \eqref{eq:Z10Z2*bound} we require a factor $\delta_5^{-1/2}$ (or better) to secure convergence of the sum over $\delta_4$.  The overall power of $k_1$ is also important for securing convergence in each case.  In terms of the final power of $q$ that occurs in our bound on $\mathcal{S}_c$, the most important feature is the power of $\delta_2$.  This is because $\delta_2$ contains the $m_1'$ variable which can be as large as $q^{1/2+\varepsilon}$.
Note that although $\flooroot(n)$ may occasionally be as large as $\sqrt{n}$, it is small on average, indeed $\sum_{n \leq x} \flooroot(n) \ll x \log x$.

Using Lemma \ref{lemma:ZtcBound}, we bound $\mathcal{T}_c$.  
For the non-oscillatory cases, we return to \eqref{eq:TcExpressionSpectralMoment}.
Technically, we should return to \eqref{eq:TcFormulaSpectralPreAbsoluteValues}, decompose $\mathcal{T}_c^\pm$ according to $Z_{\c, t} = (Z_1^* + Z_1^0)(Z_2^* + Z_2^0)$ into four pieces, shift contours appropriately, and only then apply the absolute values.   
We found it slightly easier to bound $Z_1^0$ and $Z_2^0$ slightly to the right of the $1$-line instead of bounding the residues of $Z_1$ and $Z_2$, but this is more-or-less equivalent.

Then (note that the sum over $f$ converges absolutely, and the $t$-integral is easily estimated, so we may simplify a bit in these aspects) 
\begin{multline*}
 \mathcal{T}_c^{\pm} \ll q^{\varepsilon}
 \frac{g_0 k_1 c_2}{KC} \Big(\frac{\sqrt{aMN}}{C}\Big)^{\exponent} 
 (M_2 N)^{1/2} \frac{\sqrt{\delta_2}}{h}
\sum_{\delta_3 | \frac{c_2}{(g_0, c_2)}}
 \sum_{(\delta_4, \delta_2 g_0 m_1' ) = 1}   
 \flooroot(\delta_2) \flooroot(\delta_3)
\\
\Big( \frac{(\delta_4, q)^{1/2}}{q^{1/2}}   (1+Y) \frac{(P_1 P_2 P_3 C)^{1/2} \flooroot(\delta_3)^{1/2}}{\sqrt{\delta_2 \delta_4 \delta_5 c_2 }}
+
 \frac{(\delta_4, q)}{q} \frac{P_1 P_2 P_3C}{\delta_2 \delta_4 \delta_5 c_2 \sqrt{k_1 k_1^*}} 
\\
+
 \frac{P_1 (P_2 P_3)^{1/2} C}{\delta_2 \delta_4 c_2 \sqrt{\delta_5}} \frac{(\delta_4, q)}{q \sqrt{k_1 k_1^*}}
+
 \frac{ P_1^{1/2} P_2 P_3 C^{1/2}}{\delta_4^{1/2} c_2^{1/2} \delta_2 \delta_5}
 \frac{(\delta_4, q)^{1/2}}{q^{1/2}}
\Big).
\end{multline*}
Using $\delta_5 \geq \sqrt{\delta_3 \delta_4}$ (recall that $\delta_5 = [\delta_3,\delta_4]$), the sums over $\delta_3$ and $\delta_4$ are easily evaluated, and lead to a factor of size at most $O(q^{\varepsilon})$; the only slightly tricky case uses instead
\begin{equation}\label{eq:trickyDeltaSummation}
 \flooroot(\delta_3)^{1/2}  
 \sum_{\delta_4} \frac{(\delta_4, q)^{1/2}}{\sqrt{\delta_4 \delta_5}} = 
 \flooroot(\delta_3)^{1/2}
 \sum_{\delta_4} \frac{(\delta_4, \delta_3 q)^{1/2}}{\delta_4\sqrt{ \delta_3}} 
 \ll 
 \flooroot(c_2) \frac{(\delta_3 q)^{\varepsilon}}{\sqrt{\delta_3}}.
\end{equation}

It is helpful to observe the following nice simplification.  At this point we can see that the first term within the parentheses which occurred from $Z_1^* Z_2^*$ will lead to the same bound we obtained on $\mathcal{T}_d^{\pm}$, by comparison to \eqref{eq:TdBound}. The only difference is the benign factor of $\flooroot(c_2)$, which does not make the sum over $c_2$ appreciably larger, since $\sum_{c_2} c_2^{-1} \flooroot(c_2) \ll q^{\varepsilon}$.  Actually, apart from $\flooroot(c_2)$, the bound is better in two ways: firstly, the factor $q^{\theta}$ may be omitted, and secondly, instead of using $\frac{\flooroot(\delta_2)}{\sqrt{\delta_2}} \leq 1$, we could use that $\flooroot(n)$ is $O(n^{\varepsilon})$ on average, which could lead to a saving of the factor $M_1^{1/2}$.  Instead of carrying through the calculations, we will simply abbreviate this term by $(**)$ in the forthcoming calculations.

Next we wish to sum over the outer variables that make $\mathcal{S}_c$ from $\mathcal{T}_c$.  To this end, we need to write the $P_i$, $Y$, and $\delta_2$ variables in terms of these outer ones.  Let $P_i^* = \frac{K}{N_i}$, so that $P_i \ll q^{\varepsilon} P_i^* \frac{h_i}{g_0 k_1}$ where $h_1 = e_1  r_1$, $h_2 = e_2 r_2$, and $h_3 = r_3$ (so $h = h_1 h_2 h_3$). 
With this, we obtain
\begin{multline*}
 \mathcal{T}_c^{\pm} \ll q^{\varepsilon}
 \frac{g_0 k_1 c_2}{KC} \Big(\frac{\sqrt{aMN}}{C}\Big)^{\exponent} 
 (M_2 N)^{1/2} \frac{\sqrt{\delta_2}}{h}
 \flooroot(\delta_2) 
\\
\Big[   (**)
+
 \frac{h}{(g_0 k_1)^3} \frac{P_1^* P_2^* P_3^* C}{q\delta_2  c_2 \sqrt{k_1 k_1^*}} 
+
\frac{h_1 (h_2 h_3)^{1/2}}{(g_0 k_1)^2} \frac{P_1^* (P_2^* P_3^*)^{1/2} C}{q \delta_2  c_2  \sqrt{k_1 k_1^*}} 
+
\frac{h_1^{1/2} h_2 h_3}{(g_0 k_1)^{5/2}} \frac{ (P_1^*)^{1/2} P_2^* P_3^* C^{1/2}}{q^{1/2} c_2^{1/2} \delta_2 }
\Big].
\end{multline*}
Recall that
\begin{equation*}
 \mathcal{S}_c^{\pm} \ll q^{\varepsilon} \sum_a \frac{1}{a^{3/2}} \sum_{c_2} \frac{1}{c_2^{3/2}} \sum_{d | c_2} d \sum_{k_1} k_1^{1/2} \sum_{m_1'} \frac{1}{\sqrt{m_1'}} \sum_{r_1 r_2 r_3 = \delta_1} \sum_{e_1 | r_2 r_3} \sum_{e_2 | r_3} \sum_{\substack{g_0 | e_1 e_2 \delta_1 a m_1' \\ g_0 \equiv 0 \shortmod{d}}} \mathcal{T}_c^{\pm},
\end{equation*}
and that $\delta_2 = \frac{h a m_1'}{g_0}$, and $h = e_1 e_2 \delta_1 = e_1 e_2 r_1 r_2 r_3$.

Next we analyze the sum over $g_0$ in all four terms.  For the $\flooroot$ part, we use that $\flooroot(\delta_2) = \flooroot(\frac{e_1 e_2 \delta_1 am_1'}{g_0}) \leq \flooroot(e_1 e_2 \delta_1 am_1') = \flooroot(h a m_1')$, and otherwise we see that the overall power of $g_0$ is negative in all terms, and so the smallest value of $g_0$, namely $d$, leads to the dominant part.

Putting this together, and simplifying, we obtain
\begin{multline*}
 \mathcal{S}_c^{\pm} \ll 
 q^{\varepsilon} \sum_a \frac{1}{a^{3/2}} 
 \sum_{c_2} \frac{1}{c_2^{3/2}} \sum_{d|c_2} d \sum_{k_1} k_1^{3/2} \sum_{m_1'} \frac{1}{\sqrt{m_1'}} \sum_{r_1 r_2 r_3 = \delta_1} \sum_{e_1 | r_2 r_3} \sum_{e_2 | r_3} 
 \frac{ (M_2 N)^{1/2} }{CK} \Big(\frac{\sqrt{aMN}}{C}\Big)^{\exponent} 
 \\
 \frac{\flooroot(h a m_1')}{\sqrt{ha m_1'}}
\Big[ (**)
+
\frac{P^* C}{q d^{3/2} c_2 k_1^3 \sqrt{k_1 k_1^*}} 
+
\frac{P_1^* (P_2^* P_3^*)^{1/2} C}{q k_1^2 c_2  \sqrt{d h_2 h_3 k_1 k_1^*}  } 
+
\frac{(P_1^*)^{1/2} P_2^* P_3^* C^{1/2}}{d k_1^{5/2}  \sqrt{q h_1 c_2} } 
\Big].
\end{multline*}

Our next goal is to estimate the sum over $m_1'$.  Since $m_1'$ is independent of $\delta_1$ (and hence $e_1, e_2$), we may move the sum over $m_1'$ to the inside.  
We shall use the following estimate:
\begin{equation*}
\sum_{n \leq X} \frac{\flooroot(nN)}{n} \ll \flooroot(N) (XN)^{\varepsilon},
\end{equation*}
  which can
be proved by elementary methods.  Applying this to the sums over $m_1'$, we obtain with easy simplifications
\begin{multline}
\label{eq:ScExpression}
 \mathcal{S}_c^{\pm} \ll 
 q^{\varepsilon} \sum_a \frac{1}{a^{3/2}} 
 \sum_{c_2} \frac{1}{c_2^{3/2}} \sum_{d|c_2} d \sum_{k_1} k_1^{3/2}  \sum_{r_1 r_2 r_3 = \delta_1} \sum_{e_1 | r_2 r_3} \sum_{e_2 | r_3} 
 \frac{ (M_2 N)^{1/2} }{CK} \Big(\frac{\sqrt{aMN}}{C}\Big)^{\exponent} 
 \\
 \frac{\flooroot(h a )}{\sqrt{ha }}
\Big[ (**)
+
\frac{P^* C}{q d^{3/2} c_2 k_1^3 \sqrt{k_1 k_1^*}} 
+
\frac{P_1^* (P_2^* P_3^*)^{1/2} C}{q  c_2 k_1^2 \sqrt{d h_2 h_3 k_1 k_1^*}   } 
+
\frac{(P_1^*)^{1/2} P_2^* P_3^* C^{1/2}}{d k_1^{5/2}  \sqrt{q h_1 c_2} } 
\Big].
\end{multline}
Using $\flooroot(ha) \leq \sqrt{ha}$, we can easily see that the outer variables sum to give no significant contribution.  
Therefore,  we have
\begin{multline*}
 \mathcal{S}_c^{\pm} \ll 
 q^{\varepsilon}  \max_a
 \frac{ (M_2 N)^{1/2} }{CK} \Big(\frac{\sqrt{aMN}}{C}\Big)^{\exponent} 
\\
\Big[(**)
+
q^{-1} P^* C 
+
q^{-1} P_1^* (P_2^* P_3^*)^{1/2}
  C
+
q^{-1/2} (P_1^*)^{1/2} P_2^* P_3^*
 C^{1/2} \Big].
\end{multline*}
Substituting for $P_i^*$ and simplifying, we obtain
\begin{equation}
\label{eq:ScboundMiddleOfProof}
 \mathcal{S}_c^{\pm} \ll 
 q^{\varepsilon}  \max_a
 \frac{ (M_2 N)^{1/2} }{CK} \Big(\frac{\sqrt{aMN}}{C}\Big)^{\exponent} 
\Big[(**)
+
\frac{K^3 C}{Nq} 
+
\frac{K^2 C}{q N_1 \sqrt{N_2 N_3}}
+
\frac{K^{5/2}  C^{1/2}}{q^{1/2} N_1^{1/2} N_2 N_3}
 \Big].
\end{equation}

Now we split once more into the two types of non-oscillatory behavior.  
The {\bf post-transition} case from Lemma \ref{lemma:BpostTransition}
has $\frac{\sqrt{aMN}}{C} \gg q^{\varepsilon}$, which leads to $\exponent = -1$ and $K \asymp \frac{(aMN)^{1/2}}{M_2}$.  Therefore,
\begin{equation*}
 \mathcal{S}_c^{\pm} \ll 
 q^{\varepsilon}  \max_a
 (M_2 N)^{1/2}
\Big[(**)
+
\frac{a MN}{M_2^2 Nq} 
+
\frac{\sqrt{aMN}}{q M_2 N_1 \sqrt{N_2 N_3}}
+
\frac{\sqrt{aMN}}{q^{1/2} M_2^{3/2} N_1^{1/2} N_2 N_3}
 \Big].
\end{equation*}
This simplifies to give
\begin{equation}
\label{eq:ScBoundSomewhatSimplified}
 \mathcal{S}_c^{\pm} \ll 
 q^{\varepsilon} \Big[1 
+
\frac{\sqrt{M_1}}{\sqrt{M_2}} \frac{\sqrt{M_1 a^2 N}}{q}
+
\frac{M_1^{1/2} \sqrt{a N_2 N_3}}{q}
+
\frac{\sqrt{M_1}}{\sqrt{M_2}} \frac{\sqrt{aN_1}}{\sqrt{q}}
 \Big].
\end{equation}
Using $M_1 \ll M_2$, $M_1 \ll q^{1/2+\varepsilon}$, and $aN_i \ll q^{1/2+\varepsilon}$, we deduce that $\mathcal{S}_c^{\pm} \ll q^{\varepsilon}$.  One may observe that the part of $\mathcal{S}_c^{\pm}$ arising from $Z_1^* Z_2^0$ and $Z_1^0 Z_2^*$ contributes at most $O(q^{-1/4+\varepsilon})$.  With some additional work, one could show the contribution from $Z_1^* Z_2^*$ is also at most $O(q^{-1/4+\varepsilon})$. 

For the {\bf non-oscillatory, pre-transition} case from Lemma \ref{lemma:CasePreNonOsc} with $\exponent  = \kappa -1 \geq 2$ (here is the only place where the choice of $\kappa = 2$ does not work), we have $K \ll q^{\varepsilon} \frac{C}{M_2}$, and $C \gg q^{\varepsilon} \sqrt{aMN}$, so we obtain
\begin{equation}
\label{eq:ScBoundNonoscillatoryCase2}
 \mathcal{S}_c^{\pm} \ll 
 q^{\varepsilon}  \max_a
  (M_2 N)^{1/2} 
\Big[(**)
+
\frac{aMN  }{M_2^2 Nq} 
+
\frac{\sqrt{aMN}  }{q M_2 N_1 \sqrt{N_2 N_3}}
+
\frac{\sqrt{aMN} }{M_2^{3/2} q^{1/2} N_1^{1/2} N_2 N_3}
 \Big].
\end{equation}
This is precisely the same bound as in \eqref{eq:ScBoundSomewhatSimplified}, and so $\mathcal{S}_c^{\pm} \ll q^{\varepsilon}$.  Actually, we only need $\kappa -1 \geq 2$ for the term arising from $Z_1^0 Z_2^0$.

Finally, we consider the {\bf Oscillatory} case (where recall $\exponent = \kappa - 1$ and only the $+$ sign enters).  For this, we return to \eqref{eq:TcSpectralBoundOscillatoryCase}, that is,
\begin{multline*}
 \mathcal{T}_c^{+} \ll 
 \frac{g_0 k_1 c_2}{KC} \Big(\frac{\sqrt{aMN}}{C}\Big)^{\kappa - 1} (M_2 N)^{1/2}    \frac{\sqrt{\delta_2} }{h}
\sum_{\delta_3 | \frac{c_2}{(g_0, c_2)}}
 \sum_{(\delta_4, \delta_2 g_0 m_1' ) = 1}   
\sum_{\substack{(f,\delta_2 \delta_4 c_0') = 1 \\ \eqref{eq:delta3def} \text{ is true}}}  
\int_{\substack{ |t| \ll Y'q^{\varepsilon}}} 
\\
\Big(\frac{CK}{a M_1 N}\Big)^2 
 \int_{(1+\varepsilon)} 
 |\widetilde{w_T}(t,{\bf u }, \cdot)|
 \Big| P_1^{u_1}  
 \Big(\frac{P_2}{f }\Big)^{u_2}  
  \Big(\frac{P_3}{f }\Big)^{u_3}
   \Big(\frac{C}{\delta_4  c_2}\Big)^{u_4} \Big|
  \sum_{\mathfrak{c}} |Z_{\mathfrak{c},t}({\bf u})|
  d{\bf u} dt,
\end{multline*}
where again we should technically move the contours before applying the absolute values.
Then we obtain
\begin{multline*}
 \mathcal{T}_c^{+} \ll q^{\varepsilon}
 \frac{g_0 k_1 c_2}{KC} \Big(\frac{\sqrt{aMN}}{C}\Big)^{\kappa - 1} (M_2 N)^{1/2}    \frac{\sqrt{\delta_2} }{h}
\sum_{\delta_3 | \frac{c_2}{(g_0, c_2)}}
 \sum_{(\delta_4, \delta_2 g_0 m_1' ) = 1}   
\Big(\frac{CK}{a M_1 N}\Big)^2  
\\
\flooroot(\delta_2) \flooroot(\delta_3)
\Big[\frac{(\delta_4, q)^{1/2}}{q^{1/2}}   Y'^2 \frac{(P_1 P_2 P_3 C)^{1/2} \flooroot(\delta_3)^{1/2}  }{\sqrt{\delta_2 \delta_4 \delta_5 c_2 }}
+
 Y' \frac{(\delta_4, q)}{q} \frac{P_1 P_2 P_3C}{\delta_2 \delta_4 \delta_5 c_2 \sqrt{k_1 k_1^*}} 
\\
+
 Y' \frac{P_1 (P_2 P_3)^{1/2} C}{\delta_2 \delta_4 c_2 \sqrt{\delta_5}} \frac{(\delta_4, q)}{q \sqrt{k_1 k_1^*}}
+
 Y' \frac{ P_1^{1/2} P_2 P_3 C^{1/2}}{\delta_4^{1/2} c_2^{1/2} \delta_2 \delta_5}
 \frac{(\delta_4, q)^{1/2}}{q^{1/2}}
\Big].
\end{multline*}
Luckily, we may re-use some of the previous analysis in the non-oscillatory cases.  We wish to sum over all the outer variables.  Note that $Y'$ is independent of them, and we have $P_i \asymp \frac{N aM_1}{CK} \frac{k_0'}{N_i'}$; previously we had $P_i \ll \frac{k_0'}{N_i'}$, so the only difference here is the extra factor $\frac{N a M_1}{CK}$ (which happens to be $Y'^2$).  Therefore, the previous method of bounding the outer variables works identically as in this case.  
This time the term arising from $Z_1^*Z_2^*$ is identical to \eqref{eq:SdOsc}, save for the $\flooroot(c_2)$. We again denote it by $(**)$.
Therefore, we have by altering \eqref{eq:ScboundMiddleOfProof} with the appropriate factors of $Y'$ that
\begin{equation*}
 \mathcal{S}_c^{\pm} \ll 
 q^{\varepsilon}  \max_a
 \frac{ (M_2 N)^{1/2} }{CK Y'^4} \Big(\frac{\sqrt{aMN}}{C}\Big)^{\kappa-1}  
\Big[(**)
+
 \frac{Y'^7 K^3 C}{Nq} 
+
\frac{Y'^5 K^2 C}{q N_1 \sqrt{N_2 N_3}}
+
 \frac{Y'^{6} K^{5/2}  C^{1/2}}{q^{1/2} N_1^{1/2} N_2 N_3}
 \Big].
\end{equation*}

Substituting for $Y'$, simplifying, and using $K \ll q^{\varepsilon} \frac{C}{M_2}$, this becomes
\begin{multline*}
 \mathcal{S}_c^{\pm} \ll 
 q^{\varepsilon}  \max_a
 (M_2 N)^{1/2}  \Big(\frac{\sqrt{aMN}}{C}\Big)^{\kappa-1} 
 \\
\Big[(**)
+
 \frac{(M_1 a N)^{3/2}}{C M_2^{1/2} Nq}  
+
\frac{(M_1 a N)^{1/2}  }{q M_2^{1/2} N_1 \sqrt{N_2 N_3}}
+
 \frac{ M_1 a N}{q^{1/2} C M_2^{1/2} N_1^{1/2} N_2 N_3}
 \Big].
\end{multline*}
Using $C \gg q^{-\varepsilon} \sqrt{MaN}$ and $\kappa -1 \geq 1$, this gives
\begin{equation*}
 \mathcal{S}_c^{\pm} \ll 
 q^{\varepsilon}  \max_a
\Big[1
+
 \frac{M_1}{M_2} \frac{ (M_2 a^2 N)^{1/2} }{ q}  
+
\frac{(M_1 a N_2 N_3)^{1/2}  }{q}
+
 \frac{\sqrt{M_1}}{\sqrt{M_2}} \frac{\sqrt{a N_1}}{\sqrt{q}}
 \Big].
\end{equation*}
Again as in the non-oscillatory case, we see that
 $\mathcal{S}_c^{\pm} \ll q^{\varepsilon}$.

\section{Bounding the Dirichlet series}
\label{section:DirichletSeries}
\subsection{Discrete spectrum}
In this section we prove Lemma \ref{lemma:spectralmomentbound}.
Towards this, we develop some properties of an auxilliary Dirichlet series with the following
\begin{mylemma}\label{lemma:oldformDirichletSeriesBound}
Suppose $N = LM$, $f^*$ is a newform of level $M$, and $d,Q$ are nonzero integers. For $\ell | L$, let $f = f^*\vert_{\ell}$, write $d/\ell = d_1/\ell_1$ in lowest terms (so $d = (d,\ell) d_1$, $\ell = (d,\ell) \ell_1$), and let $d_1 = d_M d_0$ with $d_M |M^{\infty}$ and $(d_0, M) =1$. Define the Dirichlet series
\begin{equation*}
Z_{d,\ell, Q}(s,u) := \sum_{\substack{m,n \geq 1 \\ (n,Q) = 1}} \frac{\nu_{f^*\vert_{\ell}}(dmn)}{m^{s } n^{u}},
\end{equation*} 
initially for $\text{Re}(s), \text{Re}(u)$ large.
Then $Z_{d,\ell, Q}$ has analytic continuation to $\text{Re}(s), \text{Re}(u) \geq \sigma > 1/2$, wherein it
satisfies the bound
\begin{equation}
\label{eq:Zbound}
 |Z_{d,\ell,Q}(s,u)| \ll_{\sigma} 
 |\nu_{f^*}(1)|
 (d,\ell)^{1/2} d_M^{-1/2}  d_0^{\theta} (dNQ)^{\varepsilon} |L(f^*, s) L(f^*, u)|.
\end{equation}
\end{mylemma}

\begin{proof}
Firstly, from \eqref{eq:fvertellFourierCoefficientFormula}, we have
\begin{equation*}
Z_{d,\ell, Q}(s,u) =
\sum_{\substack{m,n \geq 1 \\ (n,Q) = 1}} \frac{ \ell^{1/2} \nu_{f^*}(dmn/\ell)}{m^{s } n^{u}}.
\end{equation*}    
We have $\nu_{f^*}(dmn/\ell) = \nu_{f^*}(d_1 mn/\ell_1) =  \lambda_{f^*}(d_M) \nu_{f^*}(d_0 mn/\ell_1)$, and so
\begin{equation*}
 Z_{d,\ell,Q}(s,u) = \ell^{1/2} \lambda_{f^*}(d_M) \nu_{f^*}(1) 
 \sum_{\substack{m,n \geq 1 \\ (n,Q) = 1}} \frac{\lambda_{f^*}(d_0 mn/\ell_1)}{m^{s } n^{u}}.
\end{equation*}
Using \eqref{eq:smallHecke}, and complete multiplicativity of Hecke eigenvalues for primes dividing $M$, we get that $|\lambda_{f^*}(d_M) | \leq d_M^{-1/2}.$

By an exercise with the Hecke relations (somewhat in the spirit of \eqref{eq:LsquaredWithHecke}), one may derive the analytic continuation and the bound
\begin{equation*}
 \sum_{\substack{m,n \geq 1 \\ (n,Q) = 1}} \frac{\lambda_{f^*}(d_0 mn/\ell_1)}{m^s n^u} \ll_{\sigma} \frac{(d N Q)^{\varepsilon}}{\ell_1^{1/2}} d_0^{\theta} |L(f^*,s) L(f^*, u)|,
\end{equation*}
where recall that $\Re(s), \Re(w) \geq \sigma > 1/2$.
\end{proof}

Now we proceed to prove Lemma \ref{lemma:spectralmomentbound}.
Recall the definition
\begin{equation}
\label{eq:ZjuFormulaReminder}
 Z_j({\bf u}) = \Big(\sum_{\substack{(c_0',fg_0 m_1') = 1\\ \delta_4 c_0'\equiv 0 \shortmod{q k_1 k_1^*}}} 
 \sum_{p_1 \geq 1} \frac{\nu_{\infty,j}(p_1 \delta_4 c_0')}{p_1^{u_1} c_0'^{u_4}} \Big)
\Big(\sum_{\substack{ p_2 , p_3   \geq 0 }}  \frac{  \nu_{\frac{1}{\delta_5},j}(p_2  p_3 ) }{ p_2^{u_2} p_3^{u_3} } \Big).
\end{equation}

We begin by decomposing into newforms.  
By the choice of basis from Section \ref{subsec:oldForms}, we have
\begin{multline}
\label{eq:ZjuFormulaReminder2}
 \sum_{|t_j| \leq T} |Z_j({\bf u})|
 \ll
 \sum_{AB = \delta_2 \delta_5} \sum_{\substack{f^* \text{ new, level $B$} \\ |t_{f^*}| \leq T}}
 \\
 \times
  \sum_{\substack{\ell | A \\ \ell' | A}}
  \Big|
  \Big(\sum_{\substack{(c_0',fg_0 m_1') = 1\\ \delta_4 c_0'\equiv 0 \shortmod{q k_1 k_1^*}}} 
 \sum_{p_1 \geq 1} \frac{\nu_{\infty,f^*\vert_{\ell}}(p_1 \delta_4 c_0')}{p_1^{u_1} c_0'^{u_4}} \Big)
\Big(\sum_{\substack{ p_2 , p_3   \geq 0 }}  \frac{  \nu_{\frac{1}{\delta_5},f^*\vert_{\ell'}}(p_2  p_3 ) }{ p_2^{u_2} p_3^{u_3} } \Big)
\Big|.
\end{multline}
By Lemmas \ref{lemma:FourierExpansionDifferentCusps} and \ref{lemma:oldformDirichletSeriesBound}, 
we see that the sum over $p_2,p_3$ has analytic continuation to the desired region, and satisfies
\begin{equation*}
 \sum_{p_2, p_3 \geq 1} \frac{\nu_{\frac{1}{\delta_5},f^* \vert_{\ell'}}(p_2 p_3)}{p_2^\alpha p_3^\beta} \ll |\nu_{f^*}(1)| |L(f^*, \alpha) L(f^*, \beta)|,
\end{equation*}
uniformly in $\ell'$ and $\delta_5$.

The first product in \eqref{eq:ZjuFormulaReminder2} is a bit trickier.  
Recall from \eqref{eq:qandk1areCoprime} that $(q,k_1) = 1$ and from \eqref{eq:k1dividesdelta2} that $k_1 | \delta_2$.  
We apply \eqref{eq:Zbound}, with 
$d = D  \frac{\delta_4}{(D,\delta_4)}$, and $D = q k_1 k_1^*$.
This gives
\begin{equation*}
 \sum_{\substack{(c_0',fg_0 m_1') = 1\\ \delta_4 c_0'\equiv 0 \shortmod{qk_1 k_1^*} }} 
 \sum_{p_1 \geq 1} \frac{\nu_{f^* \vert_{\ell}}(p_1 \delta_4 c_0')}{p_1^{u_1} c_0'^{u_4}} 
 \ll 
 q^{\varepsilon} 
|\nu_{f^*}(1)|
 \Big(\frac{(D, \delta_4)}{D}\Big)^{1/2}
 \frac{(d, \ell)^{1/2} d_0^{\theta}  }{d_B^{1/2}}
   |L(f^*, u_1) L(f^*, u_4)|,
\end{equation*}
where $\frac{d}{\ell} = \frac{d_1}{\ell_1}$ is in lowest terms, and then we factor $d_1 = d_B d_0$ where $d_B | B^{\infty}$, and $(d_0, B) = 1$.

Using $|\nu_{f^*}(1)|^2 = (AB)^{-1} q^{o(1)}$ via \eqref{eq:nuNormalizationMaass}, 
we then have
\begin{multline*}
 \sum_{\substack{t_j \text{ level $\delta_2 \delta_5$} \\ |t_j| \leq T} } |Z_j({\bf u})| \ll q^{\varepsilon} 
 \Big(\frac{(D, \delta_4)}{D}\Big)^{1/2} \sum_{AB = \delta_2 \delta_5} \sum_{\ell | A} 
 (d, \ell)^{1/2} d_B^{-1/2} d_0^{\theta}
 \\
 \times
 \frac{1}{AB} 
 \sum_{\substack{f^* \text{ new, level $B$} \\ |t_{f^*}| \leq T}}
 |L(f^*, u_1) L(f^*, u_2) L(f^*, u_3) L(f^*, u_4)|.
\end{multline*}
A standard argument with the spectral large sieve (e.g., see \cite[Theorem 3.4]{motohashi1997spectral} for the level $1$ case) implies
\begin{equation}
\label{eq:ZjboundIntermediateStep}
\sum_{\substack{t_j \text{ level $\delta_2 \delta_5$} \\ |t_j| \leq T} } |Z_j({\bf u})| \ll_{{\bf u}, \varepsilon} T^{2+\varepsilon} q^{\varepsilon} \Big(\frac{(D, \delta_4)}{D}\Big)^{1/2} 
\sum_{AB = \delta_2 \delta_5} \frac{1}{A} 
  \sum_{\ell | A} (d, \ell)^{1/2}
  d_B^{-1/2} d_0^{\theta}
,
\end{equation}
where here and throughout the implied dependence on ${\bf u}$ is at most polynomial.

At this point, the proof of Lemma \ref{lemma:oldformDirichletSeriesBound} has reduced to elementary estimates with arithmetic functions.
The factorization $d = (d,\ell) d_B d_0$ is dependent on $\ell$, so it takes some work to estimate the sum over $\ell$.   
To this end, we also factor $d$ in an alternative way, independently of $\ell$, by $d = d' f g h$ where $(d', AB) = 1$, $f|A^{\infty}$, $(f,B) = 1$, $g|B^{\infty}$, $(g,A) = 1$, and $h|A^{\infty}$ and $h|B^{\infty}$.  
 Note that $(f,g) = (f,h) = (g,h) = 1$.
 Then writing the old variables in terms of these, we have 
 \begin{equation*}
  d = 
  \underbrace{(f,\ell) (h,\ell)}_{(d, \ell)} 
  \underbrace{g \frac{h}{(h,\ell)}}_{d_B}
  \underbrace{d' \frac{f}{(f,\ell)}}_{d_{0}}.
 \end{equation*}
  Inserting this into \eqref{eq:ZjboundIntermediateStep}, 
  and summing over $\ell | L$, we obtain
\begin{equation*}
 \sum_{\substack{t_j \text{ level $\delta_2 \delta_5$} \\ |t_j| \leq T} } |Z_j({\bf u})| 
 \ll_{{\bf u}, \varepsilon}
 T^{2 + \vep} q^{\varepsilon} d'^{\theta} \Big(\frac{(D, \delta_4)}{D}\Big)^{1/2}
  \sum_{AB = \delta_2 \delta_5} \frac{1}{A} 
 \frac{(f,A)^{1/2-\theta} f^{\theta} (h,A) }{\sqrt{gh}}.
\end{equation*}
Writing $d'f = \frac{D}{(D,\delta_4)} \delta_4 \frac{1}{gh}$, we obtain
\begin{equation*}
\sum_{\substack{t_j \text{level $\delta_2 \delta_5$} \\ |t_j| \leq T} } |Z_j({\bf u})| \ll_{{\bf u}, \varepsilon} T^{2 + \vep} q^{\varepsilon}\Big(\frac{(D, \delta_4)}{D}\Big)^{1/2-\theta} 
 \delta_4^{\theta} 
 \sum_{AB = \delta_2 \delta_5} \frac{1}{A} 
 \frac{(f,A)^{1/2-\theta}  (h,A)  }{(gh)^{1/2+\theta}}.
\end{equation*}

Now let us pause to gauge our progress towards \eqref{eq:ZjuBound}.  The inner sum over $A$ easily gives $O(q^{\varepsilon})$, and so if we trivially bound this part, and use $D = q k_1 k_1^*$ (also recall $(k_1, \delta_4) = 1$ from \eqref{eq:delta4iscoprimetok1}), we get the bound
\begin{equation*}
  \sum_{t_j} |Z_j({\bf u})| \ll_{{\bf u}, \varepsilon} q^{\theta-\half} \frac{1}{(k_1 k_1^*)^{1/2}} (q, \delta_4)^{1/2-\theta} 
 (k_1 k_1^*)^{\theta} \delta_4^{\theta} T^2 q^{\varepsilon} 
 .
\end{equation*}
so we need to save 
$\delta_4^{\frac12 + \theta} (k_1 k_1^*)^{\theta}$, 
which will come from better-estimating the sum over $A$.

This inner sum over $A,B$ may be factored into prime powers.  For the primes $p \nmid k_1 \delta_4$, all we use is that the local factor is $\leq 1$ (leading to a $O(q^{\varepsilon})$ bound from these primes, by the observation in the previous paragraph).  
Recall $\delta_4 | \delta_2 \delta_5$ since $\delta_5 = [\delta_3,\delta_4]$ (see \eqref{eq:delta4anddelta5definitions}), and $\delta_4 | d$, so $\delta_4 | fgh$.
For $p | \delta_4$,  say $p^{\nu} || \delta_4$, 
$p^f || f$, and so on for $g$, $h$, $A$, and $B$, by an abuse of notation. Now the variables in the exponents are written additively.  Since $\delta_4 | fgh$, in additive notation we have $f + g + h \geq \nu$.  Also, $A+B \geq \nu$, since $\delta_4 | \delta_2 \delta_5$.

In the case $B=0$, then $g = h = 0$ and  the local factor is
\begin{equation*}
 \frac{1}{p^{A}} (p^f,  p^A)^{\half-\theta} \leq p^{\nu(-\half - \theta)},
\end{equation*}
which is the local factor of $\delta_4^{-\half - \theta}$.  In the case $A=0$, the local factor is  no larger than the local factor of $\delta_4^{-\half-\theta}$ as can be seen easily.  Finally, if $A, B > 0$, then $f=g = 0$, $h \geq \nu$, and so the local factor equals
\begin{equation*}
 \frac{1}{p^A} \min(p^h, p^A) p^{-h(\half + \theta)} \leq p^{-\nu (\half + \theta)}.
\end{equation*}

Now suppose that $p|k_1$.  Then since $k_1 | \delta_2$ and $k_1 | d$ (whence $k_1 | fgh$) essentially the same proof used for $\delta_4$ shows the local factors for primes dividing $k_1$ give $O(k_1^{-1/2})$.

In summary, this shows
\begin{equation*}
 \sum_{\substack{t_j \text{ level $\delta_2 \delta_5$} \\ |t_j| \leq T} } |Z_j({\bf u})| \ll_{{\bf u}, \varepsilon} T^2 q^{\varepsilon}\Big(\frac{(q, \delta_4)}{q k_1 k_1^*}\Big)^{1/2-\theta} (k_1 \delta_4)^{-1/2} 
 .
\end{equation*}
Using $(k_1 k_1^*)^{\theta} \leq k_1^{2 \theta} \leq  k_1^{1/2}$  gives \eqref{eq:ZjuBound}, as desired.


\subsection{Continuous spectrum}

In this section we prove Lemma \ref{lemma:ZtcBound}.

\subsubsection{Fourier coefficients of Eisenstein Series}
\label{section:EisensteinFourierExpansion}
Here we quote an explicit evaluation of the Fourier coefficients of $\phi_{\a \c}(n,u)$, where $\a = 1/r$ is an Atkin-Lehner cusp, and $\c$ is an arbitrary cusp of $\Gamma_0(N)$. The proof appears in \cite{DoubleCosetPaper}.
Let $\c = v/f$ where $f|N$, $(v,f) = 1$, and $v$ runs modulo $(f,N/f)$; by \cite[Proposition 2.6]{IwaniecClassicalBook}, every cusp $\c$ may be represented in this form.  
Let
\begin{equation}
\label{eq:fN''formula}
 N' = \frac{N}{f}, \qquad N'' = \frac{N'}{(f,N')},
\end{equation}
and write
\begin{equation*}
 f_r = (f,r), \quad f_s = (f,s), \quad
 r = f_r r', \quad s= f_s s'.
\end{equation*}
In addition, write
\begin{equation*}
 f_r = f_r' f_0, \quad \text{where} \quad (f_0, r') = 1, \quad \text{and} \quad f_r' | (r')^{\infty},
\end{equation*}
and similarly 
\begin{equation*}
 s' = s_f' s_0, \quad \text{where} \quad (s_0, f_s) = 1, \quad \text{and} \quad s_f' | f_s^{\infty}.
\end{equation*}

\begin{myprop}
\label{prop:EisensteinFourierCoefficientFormulaDirichletCharacters}
 Let notation be as above.  Then $\phi_{\a \c}(n,u) = 0$ unless
 \begin{equation*}
n=\frac{f_r'}{(f_r',r')} \frac{s_f'}{(s_f',f_s)} k,
\end{equation*}
for some integer $k$.  In this case, write $k = k_r k_s \ell$, where 
\begin{equation*}
 k_r = (k, (f_r', r')), \qquad k_s = (k, (s_f', f_s)).
\end{equation*}
 Then
 \begin{multline}
\label{eq:phicuspsEisensteinFormulaWithDirichletCharacters}
\phi_{\a \c}(n,u) = \frac{S(\ell,0;s_0 f_0)}{(N'' s f_r^2)^u} 
\frac{f_r'}{(f_r', r')} 
\frac{s_f'}{(s_f', f_s)} 
\sum_{\substack{d|k \\ (d, f_s r') = 1}} d^{1-2u}
\frac{1}{\varphi(\frac{(f_r', r')}{k_r})}  
\frac{1}{\varphi(\frac{(s_f', f_s)}{k_s})}
\\
\sum_{\chi \shortmod{\frac{(f_r', r')}{k_r}}} 
  \sum_{\psi \shortmod{\frac{(s_f', f_s)}{k_s}}} 
  \frac{(\chi \psi)(\ell)
  \tau(\overline{\chi}) \tau(\overline{\psi})}{L(2u, \overline{\chi^2 \psi^2} \chi_0)}
  (\chi \psi)(\overline{s_0 f_0  d^2} v) 
  \chi(-k_s\overline{(s_f',f_s)})
  \psi(k_r \overline{(f_r',r')})
,
\end{multline}
where $\chi_0$ is the principal character modulo $f_s r'$.
\end{myprop}
For later calculations, it will be useful to notice that the condition $(\frac{k}{k_r}, \frac{(f_r', r')}{k_r}) = 1 $ (and similarly in the $s$-aspect) is automatic from the presence of $(\chi \psi)(\ell)$.
Moreover, we have $(f_r', r') = (f_r, \frac{r}{f_r})$, and similarly $(s_f', f_s) = (f_s, \frac{s}{f_s})$, and so $(f_r', r')(s_f', f_s) = (f, \frac{N}{f})$.  Note the condition $d|k$ together with $(d, f_s r') = 1$ implies $d | \ell$.

\subsubsection{Proof of Lemma \ref{lemma:ZtcBound}}
By \eqref{eq:rhodef} and \eqref{eq:nuDefinition},
we have
\begin{equation*}
\nu_{\a  \b}(n,u) = \alpha(u) \phi_{\a \b}(n,u) |n|^{u-\frac12},
\quad \text{where} \quad
	\alpha(u) = 
	\frac{2 \pi^{u+\frac12}}{\Gamma(u) (\cos(\pi(u-\frac12)))^{1/2}}.
\end{equation*}
Note that $|\alpha(1/2+it)| = 2 \sqrt{\pi}$ is independent of $t \in \mr$. 
Define
 \begin{equation*}
  Z_1 = Z_1(\alpha,\beta) = \sum_{m,n \geq 1} \frac{\nu_{1/r, \c}(mn,1/2+it)}{m^{\alpha} n^{\beta}},
 \end{equation*}
and
\begin{equation*}
Z_2 = Z_2(\alpha,\beta) = \sum_{\substack{m,n \geq 1 \\ \delta m \equiv 0 \shortmod{D} \\ (m,Q) = 1}} \frac{\nu_{\infty, \c}(\delta mn,1/2+it)}{m^{\alpha} n^{\beta}},
\end{equation*}
where we assume the level is $N$ as in Section \ref{section:EisensteinFourierExpansion}. This meaning of $N$ is valid only within the confines of this subsection,  and hence should not be confused with the dyadic variable $N$ in the rest of the article. For Lemma \ref{lemma:ZtcBound}, we shall need $N = \delta_2 \delta_5$, $Q = fg_0 m_1'$, $D = q k_1 k_1^*$, $\delta = \delta_4$, and $r=\delta_5$, but we do need make these specifications yet.
 With this notation, we have $Z_{\c, t}({\bf u}) = Z_1(u_2, u_3) Z_2(u_1, u_4)$.
 
The plan of the proof is to first derive bounds on $Z_1^0, Z_1^*, Z_2^0, Z_2^*$ individually, and follow this with estimates for the sums over $\c$.

Using Proposition \ref{prop:EisensteinFourierCoefficientFormulaDirichletCharacters}
, we have (with $u = \tfrac12 + it$)
\begin{multline*}
 Z_1 = 
 \frac{f_r'}{(f_r', r')} 
\frac{s_f'}{(s_f', f_s)} \frac{\alpha(u)}{(N'' s f_r^2)^u}
\sum_{k_r | (f_r', r')} \sum_{k_s | (s_f', f_s)}
\frac{1}{\varphi(\frac{(f_r', r')}{k_r})}  
\frac{1}{\varphi(\frac{(s_f', f_s)}{k_s})}
\sum_{\chi \shortmod{\frac{(f_r', r')}{k_r}}} 
  \sum_{\psi \shortmod{\frac{(s_f', f_s)}{k_s}}} 
\\
\frac{
  \tau(\overline{\chi}) \tau(\overline{\psi})}{L(2u, \overline{\chi^2 \psi^2} \chi_0)}
  (\chi \psi)(\overline{s_0 f_0 } w') 
  \chi(-k_s\overline{(s_f',f_s)})
  \psi(k_r \overline{(f_r',r')})
\\
\sum_{(d,f_s r') =1} (\overline{\chi \psi})(d^2)  d^{1-2u}
 \sum_{\substack{m , n \geq 1 \\ (*)}}  \frac{(\chi \psi)(\ell) S(\ell,0;s_0 f_0)}{m^{\alpha} n^{\beta}} (mn)^{u-\frac12}.
\end{multline*}
Here $(*)$ stands for the following conditions: We have $mn = \frac{f_r'}{(f_r',r')} \frac{s_f'}{(s_f',f_s)} k_r k_s \ell$, and also $\ell \equiv 0 \pmod{d}$.
Write $Z_1 = Z_1^{0} + Z_1^*$ where $Z_1^{0}$ corresponds to the part with both $\chi$ and $\psi$ principal.

By a trivial bound, we have for $\text{Re}(\alpha), \Re(\beta) \geq  \sigma > 1$,
\begin{equation}
\label{eq:Z10individualBound}
 |Z_1^{0}| \ll_\sigma
  \frac{N^{\varepsilon}}{f_r \sqrt{s N''} (f_r',r') (s_f',f_s)} 
\frac{1}{|\zeta(1+2it)|} \ll_\sigma \frac{(N(1+|t|))^{\varepsilon} }{(f, \frac{N}{f}) f_r \sqrt{s N''}} .
\end{equation}
Meanwhile, we have 
\begin{multline*}
 |Z_1^*| \ll 
 \frac{f_r'}{(f_r', r')} 
\frac{s_f'}{(s_f', f_s)} \frac{N^{\varepsilon}}{f_r \sqrt{sN''}}
\sum_{k_r | (f_r', r')} \sum_{k_s | (s_f', f_s)}
\frac{1}{\varphi(\frac{(f_r', r')}{k_r})}  
\frac{1}{\varphi(\frac{(s_f', f_s)}{k_s})}
\\
\sumprime_{\chi \shortmod{\frac{(f_r', r')}{k_r}}} 
  \sumprime_{\psi \shortmod{\frac{(s_f', f_s)}{k_s}}} 
\frac{
  |\tau(\overline{\chi}) \tau(\overline{\psi})|}{|L(1+2it, \overline{\chi^2 \psi^2} \chi_0)|} |Y_1|,
\end{multline*}
where the notation $\sum{}'$ means the principal character is omitted, and with
\begin{equation*}
 V = \frac{f_r'}{(f_r',r')} \frac{s_f'}{(s_f',f_s)} k_r k_s,
\end{equation*}
we let
\begin{equation*}
 Y_1 = \sum_{(d,f_s r') =1} (\overline{\chi \psi})(d^2)  d^{1-2u}
 \sum_{\substack{mn \equiv 0 \shortmod{dV}}}  \frac{(\chi \psi)(\frac{mn}{V}) S(\frac{mn}{V},0;s_0 f_0)}{m^{\alpha} n^{\beta}} (mn)^{u-\frac12}.
\end{equation*}
Moreover, the analytic continuation of $Z_1^*$ will be inherited from that of $Y_1$.

Similarly to \eqref{eq:tau3sum} and \eqref{eq:tau3sumLater}, 
one can show the formal identity
\begin{equation}
\label{eq:divisorCongruenceSumCombinatorialDecomposition}
 \sum_{mn \equiv 0 \shortmod{D}} f(m,n) = \sum_{CAB = D} \mu(C) \sum_{m,n} f(CAm, CBn).
\end{equation}
Applying this to $Y_1$ with $D = dV$, we obtain
\begin{equation*}
 Y_1 = V^{u-\frac12}
 \sum_{(d,f_s r') =1} (\overline{\chi \psi})(d)  d^{\frac12-u}
 \sum_{CAB = dV} \frac{\mu(C) (\chi \psi)(C) C^{u-\frac12}}{A^{\alpha} B^{\beta} C^{\alpha+\beta}}
 \sum_{\substack{m,n}}  \frac{(\chi \psi)(mn) S(Cdmn,0;s_0 f_0)}{m^{\alpha} n^{\beta} (mn)^{\frac12-u}}.
\end{equation*}
One can readily observe that $(d,V) = 1$, and $(V, s_0 f_0) = 1$.  In the factorization $CAB = dV$, one may split each of $C,A,B$ uniquely into its part dividing $d$ and dividing $V$ separately, and thereby factor the sum.  In this way, we obtain (with $\text{Re}(u) = \frac12$)
\begin{equation*}
 |Y_1| \ll \frac{N^{\varepsilon}}{V^{\sigma}} \Big|
 \sum_{(CAB,f_s r') =1}  
  \frac{\mu(C) (\overline{\chi \psi})(AB) }{A^{\alpha} B^{\beta} C^{\alpha+\beta} (AB)^{u-\frac12}}
 \sum_{\substack{m,n}}  \frac{(\chi \psi)(mn) S(C^2 ABmn,0;s_0 f_0)}{m^{\alpha} n^{\beta} (mn)^{\frac12-u}} \Big|.
\end{equation*}
Next we open the Ramanujan sum as a divisor sum, giving
\begin{equation*}
 |Y_1| \ll \frac{N^{\varepsilon}}{V^{\sigma}}
 \sum_{g| s_0 f_0} g
 \Big|
 \sum_{(CAB,f_s r') =1}  \sum_{\substack{m,n \\ C^2 AB mn \equiv 0 \shortmod{g}}}
  \frac{\mu(C) (\overline{\chi \psi})(AB) }{A^{\alpha} B^{\beta} C^{\alpha+\beta} (AB)^{u-\frac12}}
   \frac{(\chi \psi)(mn)}{m^{\alpha} n^{\beta} (mn)^{\frac12-u}} \Big|.
\end{equation*}
Now it is not difficult to see the analytic continuation of $Y_1$ to 
$\text{Re}(\alpha, \beta) \geq \sigma = 1/2 + \varepsilon$, and therein we obtain the bound 
\begin{equation*}
 |Y_1| \ll N^{\varepsilon}
 \Big(\frac{(f_r',r')}{f_r'} \frac{(s_f',f_s)}{s_f'}\Big)^{1/2} \frac{(s_0 f_0)^{1/2}}{(k_r k_s)^{1/2}}
  \\
  |L(\alpha-it, \chi \psi) L(\alpha+it, \overline{\chi \psi})
  L(\beta-it, \chi \psi) L(\beta+it, \overline{\chi \psi})|.
\end{equation*}

%
%
We recall the well-known bound on the fourth moment of Dirichlet $L$-functions, namely
\begin{equation}
\label{eq:dirichletfourthmoment}
 \sum_{\chi \shortmod{N}} |L(\sigma + it, \chi)|^4 \ll_{\sigma,\varepsilon} (1+|t|)^{1+\varepsilon} N^{1+\varepsilon},
\end{equation}
for $\sigma \geq 1/2$.  Moreover, we have the hybrid version
\begin{equation}
\label{eq:dirichletfourthmomenthybrid}
\int_{|t| \leq T} \sum_{\chi \shortmod{N}} |L(\sigma + it, \chi)|^4 \ll_{\sigma,\varepsilon} (1+T)^{1+\varepsilon} N^{1+\varepsilon}.
\end{equation}
For references, consult \cite[Chapter 10]{Montgomery} or \cite[Theorem 2]{Gallagher}; the statements in these sources do not precisely claim these bounds, but the methods can be easily modified.
In addition, we have
\begin{equation*}
\frac{1}{|L(1+2it, \chi)|} \ll (1 + |t|)^{\varepsilon} N^{\varepsilon},
\end{equation*}
for which see \cite[Theorem 11.4]{MontgomeryVaughan}.
Thus applying H\"older's inequality, and the bound \eqref{eq:dirichletfourthmoment} we obtain
\begin{equation}
\label{eq:Z1*bound}
 |Z_1^*| \ll_{\alpha,\beta, \varepsilon}
  \frac{N^{\varepsilon} (1+|t|)^{1+\varepsilon}}{f_r \sqrt{sN''}}
(f_r
s')^{1/2} 
  = \frac{N^{\varepsilon}(1+|t|)^{1+\varepsilon}}{\sqrt{f N''}}
  .
\end{equation}
Using \eqref{eq:dirichletfourthmomenthybrid}, 
we alternatively have
\begin{equation}
\label{eq:Z1*boundIntegrated}
\int_{|t| \leq T} |Z_1^*| \ll_{\alpha,\beta,\varepsilon} \frac{q^{\varepsilon} T^{1+\varepsilon}}{\sqrt{f N''}} 
.
\end{equation}

Next we study $Z_2$, which is more difficult than $Z_1$.  We have $\a = \infty \sim 1/N$, so $r=N$ and $s=1$, and also $f_r = f$, $r' = N'$.  First we perform a minor simplification by writing the congruence $\delta m \equiv 0 \pmod{D}$ as $m \equiv 0 \pmod{\frac{D}{(\delta, D)}}$ (so necessarily $(Q, \frac{D}{(\delta, D)}) = 1$ otherwise the sum is empty).  Then we have
\begin{equation*}
Z_2 = \Big(\frac{(D,\delta)}{D}\Big)^{\alpha} 
Y_2, \quad \text{where} \quad Y_2 :=
\sum_{\substack{m, n \geq 1 \\ (m, Q) = 1}} \frac{\nu_{\infty, \c}(a mn , \frac12 + it) }{m^{\alpha} n^{\beta}},
\end{equation*}
and where
\begin{equation*}
a = \frac{\delta D}{(\delta, D)} = [\delta, D].  
\end{equation*}

Applying Proposition \ref{prop:EisensteinFourierCoefficientFormulaDirichletCharacters}, we obtain
\begin{multline*}
 Y_2 = 
 \frac{f_N'}{(f_N', N')} 
 \frac{\alpha(u)}{(N'' f^2)^u}
\sum_{k_N | (f_N', N')} 
\frac{1}{\varphi(\frac{(f_N', N')}{k_N})}  
\sum_{\chi \shortmod{\frac{(f_N', N')}{k_N}}} 
\frac{
  \tau(\overline{\chi}) \chi (-\overline{ f_0 } w') 
  }{L(2u, \overline{\chi^2 } \chi_0)}
\\
\sum_{(d,N') =1} d^{1-2u} \overline{\chi}(d^2)
 \sum_{\substack{ (*)}}  \frac{\chi(\ell) S(\ell,0; f_0)}{m^{\alpha} n^{\beta}} (a m n)^{u-\frac12}.
\end{multline*}
Now the symbol $(*)$ stands for the following conditions.  We have $amn = \frac{f_N'}{(f_N', N')} k_N \ell$,  $amn \equiv 0 \pmod{d}$, and also the condition $(m,Q) = 1$.  From the condition $(d, N') = 1$, we equivalently obtain that $\ell \equiv 0 \pmod{d}$.
Define
\begin{equation}
\label{eq:bdef}
b = \frac{f_N'}{(f_N', N')} k_N = \frac{f_N'}{\frac{(f_N', N')}{k_N}},
\end{equation}
and write 
\begin{equation}
\label{eq:abdefs}
a = (a,b) a', \qquad \text{and} \qquad b = (a,b) b',
\end{equation}
so that the condition $b | amn$ is equivalent to $b' | mn$.  Then $\ell = a' \frac{mn}{b'}$, and we have
\begin{equation*}
 Y_2 = 
 \frac{f_N'}{(f_N', N')} 
 \frac{\alpha(u) a^{u-\frac12}}{(N'' f^2)^u}
\sum_{k_N | (f_N', N')} 
\frac{1}{\varphi(\frac{(f_N', N')}{k_N})}  
\sum_{\chi \shortmod{\frac{(f_N', N')}{k_N}}} 
\frac{
  \tau(\overline{\chi}) \chi (\overline{ f_0  } w' a') 
  }{L(2u, \overline{\chi^2 } \chi_0)} X_2,
\end{equation*}
where
\begin{equation*}
 X_2 := \sum_{(d,N') =1} d^{1-2u} \overline{\chi}(d^2)
 \sum_{\substack{mn \equiv 0 \shortmod{b'} \\ amn \equiv 0 \shortmod{d} }}  \frac{\chi(\frac{mn}{b'}) S(a' \frac{mn}{b'},0; f_0)}{m^{\alpha} n^{\beta}} (m n)^{u-\frac12}.
\end{equation*}
Here $(b', f_0 d) = 1$, since $b | (N')^{\infty}$.  Opening the Ramanujan sum, we have
\begin{equation*}
 X_2 = 
 \sum_{e | f_0} e \mu(f_0/e)
 \sum_{(d,N') =1} d^{1-2u} \overline{\chi}(d^2)
 \sum_{\substack{mn \equiv 0 \shortmod{b'} \\ amn \equiv 0 \shortmod{d} \\ a' mn \equiv 0 \shortmod{e} }}  \frac{\chi(\frac{mn}{b'}) }{m^{\alpha} n^{\beta}} (m n)^{u-\frac12}.
\end{equation*}
Let $g = (a,d)$, so that
\begin{equation*}
 X_2 = 
 \sum_{\substack{g | a \\ (g,N') = 1}} g^{1-2u} \overline{\chi}(g^2)
 \sum_{e | f_0} e \mu(f_0/e)
 \sum_{(d,N'\frac{a}{g}) =1} d^{1-2u} \overline{\chi}(d^2)
 \sum_{\substack{mn \equiv 0 \shortmod{b'} \\ mn \equiv 0 \shortmod{d} \\ a' mn \equiv 0 \shortmod{e} }}  \frac{\chi(\frac{mn}{b'}) }{m^{\alpha} n^{\beta}} (m n)^{u-\frac12}.
\end{equation*}

Applying \eqref{eq:divisorCongruenceSumCombinatorialDecomposition} to $X_2$ with the modulus $d$, we obtain
\begin{multline*}
 X_2 = 
 \sum_{\substack{g | a \\ (g,N') = 1}} g^{1-2u} \overline{\chi}(g^2)
 \sum_{e | f_0} e \mu(f_0/e)
 \sum_{(d,N'\frac{a}{g}) =1} d^{1-2u} \overline{\chi}(d^2)
 \sum_{CAB = d} \mu(C)
 \\
 \sum_{\substack{C d mn \equiv 0 \shortmod{b'} \\  a' C d mn \equiv 0 \shortmod{e} }}  \frac{\chi(\frac{C dmn}{b'}) }{(CAm)^{\alpha} (CBn)^{\beta}} (Cdmn)^{u-\frac12}.
\end{multline*}
Since $C | d$, $(d,N') = 1$, and $b' | (N')^{\infty}$, the congruence $Cdmn \equiv 0 \pmod{b'}$ is equivalent to $mn \equiv 0 \pmod{b'}$.  We can then write $X_2$ as
\begin{multline*}
 X_2 = 
 \sum_{\substack{g | a \\ (g,N') = 1}} g^{1-2u} \overline{\chi}(g^2)
 \sum_{e | f_0} e \mu(f_0/e)
 \sum_{(CAB,N'\frac{a}{g}) =1} 
 \frac{\mu(C) \overline{\chi}( AB) ( AB)^{\frac12 - u}}{C^{\alpha+\beta} A^{\alpha} B^{\beta}}
 \\
 \sum_{\substack{mn \equiv 0 \shortmod{b'} \\  a' C^2 AB mn \equiv 0 \shortmod{e} }}  \frac{\chi(\frac{mn}{b'}) }{m^{\alpha} n^{\beta}} (  mn)^{u-\frac12}.
\end{multline*}
Next we use \eqref{eq:divisorCongruenceSumCombinatorialDecomposition} again, this time on the congruence modulo $b'$, giving
\begin{multline*}
 X_2 = (b')^{u-\frac12}
 \sum_{\substack{g | a \\ (g,N') = 1}} g^{1-2u} \overline{\chi}(g^2)
 \sum_{e | f_0} e \mu(f_0/e)
 \sum_{xyz = b'} \frac{\mu(x) \chi(x) x^{u-\frac12}}{x^{\alpha+\beta} y^{\alpha} z^{\beta}}
 \\
 \sum_{(CAB,N'\frac{a}{g}) =1} 
 \frac{\mu(C) \overline{\chi}(AB) (AB)^{\frac12 -u }}{C^{\alpha+\beta} A^{\alpha} B^{\beta}}
 \sum_{\substack{ a' C^2 AB x b' mn \equiv 0 \shortmod{e} }}  \frac{\chi(mn) }{m^{\alpha} n^{\beta}} (mn)^{u-\frac12}.
\end{multline*}

Similarly to the $Z_1$ case, one can see the meromorphic continuation with a pole only in the case $\chi$ is principal.  In addition, we have the bound (with $u = \frac12 + it$)
\begin{equation*}
|X_2| \ll_{\sigma}  \frac{N^{\varepsilon}}{(b')^{\sigma}}  |L(\alpha - it, \chi) L(\alpha + it, \overline{\chi})
L(\beta -it, \chi) L(\beta + it, \overline{\chi})|
 \sum_{e|f_0} e \Big(\frac{(a',e)}{e}\Big)^{\sigma}.
\end{equation*}
Note
\begin{equation*}
 \sum_{e|f_0} e \Big(\frac{(a',e)}{e}\Big)^{\sigma} \ll N^{\varepsilon} (a',f_0)^{\sigma} (1 + f_0)^{1-\sigma}.
\end{equation*}


Now write $Z_2 = Z_2^0 + Z_2^*$ where $Z_2^0$ corresponds to the principal characters, and similarly write $Y_2 = Y_2^0 + Y_2^*$.
For $\text{Re}(\alpha, \beta) \geq \sigma  = 1+ \varepsilon$, we have trivially
\begin{equation*}
 X_2 \ll N^{\varepsilon} \frac{(a', f_0)}{b'} 
 = N^{\varepsilon} \frac{(a,f_0 b)}{b},
\end{equation*}
recalling \eqref{eq:abdefs}.
Thus
\begin{equation*}
Y_2^0 \ll \frac{f_N'}{(f_N', N')} \frac{N^{\varepsilon}}{\sqrt{N'' f^2}} \sum_{k_N | (f_N', N')} \frac{k_N}{(f_N', N')} \frac{(f_N', N')(a,f_0 b)}{f_N' k_N}.
\end{equation*}
Here $b$ is a function of $k_N$ (cf. \eqref{eq:bdef}), and is maximal when $k_N= (f_N', N')$ in which case $b = f_N'$.
Recalling $f_0 f_N' = f$, which implies $(a,f_0b) \leq (a, f)$, and using $(f_N', N') = (f, \frac{N}{f})$,
in all we obtain
\begin{equation*}
Y_2^0 \ll  \frac{N^{\varepsilon} (a,f)}{(f, \frac{N}{f}) \sqrt{N'' f^2}} \frac{1}{|\zeta(1+2it)|}.
\end{equation*}
Finally, we obtain
\begin{equation}
\label{eq:Z20bound}
Z_2^0 \ll \frac{N^{\varepsilon} (1+|t|)^{\varepsilon} }{(f, \frac{N}{f}) f \sqrt{N''}} ([\delta, D],f) \frac{(\delta,D)}{D}.
\end{equation}
Using \eqref{eq:dirichletfourthmoment} and H\"older's inequality, we have with $\sigma = 1/2 + \varepsilon$
\begin{equation*}
 |Y_2^*| \ll_{\alpha,\beta,\varepsilon}
 \frac{f_N'}{(f_N', N')} 
 \frac{N^{\varepsilon} (1+|t|)^{1+\varepsilon} }{(N'' f^2)^{1/2}}
\sum_{k_N | (f_N', N')} 
\Big(\frac{(f_N', N')}{k_N} \Big)^{1/2}
\frac{ (a',f_0)^{1/2} f_0^{1/2}}{b'^{1/2}}.
\end{equation*}
Note the simplification
\begin{equation*}
 \frac{f_N'}{(f_N', N')} \Big(\frac{(f_N', N')}{k_N} \Big)^{1/2}
 \frac{ f_0^{1/2}}{b'^{1/2}} = f^{1/2} \frac{(a,b)^{1/2}}{k_N}.
\end{equation*}
We have $(a', f_0) = (a, f_0)$ since $(b,f_0) = 1$, and $\frac{(a,b)^{1/2}}{k_N} \leq (a, f_N')^{1/2}$.
Thus
\begin{equation*}
 |Y_2^*| \ll_{\alpha,\beta,\varepsilon} 
 \frac{N^{\varepsilon} (1+|t|)^{1+\varepsilon} }{(N'' f)^{1/2}} 
 (a,f)^{1/2}.
\end{equation*}
Hence
\begin{equation}
\label{eq:Z2*bound}
 |Z_2^*| \ll_{\alpha,\beta,\varepsilon} \Big(\frac{(\delta, D)}{ D}\Big)^{1/2} 
  \frac{N^{\varepsilon} (1+|t|)^{1+\varepsilon} }{(N'' f)^{1/2}} 
 \Big(\frac{\delta D}{(\delta, D)},f \Big)^{1/2}.
\end{equation}
We also have in a similar way to the $Z_1^*$ case
\begin{equation}
\label{eq:Z2*boundIntegrated}
\int_{|t| \leq T} |Z_2^*| dt \ll_{\alpha,\beta,\varepsilon} \Big(\frac{(\delta, D)}{ D}\Big)^{1/2} 
  \frac{N^{\varepsilon} T^{1+\varepsilon} }{(N'' f)^{1/2}} 
 \Big(\frac{\delta D}{(\delta, D)},f \Big)^{1/2}.
\end{equation}

Now we proceed to prove the desired bounds in Lemma \ref{lemma:ZtcBound}.  The cusps may be parameterized by $\frac{v}{f}$ with $f|N$ and $v \pmod{(f, \frac{N}{f})}$ (with $v$ coprime to $f$).  
During the course of the proof, it will be helpful to refer to the following divisor-sum bounds:
\begin{equation}
\label{eq:dsumBound0}
\sum_{d|N} \frac{(d, \frac{N}{d})^2}{N} \ll N^{\varepsilon} \frac{\flooroot(N)}{\sqrt{N}}, \quad
\text{and}
\quad
\sum_{d | N} \frac{(d,\frac{N}{d})^2 d^{1/2}}{N} \ll N^{\varepsilon} \frac{\flooroot(N)^{3/2}}{\sqrt{N}}, 
\end{equation}
which in turn can be checked prime-by-prime by multiplicativity. If desired, the former inequality could be bounded by $N^{\varepsilon} \frac{\flooroot(N)^2}{N}$.
Along the same lines, we note 
\begin{equation}
\label{eq:dsumBound3/4}
\sum_{d | N} \frac{(d,\frac{N}{d})^2 }{d^{1/2} N}  \ll N^{\varepsilon} \frac{\flooroot(N)^{3/2}}{N},
\end{equation}
as well as
\begin{equation}
 \label{eq:dsumBound1/2}
 \sum_{d|N} \frac{(d, \frac{N}{d}) d^{1/2}}{N} \ll N^{\varepsilon} \frac{\sqrt{\flooroot(N)}}{\sqrt{N}},
 \quad
 \text{and}
 \quad
 \sum_{d|N} \frac{(d, \frac{N}{d})}{N} \ll N^{\varepsilon} \frac{\flooroot(N)}{N}.
\end{equation}

Combining \eqref{eq:Z1*boundIntegrated} and \eqref{eq:Z2*bound}, we obtain
\begin{equation*}
\int_{|t| \leq T} \sum_{\c} |Z_1^* Z_2^*| dt \ll_{{\bf u}, \varepsilon} 
N^{\varepsilon} T^{2+\varepsilon} \Big(\frac{(\delta, D)}{ D}\Big)^{1/2} \sum_{f|N} \frac{(f,\frac{N}{f})^2}{N} \Big(\frac{\delta D}{(\delta, D)},f\Big)^{1/2}.
\end{equation*}

We have $N = \delta_2 \delta_5$, $D = q k_1 k_1^*$, $\delta = \delta_4|\delta_5$.  Then with these substitutions, and recalling $(\delta_4, k_1) = 1$ we obtain
\begin{equation*}
\int_{|t| \leq T} \sum_{\c} |Z_1^* Z_2^*| dt \ll_{{\bf u}, \varepsilon}  q^{\varepsilon} T^{2+\varepsilon} \Big(\frac{(\delta_4, q)}{ q k_1 k_1^*}\Big)^{1/2}
\sum_{f| \delta_2 \delta_5}  \frac{(f,\frac{\delta_2 \delta_5}{f})^2}{\delta_2 \delta_5} \Big(k_1 k_1^* \frac{\delta_4 q}{(\delta_4, q )},f\Big)^{1/2}.
\end{equation*}
By multiplicativity, and using $(\delta_2, \delta_5) = 1 = (\delta_2, q)$, and $k_1 | \delta_2$, the inner sum over $f$ factors and simplifies as
\begin{equation*}
 \sum_{f| \delta_2 \delta_5}  \frac{(f,\frac{\delta_2 \delta_5}{f})^2}{\delta_2 \delta_5} \Big(k_1 k_1^* \frac{\delta_4 q}{(\delta_4, q )},f\Big)^{1/2} 
 = 
 \Big(\sum_{g| \delta_2 }  \frac{(g,\frac{\delta_2 }{g})^2 (g, k_1 k_1^*)^{1/2}}{\delta_2 }  \Big) \Big( \sum_{h | \delta_5} \frac{(h,\frac{\delta_5}{h})^2}{\delta_5} \Big(\frac{\delta_4 q}{(\delta_4, q )},h \Big)^{1/2} \Big).
\end{equation*}
Using $(g, k_1 k_1^*)^{1/2} \leq \sqrt{k_1 k_1^*}$, $(\frac{\delta_4 q}{(\delta_4, q)}, h)^{1/2} \leq h^{1/2}$, and \eqref{eq:dsumBound0},  we obtain
\begin{equation}
\label{eq:Z1*Z2*nearlyendofproof}
 \int_{|t| \leq T} \sum_{\c} |Z_1^* Z_2^*| dt \ll_{{\bf u}, \varepsilon}  
N^{\varepsilon} T^{2+\varepsilon} \Big(\frac{(\delta_4, q)}{ q k_1 k_1^*}\Big)^{1/2} \frac{\flooroot(\delta_2) \flooroot(\delta_5)^{3/2}}{\sqrt{\delta_2 \delta_5}} \sqrt{k_1 k_1^*}.
\end{equation}
Recall that $\delta_5 = [\delta_3, \delta_4]$, and that $\delta_4$ is square-free.  Therefore, $[\delta_3, \delta_4] = \delta_3 \frac{\delta_4}{(\delta_3, \delta_4)}$ where $(\delta_3, \frac{\delta_4}{(\delta_3, \delta_4)}) = 1$.  Since $\flooroot$ is multiplicative, and trivial on square-free numbers, this implies
\begin{equation}
\label{eq:floorootdelta5equalsdelta3}
 \flooroot(\delta_5) = \flooroot(\delta_3) \flooroot\Big(\frac{\delta_4}{(\delta_3, \delta_4)}\Big) = \flooroot(\delta_3),
\end{equation}
a simplification we will make repeatedly below.  Applying \eqref{eq:floorootdelta5equalsdelta3} to \eqref{eq:Z1*Z2*nearlyendofproof} gives \eqref{eq:Z1*Z2*bound}.

Combining \eqref{eq:Z10individualBound} and \eqref{eq:Z20bound} and specializing the variables, we obtain
\begin{equation*}
 \sum_{\mathfrak{c}} |Z_1^0 Z_2^0| \ll (q(1+|t|))^{\varepsilon}
 \Big(\frac{(\delta_4, q)}{qk_1 k_1^*}\Big)
 \sum_{\mathfrak{c}} \frac{(f,[\delta_4,qk_1 k_1^*])}{(f, \frac{N}{f})^2 N'' s^{1/2} f_r f}.
\end{equation*}
Using \eqref{eq:fN''formula} and summing over $u \pmod{(f, \frac{N}{f})}$, we obtain
\begin{equation*}
 \sum_{\mathfrak{c}} \frac{(f,[\delta_4,qk_1 k_1^*])}{(f, \frac{N}{f})^2 N'' s^{1/2} f_r f} 
 \leq 
 \frac{1}{N} \sum_{f|N} \frac{ (f, [\delta_4, qk_1 k_1^*])}{s^{1/2} f_r}.
\end{equation*}
Using the coprimality conditions, we have $[\delta_4, qk_1 k_1^*] = k_1 k_1^*[\delta_4, q]$, which in turn divides $k_1 k_1^*[\delta_5, q]$.  Now $k_1$ is in the $s$-part of the level (since $s=\delta_2$ and $k_1 | \delta_2$), while we also have $(q, \delta_2) = 1$ by \eqref{eq:delta2Coprime}  so that $[\delta_4, q]$ is coprime to $s$, and hence $f_s$.  Now we may see that the sum above factors as
\begin{equation*}
 \frac{1}{N} \sum_{f|N} \frac{ (f, [\delta_4, qk_1 k_1^*])}{s^{1/2} f_r} \leq  
 \Big(\sum_{f_r|r} \frac{  (f_r, [\delta_4, q])}{r f_r}
 \Big)
\Big( 
  \sum_{f_s | s} \frac{(f_s, k_1 k_1^*) }{s^{3/2}}
\Big)  
  .
\end{equation*}
Using $(f_r, [\delta_4, q]) \leq f_r$, $(f_s, k_1 k_1^*) \leq f_s^{1/2} \sqrt{k_1 k_1^*} \leq \sqrt{s k_1 k_1^*}$
 leads immediately to \eqref{eq:Z10Z20bound}.

Finally, we examine the two cross terms.
From \eqref{eq:Z1*boundIntegrated} and \eqref{eq:Z20bound}, and using simplifications as in the above cases,
we have
\begin{equation*}
 \int_{|t| \leq T} \sum_{\c} |Z_1^* Z_2^0| dt \ll_{{\bf u}, \varepsilon}
 \Big(\frac{(\delta_4, q)}{qk_1 k_1^*} \Big) 
 q^{\varepsilon} T^{1+\varepsilon}
\sum_{f|\delta_2 \delta_5} \frac{(f,\frac{\delta_2 \delta_5}{f})}{\delta_2 \delta_5 \sqrt{f}} (f, k_1 k_1^*[\delta_4, q]).
\end{equation*}
The inner sum factors as
\begin{equation*}
 \Big(\sum_{g| \delta_2} \frac{(g,\frac{\delta_2 }{g})}{\delta_2  \sqrt{g}} (g, k_1 k_1^*) \Big)
 \Big(\sum_{h | \delta_5} \frac{(h,\frac{\delta_5}{h})}{\delta_5 \sqrt{h}} (h, [\delta_4, q]) \Big). 
\end{equation*}
Using $(g, k_1 k_1^*) \leq \sqrt{g k_1 k_1^*}$, $(h, [\delta_4, q]) \leq h$, and \eqref{eq:dsumBound1/2},
then
\begin{equation*}
 \sum_{f|\delta_2 \delta_5} \frac{(f,\frac{\delta_2 \delta_5}{f})}{\delta_2 \delta_5 \sqrt{f}} (f, k_1 k_1^*[\delta_4, q])
 \ll q^{\varepsilon} \frac{\sqrt{\flooroot(\delta_5)} \flooroot(\delta_2)}{\delta_2 \sqrt{\delta_5}} \sqrt{k_1 k_1^*}.
\end{equation*}
Using \eqref{eq:floorootdelta5equalsdelta3}, 
\eqref{eq:Z1*Z20bound} follows.

Similarly, combining \eqref{eq:Z10individualBound} and \eqref{eq:Z2*bound},  we have
\begin{equation*}
 \int_{|t| \leq T} \sum_{\c} |Z_1^0 Z_2^*| dt \ll_{{\bf u}, \varepsilon} 
 \Big(\frac{(\delta_4, q)}{qk_1 k_1^*} \Big)^{1/2} 
 q^{\varepsilon} T^{1+\varepsilon}
\sum_{f|N}
\frac{f (f, \frac{N}{f})(k_1 k_1^* \frac{\delta_4 q}{(\delta_4, q)}, f)^{1/2}}{f_r N \sqrt{sf}}.
\end{equation*}
Following the discussion of the $Z_1^0 Z_2^0$ case (recall $r = \delta_5$ and $s= \delta_2$), the inner sum over $f$ factors as
\begin{equation*}
\Big( 
\sum_{g|\delta_2}
\frac{ (g, \frac{\delta_2}{g})(k_1 k_1^*, g)^{1/2} g^{1/2}}{ \delta_2^{3/2}}
\Big)
\Big(
\sum_{h|\delta_5}
\frac{ (h, \frac{\delta_5}{h})(\frac{\delta_4 q}{(\delta_4, q)}, h)^{1/2}}{  \delta_5 \sqrt{h}}
\Big) 
\end{equation*}
Using \eqref{eq:dsumBound1/2}  and \eqref{eq:floorootdelta5equalsdelta3}, we have that this is
\begin{equation*}
\ll  q^{\varepsilon} \sqrt{k_1 k_1^*} \frac{\sqrt{\flooroot(\delta_2)} \flooroot(\delta_3)}{\delta_2 \delta_5}.
\end{equation*}
Hence we obtain \eqref{eq:Z10Z2*bound}, (in fact, with a slightly better power of $\flooroot(\delta_2)$).

  \section{Zero terms}
  \label{section:ZeroTerms}
  \subsection{Overview}
 In this section, we analyze the contribution to $\mathcal{S}$ from the terms with some $p_i = 0$.  Recall the original expression for $\mathcal{S}'''$ from \eqref{eq:S'''prePoisson}, and Proposition \ref{thm:TriplePoisson}.
 
 Let us write 
\begin{equation*} 
 \mathcal{S}''' = (\sum_{P} \mathcal{T}_{P} ) + \mathcal{S}_{0,0,0}''' + \mathcal{S}_{0,0}''' + \mathcal{S}_{0}''',
\end{equation*} 
where $\sum_{P} \mathcal{T}_{P}$ corresponds to the terms with all $p_i \neq 0$, $\mathcal{S}_{0,0,0}'''$ corresponds to the terms with all three $p_i = 0$, $\mathcal{S}_{0,0}'''$ corresponds to the terms with exactly two $p_i=0$, and finally $\mathcal{S}_{0}'''$ has the terms with exactly one $p_i = 0$.  Recall that the sum over $P$ is the sum over the dyadic partitions of unity.  The partition is mainly beneficial for estimating $\mathcal{T}_P$, and we usually wish to remove the partition as much as possible when estimating the zero terms.

Applying the additional summations that led from $\mathcal{S}$ to $\mathcal{S}'''$ (see \eqref{eq:S''intoS'''}, \eqref{eq:S''andS'}, \eqref{eq:SintoS'} or alternatively \eqref{eq:Sprimes} and \eqref{eq:S0Definition} below), we likewise define 
$\mathcal{S}_{0,0,0}$, $\mathcal{S}_{0,0}$, and $\mathcal{S}_0$.  
Implicit in the definition of these quantities is that prior to the definition of $\mathcal{S}'''$, we applied a partition of unity.  When it is necessary to emphasize this, we may write $\mathcal{S}_{0,0,0}^{(T)}$ where $T$ stands for the tuple $(M_1, M_2, C, N_1, N_2, N_3, K)$, and likewise for $\mathcal{S}_{0,0}$ and $\mathcal{S}_{0}$.  Then $\sum_T \mathcal{S}_{0,0,0}^{(T)}$ represents the quantity after re-assembling the partition.  For ease of notation we may on occasion drop the superscript $T$.

 Our primary goal is to show
 \begin{mytheo}
 \label{thm:S000bound}
  With an appropriate choice of $G_i(s)$ in the approximate functional equations, we have
  \begin{equation*}
  \sum_T \mathcal{S}_{0,0,0}^{(T)} \ll q^{\varepsilon}.
  \end{equation*}
 \end{mytheo}
We will show the same bounds for $\mathcal{S}_{0,0}$ and $\mathcal{S}_{0}$. 
We make extensive use of the assumption 
        \begin{equation} \label{GhalfIsZero}
                G_i(1/2) = 0.
        \end{equation}
 
Next we specialize Lemma \ref{lem:Pnonzero} to degenerate $p_i$.
        \begin{mylemma}\label{lem:Pzero}
        Let $(\alpha,k) = 1$. If some $p_i$ is zero, then $A(p_1,p_2,p_3;\alpha;k)$ does not depend on $\alpha$. Furthermore, we have
        \begin{equation} \label{eq:AsumZeros}
               \frac{1}{k} A(0,0,0;\alpha;k) =  (\operatorname{Id} * \phi) (k),
        \end{equation}              
        where $*$ indicates Dirichlet convolution, $\operatorname{Id}(n) = n$, and $\phi$ is Euler's $\phi$-function. 

\end{mylemma}
This is a short calculation, so we omit the proof.

 
        \subsection{The case with all $p_i = 0$. }
                The case with $p_1 = p_2 = p_3 = 0$ is surprisingly delicate.
        It turns out that trivially bounding these terms leads only to  $\mathcal{S}_{0,0,0} \ll q^{\qtr+\varepsilon}$.
                Therefore, we have to make use of some further cancellation.  

For notational simplicity, let us write $A(0,0,0;*;k_0') = A(k_0')$ (it is independent of $*$)  and $B(0,0,0;k_0') = B(k_0')$.
We then have
        \begin{equation} \label{eq:STriplePrimeDefinition}
                \mathcal{S}'''_{0,0,0} = 
                \sum_{\substack{(c_0,g_0m_1') = 1\\ c_0\equiv 0 \shortmod{q k_1 k_1^*}} }
                \frac{1}{c_0} 
                \sum_{\substack{(k_0',\delta_2 c_0)=1 \\ k_0' \equiv 0 \shortmod{\frac{c_2}{(g_0,c_2)}}  }}
                \frac{1}{k_0'^3} A(k_0') B(k_0').  
        \end{equation}
The function $B$ depends on a choice of a partition of unity in the $n_1,n_2,n_3$ variables (as well as $c$, $k$, $m_1, m_2$, but here the focus is on the $n_i$). 
Our next goal is to recombine the partitions of 
        unity in the dyadic numbers $N_1,N_2,N_3$.

        \subsection{Recombining partitions of unity}
        \label{section:RecombiningPartitionsOfUnity}
                We write the weight function explicitly. Say 
                \[
                	J(n_1,n_2,n_3) = J_*(n_1n_2n_3,a,m_1',c_0, g_0k_0', c_2,k_1) 
                              F_a\left(\frac{n_1}{\sqrt{q}} , \frac{n_2}{\sqrt{q}}, \frac{n_3}{\sqrt{q}}\right) \frac{ \omega\left(\frac{n_1}{N_1},\frac{n_2}{N_2} , \frac{n_3}{N_3}\right)}{\sqrt{n_1n_2n_3}} 
                ,
                \] 
                where $\omega(t_1, t_2, t_3) = \omega(t_1) \omega(t_2) \omega(t_3)$ (recall $\omega$ gave rise to the dyadic partition of unity), and
                \[
                	J_*(n, \cdot) = e\left(-\frac{nam_1}{ck}\right) \int_{0}^\infty
	e\Big(\frac{-k t}{c}\Big) J_{\kappa-1}\left(\frac{4\pi \sqrt{m_1na t}}{c}\right)  w_{M_2}(t, \cdot ) \frac{\d t}{\sqrt{t}}.
                \]
               Here the weight function $w_{M_2}$ is a piece of a dyadic partition of unity in the $m_1$, $m_2$, $c$, and $k$ variables times $V(m_1 m_2/q)$. The function $J_*$ has $n = n_1 n_2n_3$ appearing as a block.

		By \eqref{eq:Bdef}, we have (introducing subscripts on $B$ now to re-emphasize the choice of the partition of unity)
        \begin{multline} \label{eq:BIntegral}
                B_{N_1,N_2,N_3}(k_0') =  
                \iiint_{(\R^+)^3}  \frac{J_*(e_1e_2\delta_1t_1t_2t_3,a,m_1',c_0, g_0k_0', c_2,k_1)}{\sqrt{\delta_1 e_1 e_2}} 
                \\
                \times
                F_a\left(\frac{t_1r_1e_1}{\sqrt{q}} , \frac{t_2r_2e_2}{\sqrt{q}}, \frac{t_3r_3}{\sqrt{q}}\right)
                 \omega\left(\frac{t_1e_1r_1}{N_1},\frac{t_2e_2r_2}{N_2} , \frac{t_3r_3}{N_3}\right) \frac{\d t_1 \d t_2 \d t_3}{\sqrt{t_1t_2t_3}}.
        \end{multline}
	The $c,k,m_1,m_2$ partitions are implicit in the definition of $J_*$.
        
        Summing over all dyadic numbers $N_1,N_2,N_3 \geq 2^{-1/2}$, we obtain that 
        \begin{multline}
        \label{eq:BSummedOverPartitionOfUnity}
                \sum_{2^{-1/2} \leq N_1,N_2,N_3 \text{ dyadic}} B_{N_1,N_2,N_3}(k_0') 
                \\
                =
                \iiint_{(\R^+)^3}  \frac{J_*(t_1t_2t_3,a,m_1',c_0, g_0k_0', c_2,k_1)}{\delta_1 e_1 e_2} 
                F_a\left(\frac{t_1}{\sqrt{q}} , \frac{t_2}{\sqrt{q}}, \frac{t_3}{\sqrt{q}}\right)
                                 W(t_1,t_2,t_3)\frac{\d t_1\d t_2 \d t_3}{\sqrt{t_1t_2t_3} },
        \end{multline}
        where $W(t_1,t_2,t_3) = \sum_{2^{-1/2} \leq N_1,N_2,N_3 \text{ dyadic}} \omega(t_1/N_1,t_2/N_2, t_3/N_3)$. Note that the function $1 - W(t_1,t_2,t_3)$ is $0$ if $t_i\geq 1$ for all $i$.  It is a slightly subtle point that it is not true that $W(t_1,t_2,t_3)=1$ for all $t_i > 0$.  
       
        Our immediate goal is to replace the $W$ function by $1$, and estimate the error.  
        The basic idea is that $1-W(t_1,t_2,t_3)$ should save a factor $q^{1/4}$ from the fact that at least one of the $t_i$ is $\leq 1$, in place of $q^{1/2+\varepsilon}$.  
Here this numerology comes from that $\int_1^{q^{1/2}} t^{-1/2} dt \asymp q^{1/4}$, but $\int_0^{1} t^{-1/2} dt \ll 1$.       
        In light of the claim that the trivial bound on $\mathcal{S}_{0,0,0}$ leads to $O(q^{1/4+\varepsilon})$, one naturally expects that this reasoning should lead to an acceptable final bound.  Our next order of business is to confirm this expectation.

        \begin{mylemma}
        \label{lemma:BDeltabound}
            Let 
            \begin{multline*}
                B_{\Delta}(k_0') =          
                \iiint_{(\R^+)^3} 
 \frac{ J_*(t_1t_2t_3,a,m_1',c_0, g_0k_0', c_2,k_1)}{\delta_1 e_1 e_2}                 
                   \\
                \times
                F_a\left(\frac{t_1}{\sqrt{q}} , \frac{t_2}{\sqrt{q}}, \frac{t_3}{\sqrt{q}}\right)
                 (1 - W(t_1,t_2,t_3) )\frac{\d t_1\d t_2 \d t_3}{\sqrt{t_1t_2t_3} }.
            \end{multline*}
            Let $\mathcal{S}_{\Delta}'''$ be as in \eqref{eq:STriplePrimeDefinition} but with 
            $B$ replaced with $B_{\Delta}$.
            Then  
            \begin{equation}
            \label{eq:SDelta'''bound}
                  \mathcal{S}_{\Delta}''' \ll  q^{\varepsilon} \frac{(g_0, c_2)}{C c_2 k_1 k_1^*} \frac{m_1^{1/2} M_2}{a^{3/2} \delta_1 e_1 e_2}.
            \end{equation}
        \end{mylemma}
        
        \begin{proof}
            Notice that the support of $1 - W(t_1,t_2,t_3)$ is essentially included in the union of domains where one of the variables is in $[0,1]$ and the other two are
            restricted to $ [0, q^{\hf + \vep}/a]$. The bound on the other two variables comes from the dropoff due to the function $F_{a,\sqrt{q}}(t_1,t_2,t_3)$.
            
            Using only the trivial bound $J_{\kappa-1}(x) \ll x$ and $|I| = |J_*|$, we derive from \eqref{eq:Idef} (which we bound trivially) that
            \begin{equation}
            \label{eq:JtrivialBound}
|J_*(t_1 t_2 t_3, a m_1', c_0, g_0 k_0', c_2, k_1)| \ll \frac{M_2 \sqrt{m_1  a  t_1t_2t_3}}{C}. 
            \end{equation}
            Therefore, using the above restrictions on the size of the $t_i$, and \eqref{eq:FaInert} we derive
            \begin{equation*}
            B_{\Delta}(k_0') \ll q^{\varepsilon} \frac{q m_1^{1/2} M_2}{a^{3/2} \delta_1 e_1 e_2 C}.
            \end{equation*}
For the arithmetical part, we have $$\frac{1}{k_0'^3} A(k_0') \ll \frac{\tau(k_0')}{ k_0'}. 
             $$ 
             Hence
             \begin{equation*}
              \mathcal{S}_{\Delta}''' \ll q^{\varepsilon}
              \sum_{\substack{(c_0,g_0m_1') = 1\\ c_0\equiv 0 \shortmod{q k_1 k_1^*}} }
                \frac{1}{c_0} 
                \sum_{\substack{(k_0',\delta_2 c_0)=1 \\ k_0' \equiv 0 \shortmod{\frac{c_2}{(g_0,c_2)}}  }} 
                \frac{q m_1^{1/2} M_2}{a^{3/2} \delta_1 e_1 e_2 C} \frac{\tau(k_0')}{ k_0'},
             \end{equation*}
which quickly leads to \eqref{eq:SDelta'''bound}.
        \end{proof}

        Recall that
        \begin{equation}
        \label{eq:Sprimes}
             \mathcal{S}_{0,0,0}'' =  \sum_{\substack{g_0 | e_1 e_2 \delta_1 am_1'\\ g_0 \equiv 0 \shortmod{d}}} \mathcal{S}_{0,0,0}''', \quad 
             \mathcal{S}_{0,0,0}' =  \sum_{r_1 r_2 r_3 = \delta_1} 
             \sum_{e_1 | r_2 r_3} \sum_{e_2 | r_3} \mu(e_1)  \mu(e_2) \mathcal{S}_{0,0,0}'',
        \end{equation}
        and
        \begin{equation} \label{eq:S0Definition}
            \mathcal{S}_{0,0,0} =   \sum_{(a,q) = 1} \frac{\mu(a)}{a^{3/2}} 
            \sum_{c_2} \frac{1}{c_2^{3/2}} \sum_{d|c_2} d \mu(c_2/d)
            \sum_{\substack{k_1  } }  k_1^{1/2} \sum_{m_1'} \frac{1}{\sqrt{m_1'}} \mathcal{S}'_0.
        \end{equation}
        
        Let $\mathcal{S}_{\Delta}'', \mathcal{S}_{\Delta}'$ and $\mathcal{S}_{\Delta}$ be defined similarly.
Using        Lemma \ref{lemma:BDeltabound} and $(g_0,c_2) \leq c_2$ then implies that 
        \begin{equation*}
            \mathcal{S}_{\Delta}  \ll q^{\vep} \sum_{a} \frac{1}{a^{3/2}} 
            \sum_{c_2} \frac{1}{c_2^{3/2}} \sum_{d|c_2} d 
            \sum_{\substack{k_1  } }  k_1^{1/2} \sum_{m_1'} \frac{1}{\sqrt{m_1'}}
            \frac{(m_1' k_1 c_2)^{1/2} M_2}{C  k_1 k_1^* a^{3/2}} . 
        \end{equation*}
Using $m_1 = m_1' k_1 c_2 \ll \frac{q^{1+\varepsilon}}{M_2}$, and $C \gg q$, we obtain 
        \begin{equation*}
        \mathcal{S}_{\Delta}  \ll \frac{q^{1+\varepsilon}}{C} \ll q^{\varepsilon}.
        \end{equation*}

Define $\overline{\mathcal{S}}_{0,0,0}'''$ to be the same as $\mathcal{S}_{0,0,0}'''$ but with $W$ replaced by $1$, so that
\begin{equation*}
 \mathcal{S}_{0,0,0}''' = \overline{\mathcal{S}}'''_{0,0,0} + \mathcal{S}_{\Delta}''',
\end{equation*}
and similarly for $\overline{\mathcal{S}}_{0,0,0}''$, etc.  
To show Theorem \ref{thm:S000bound}, we therefore need to show $\sum_T \overline{\mathcal{S}}_{0,0,0} \ll q^{\varepsilon}$.

        \subsection{The function $B(k_0')$}
        From now on, we let $\overline{B}(k_0')$ be the function obtained from the right hand side of \eqref{eq:BSummedOverPartitionOfUnity} after replacing $W$ by $1$, and summing over the dyadic variables $C$ and $K$.  This has the shape
        \begin{equation}
        \label{eq:BdefNoPartition}
                \overline{B}(k_0') =  \iiint_{(\R^+)^3}
                 \frac{J_*(t_1t_2t_3,a,m_1',c_0, g_0k_0', c_2,k_1)}{\delta_1 e_1e_2}
                F_a\left(\frac{t_1}{\sqrt{q}} , \frac{t_2}{\sqrt{q}}, \frac{t_3}{\sqrt{q}}\right) \frac{\d t_1 \d t_2 \d t_3}{\sqrt{t_1t_2t_3}},
        \end{equation}
where we did not give a new name to $J_*$ after summing over $C$ and $K$.
        This is the relevant function for evaluating $\overline{\mathcal{S}}_{0,0,0}$.
  
        \begin{myprop} \label{prop:BTripleMellin}
                Denote
                \begin{equation}
                \label{eq:Hdef}
                        \mathcal{H}(s,w,u,\kappa) = (-1)^{\frac{\kappa}{2}} 
                        \frac{ (2 \pi)^{s+w+u-1}\Gamma( s+ w + u) \Gamma(\frac{\kappa}{2}  - w - u ) \Gamma(\frac{\kappa }{2}-s)}
     {\Gamma(\frac{\kappa}{2}+ s)\Gamma(\frac{\kappa}{2} + w + u)}
.
                \end{equation}
                Here $\mathcal{H}$ is holomorphic in the region 
                \begin{equation*}
                        \Re(s) , \Re(w + u) < \frac{\kappa}{2}, \quad \text{ and } \quad  \Re(s + w + u)>0,
                \end{equation*}
                with polynomial growth in $\Im(s)$, $\Im(w)$, and $\Im(u)$ in vertical strips. With this notation, 
                \begin{multline*} 
                        \overline{B}(k_0') = \frac{1}{\delta_1e_1e_2} \int_{(1-\vep)}\frac{\gamma(1/2+s, \kappa)^3 G(s)^3}{\gamma(1/2, \kappa)^3 s^3} \zeta_q(1+2s)^3  q^{3s/2} a^{-3s}\\
\int_{(1-2\vep )} \widetilde{V}(w) \left(\frac{m_1}{q}\right)^{-w} \int_{(-2\vep)} M_2^{u} \widetilde{\omega}(u, \cdot) 
 \frac{1}{k^{s - w - u}}  \frac{(am_1)^{s - 1/2}}{c^{s + w + u-1}} \mathcal{H}(s,w,u,\kappa)  \frac{\d u \d w \d s}{(2 \pi i)^3},
                \end{multline*}
                where $k = g_0k_0'k_1, m_1 = m_1'k_1c_2$, $c = c_0c_2$, and $\widetilde{\omega}(u, \cdot) = \widetilde{\omega}(u) \omega(m_1/M_1)$.
        \end{myprop}   
        
\begin{proof}
Note that in \eqref{eq:BdefNoPartition}, the factor $t_1t_2t_3$ shows up as a block in both $J$ and the denominator. Letting $y = t_1t_2t_3$ (viewing $t_2$ and $t_3$ as fixed), we have 
        \begin{equation*}
                \overline{B}(k_0') =  \int_0^\infty 
                 \frac{J_*(y,a,m_1',c_0, g_0k_0', c_2,k_1)}{\delta_1 e_1e_2}
                \int_{0}^{\infty}\int_0^\infty F_a\left(\frac{y/t_2t_3}{\sqrt{q}} , \frac{t_2}{\sqrt{q}}, \frac{t_3}{\sqrt{q}}\right) 
                \frac{\d t_2}{t_2} \frac{\d t_3}{t_3} \frac{\d y}{\sqrt{y}}.
        \end{equation*}
We first claim that
\begin{equation*}
\int_{0}^{\infty}\int_0^\infty F_a\left(\frac{y/t_2t_3}{\sqrt{q}} , \frac{t_2}{\sqrt{q}}, \frac{t_3}{\sqrt{q}}\right) 
                \frac{\d t_2}{t_2} \frac{\d t_3}{t_3}
                =
\int_{(2)}  \frac{q^{\frac{3s}{2}}}{(a^3y)^{s}}
                \frac{\gamma(\hf + s, \kappa)^3 G(s)^3}{\gamma(\hf,\kappa)^3 s^3} \zeta_q(1 + 2s)^3 \frac{\d s }{2\pi i }.
\end{equation*}
This is an exercise with Mellin inversion, directly using the definition \eqref{eq:FaDef}.
Secondly, we claim
\begin{multline}
\label{eq:JMellinformula}
\int_0^{\infty} J_*(y,a,m_1',c_0, g_0k_0', c_2,k_1) y^{-\half-s} dy 
\\
= 
 \int_{(1-2\vep )} \widetilde{V}(w) \left(\frac{m_1}{q}\right)^{-w} \int_{(-2\vep)} M_2^{u} \widetilde{\omega}(u, \cdot) 
 \frac{1}{k^{s - w - u}}  \frac{(am_1)^{s - 1/2}}{c^{s + w + u-1}} \mathcal{H}(s,w,u,\kappa)  \frac{\d u \d w}{(2\pi i)^2}.
\end{multline}
Putting these two claims together then completes the proof.

Now we show \eqref{eq:JMellinformula}.
    From \eqref{eq:Jdef} and \eqref{eq:Idef}, and summing over the $C$ and $K$ partitions, we have
        \begin{multline*}
                J_*(y, a, m_1', c_0, g_0 k_0', c_2, k_1) 
                =  e\left(-\frac{y a m_1' }{c_0 g_0 k_0' }\right) I(m_1'k_1c_2 , g_0 k_0' k_1, y a, c_0c_2)
 \\
                = 
                 \int_0^{\infty} e\left(-\frac{y a m_1' }{c_0 g_0 k_0' }\right)
                e\Big(\frac{-g_0 k_0' k_1 t}{c_0c_2}\Big)
                J_{\kappa-1}\left(\frac{4\pi \sqrt{m_1'k_1c_2 y a t}}{c_0c_2}\right)                    
                V_1\left(\frac{m_1'k_1c_2 t}{q} \right) \omega_{M_2}(t, \cdot ) \frac{d t}{\sqrt{t}}
                \\
                =
                 \int_0^{\infty} e\left(-\frac{y a m_1 }{c k }\right) e\Big(\frac{-k t}{c }\Big)
                J_{\kappa-1}\left(\frac{4\pi \sqrt{m_1  y a t}}{c }\right) V_1\left(\frac{m_1  t}{q} \right) \omega_{M_2}(t , \cdot) \frac{d t}{\sqrt{t}},
\end{multline*}
where for simplicity in the final line above we have written the expression in terms of the earlier variable names, and where $\omega_{M_2}(t, \cdot) = \omega(t/M_2) \omega(m_1/M_1)$ (since we have summed over $C$ and $K$, as well as the $N_i$).
Therefore, \eqref{eq:JMellinformula} equals
\begin{multline*}
         \int_0^\infty \int_0^\infty e\left(-\frac{y a m_1 }{c k }\right)
         e\Big(\frac{-k t}{c }\Big) 
         J_{\kappa-1}\left(\frac{4\pi \sqrt{m_1  y a t}}{c }\right) V_{1} \left(\frac{m_1  t}{q} \right) \omega_{M_2}(t, \cdot ) \frac{d t}{\sqrt{t}} y^{-s} \frac{dy}{\sqrt{y}}.
\end{multline*}
Change variables by $y = z/t$ (after interchanging the order of integration), giving that \eqref{eq:JMellinformula} equals
\begin{multline*}
         \int_0^\infty \int_0^\infty e\left(-\frac{z a m_1 }{c k t}\right)
         e\Big(\frac{-k t}{c }\Big) 
         J_{\kappa-1}\left(\frac{4\pi \sqrt{m_1  z a }}{c }\right) V_{1} \left(\frac{m_1  t}{q} \right) \omega_{M_2}(t, \cdot ) t^s \frac{d t}{t} z^{-s} \frac{dz}{\sqrt{z}}.
\end{multline*}

Rewriting $V_1$ and $\omega_{M_2}$ in terms of their Mellin transforms, 
we have that \eqref{eq:JMellinformula} is
        \begin{equation}\label{eq:BIntegral2}
        \int_{(1- 2\vep)} \widetilde{V}(w) \left(\frac{m_1}{q}\right)^{-w} 
        \int_{(-\vep)} M_2^{u} \widetilde{\omega}(u, \cdot)  \mehmetletter 
        \frac{du dw}{(2 \pi i)^2},
        \end{equation}
        where $\mehmetletter$ is shorthand for
        \begin{equation} \label{eq:twoMellinIntegrals}
                \mehmetletter = \int_0^\infty 
                J_{\kappa-1}\left(\frac{4\pi \sqrt{m_1  az}}{c }\right) z^{-s}
                \Big(\int_0^\infty e\left(-\frac{m_1 a z }{c kt }\right)
                e\Big(\frac{-k t}{c }\Big) 
                t^{s-w-u} \frac{d t}{t} \Big)
                 \frac{dz}{\sqrt{z}}.
        \end{equation}
 	We will derive an explicit formula for $\mathcal{I}$ by consulting tables of integrals.
        \begin{mylemma}
                For $|\Re(s - w - u)| <1$, we have
                \begin{multline*}
                        \int_0^\infty e\left(-\frac{m_1 a z }{c kt }\right) e\left(\frac{-k t}{c }\right)  t^{s-w-u} \frac{d t}{t}
\\                        
                        = - i \pi \left(\frac{\sqrt{m_1 a z}}{k}\right)^{s-w-u}  e^{-\pi i  \frac{s-w-u}{2}} H^{(2)}_{s-w-u}\left(\frac{4\pi \sqrt{m_1 a z}}{c}\right).
                \end{multline*}
                %
        \end{mylemma}
        \begin{proof}
        This follows from \cite[(3.871.1), (3.871.2)]{GR}, or
        formulas (17) and (36) in \cite[Section 6.5]{ErdelyiTablesVol1}. 
        \end{proof}
        
        Even though the original calculation requires $|\Re(s - w - u)| <1$ for convergence, note that the Hankel function $H^{(2)}_{\nu}$ is an analytic function of $\nu$ and hence we may move our lines of integration in $s, w$ and $u$ to any location without encountering any poles from the Hankel function. 
        
        Inserting this evaluation into \eqref{eq:twoMellinIntegrals}, we have
        \begin{equation*}
        \mehmetletter=-i \pi   e^{-\pi i  \frac{s-w-u}{2}} 
        \int_0^\infty  \left(\frac{\sqrt{m_1 a z}}{k}\right)^{s-w-u} H^{(2)}_{s-w-u}\left(\frac{4\pi \sqrt{m_1 a z}}{c}\right) J_{\kappa-1}\left(\frac{4\pi \sqrt{m_1  az}}{c }\right) z^{\frac12-s} \frac{dz}{z}.
        \end{equation*}
        
        The Bessel and Hankel functions have the same argument, which is quite pleasant.  By changing variables, we have 
        \begin{equation} \label{eq:IMellinTransform}
        \mehmetletter =  \frac{-  i e^{-\pi i  \frac{s-w-u}{2}}}{2 k^{s - w - u}}  \frac{( am_1)^{s - 1/2} (4\pi)^{s+w+u}}{c^{s + w + u-1}}  \int_0^\infty H^{(2)}_{s - w-u} (z) J_{\kappa - 1} (z)  z^{1-s - w - u} \frac{\d z}{z}
        .
        \end{equation}
        
        The $z$-integral may be evaluated in closed form.
        \begin{mylemma}
        For $\Re(\pm\nu -  \mu) < \Re(\lambda) <1$, we have
        \begin{equation}
        \label{eq:H2integralwithJ}
        \int_0^\infty H^{(2)}_\nu (x) J_{\mu}(x) x^\lambda \frac{\d x }{x}
        =\frac{i 2^{\lambda - 1}\Gamma(1 - \lambda) \Gamma(\frac{\nu + \mu + \lambda}{2}) \Gamma(\frac{\mu - \nu + \lambda}{2})}{\pi \Gamma(\frac{\nu + \mu - \lambda}{2} + 1)\Gamma(\frac{\mu- \nu - \lambda}{2} + 1)} e^{-\frac \pi 2 i (\mu - \nu + \lambda)}.
        \end{equation}
        \end{mylemma}
         \begin{proof}
         We may calculate this using formulas (33) and (36) 
         in \cite[Section 6.8]{ErdelyiTablesVol1} (but note (36) is missing a $\Gamma(1 - \lambda)$ term), and simplifying using gamma function identities.
        \end{proof}

Substituting 
        \[
                \lambda = 1- s- w - u, \qquad \nu = s - w - u, \qquad \text{ and } \qquad \mu = \kappa -  1,
        \]
        the region of convergence corresponds to
        \begin{equation}\label{eq:productMellinRegion}
                \Re(s), \Re(w + u)  < \frac{ \kappa}{2} \qquad \text{ and } \qquad \Re(s + w + u) >0,
        \end{equation}
        which are satisfied by the lines of integration given in \eqref{eq:BIntegral2}.  Furthermore,
    \begin{equation*}
     \mehmetletter = \frac{ (-i)^{\kappa} e^{\pi i  \frac{s+w+u}{2}} }{k^{s - w - u}}  \frac{( am_1)^{s - 1/2} }{ c^{s + w + u-1}}
     \frac{ (2 \pi)^{s+w+u-1}\Gamma( s+ w + u) \Gamma(\frac{\kappa}{2}  - w - u ) \Gamma(\frac{\kappa }{2}-s)}
     {\Gamma(\frac{\kappa}{2}+ s)\Gamma(\frac{\kappa}{2} + w + u)},
    \end{equation*}
 giving
 \begin{equation}
 \label{eq:mehmetletterdef}
                \mehmetletter =  \frac{\mathcal{H}(s,w,u,\kappa)}{k^{s - w - u}}  \frac{(am_1)^{s - 1/2}}{c^{s + w + u-1}} .  
        \end{equation}
    
        An application of Stirling's approximation shows the growth in $\Im(s), \Im(w)$ and $\Im(u)$ is bounded by a polynomial.
  
  Inserting this formula for $\mehmetletter$ into \eqref{eq:BIntegral2}, we complete the proof of
       Proposition \ref{prop:BTripleMellin}.
\end{proof}

        \subsection{Bounding the zero term}
        \label{section:boundingallpi=0term}
        Now let us recall that $k = g_0k_0'k_1$, $m_1 = m_1'k_1c_2$ and $c = c_0c_2$. We will substitute the evaluation of $\overline{B}$ into $\overline{\mathcal{S}}_{0,0,0}'''$ which was defined as \eqref{eq:STriplePrimeDefinition} (with the partition of unity removed).
     This gives
     \begin{multline}
      \label{eq:S0bar'''formula0}
       \overline{\mathcal{S}}_{0,0,0}'''= \frac{1}{\delta_1e_1e_2}
       \sum_{\substack{(c_0,g_0m_1') = 1\\ c_0\equiv 0 \shortmod{q k_1 k_1^*}} }
                \frac{1}{c_0} 
                \sum_{\substack{(k_0',\delta_2 c_0)=1 \\ k_0' \equiv 0 \shortmod{\frac{c_2}{(g_0,c_2)}}  }}
                \frac{A(k_0')}{k_0'^3} 
                 \int_{(1-\vep)}\frac{\gamma(1/2+s, \kappa)^3 G(s)^3}{\gamma(1/2, \kappa)^3 s^3} 
                \\
                \zeta_q(1+2s)^3  q^{3s/2} a^{-3s} 
\int_{(1-2\vep )} \widetilde{V}(w) \left(\frac{q}{m_1' k_1 c_2}\right)^{w} \int_{(-2\vep)} M_2^{u} \widetilde{\omega}(u, \cdot) 
\\
 \frac{\mathcal{H}(s,w,u,\kappa)}{(g_0 k_0' k_1)^{s - w - u}}  \frac{(am_1' k_1 c_2)^{s - 1/2}}{(c_0 c_2)^{s + w + u-1}}  
\frac{\d u \d w \d s}{(2 \pi i)^3} 
 .
     \end{multline} 

Next examine the Dirichlet series
\begin{equation*}
\mathcal{Z}(s,w,u) = \zeta_q(1+2s)^3 \sum_{\substack{(c_0,g_0 m_1') = 1\\  c_0 \equiv 0 \shortmod{qk_1 k_1^*}}} \frac{1}{c_0^{s + w + u}}
 \sum_{\substack{( k_0',\delta_2  c_0) = 1 \\ k_0' \equiv 0 \shortmod{\frac{c_2}{(g_0, c_2)}} }}   
\frac{\frac{1}{k_0'^3}A(0,0,0;k_0')}{k_0'^{s - w - u}}.
\end{equation*}
Using the evaluation $A(0,0,0;k_0') = k_0' \sum_{d|k_0'} d \phi(k_0'/d)$ and M\"obius inversion to remove the condition $(k_0', c_0) = 1$, one may  derive 
\begin{equation}
\label{eq:Zevaluation}
\mathcal{Z}(s,w,u) =  \frac{\zeta(1+s-w-u)^2 \zeta(s+w+u)}{c_2^{s-w-u+1} (qk_1 k_1^*)^{s+w+u}} (g_0, c_2)^{s-w-u+1} \Delta(s,w,u),
\end{equation}
where $\Delta(s,w,u)$ is analytic for 
\begin{equation}
\label{eq:DeltaRegionAnalytic}
\text{Re}(s) > 0, \qquad \text{Re}(u+w) < 1+\text{Re}(s), 
 \qquad \text{Re}(s+w+u) > 0,
\end{equation}
and bounded by $O(q^{\varepsilon})$ in that region.  The sum defining $\mathcal{Z}(s,w,u)$ converges absolutely for $\text{Re}(s+w+u) > 1$ and $\text{Re}(w+u) < \text{Re}(s)$.

Inserting \eqref{eq:Zevaluation} into \eqref{eq:S0bar'''formula0}, we obtain
\begin{multline}
\label{eq:S0bar'''formula2}
     \overline{\mathcal{S}}_{0,0,0}'''= 
     \int_{(1-\vep)}\frac{\gamma(1/2+s, \kappa)^3 G(s)^3}{\delta_1 e_1 e_2 \gamma(1/2, \kappa)^3 s^3} 
                  \frac{q^{3s/2}}{a^{3s}} 
\int_{(1-2\vep )} \widetilde{V}(w) \left(\frac{q}{m_1' k_1 c_2}\right)^{w} \int_{(-2\vep)} M_2^{u} \widetilde{\omega}(u, \cdot) 
\\
\frac{\zeta(1+s-w-u)^2 \zeta(s+w+u)}{c_2^{s-w-u+1} (qk_1 k_1^*)^{s+w+u}} (g_0, c_2)^{s-w-u+1} 
 \frac{\Delta \mathcal{H}(s,w,u,\kappa)}{(g_0  k_1)^{s - w - u}}  \frac{(am_1' k_1 c_2)^{s - 1/2}}{c_2^{s + w + u-1}}  \frac{\d u \d w \d s}{(2 \pi i)^3}.
\end{multline}

Now move the contour of integration in $w$ to the line $\Re(w) = 4\vep$. In doing that note that we still have $\Re(s + w + u) = 1 - \vep + 4\vep- \vep = 1 +  2\vep >1$ and 
$\Re(s - w - u + 1) = 1 - \vep - 4\vep + 2 \vep +1= 2 - 3\vep > 1$, so we do not pass over any poles. Now move the line of integration in $s$ to $\Re(s) = 3\vep$. 
By doing so, we pick up the residue from the simple pole at $s = 1 - w - u$. 

\noindent {\bf The remaining integral.}
    The contribution from the final integral to $\overline{\mathcal{S}}_{0,0,0}'''$ is at most
    \begin{equation} \label{eq:integralBound}
  \ll \frac{(g_0,c_2)q^{\varepsilon} }{\delta_1 e_1 e_2 \sqrt{am_1' k_1 c_2}} .
    \end{equation}
The contribution to $\overline{\mathcal{S}}_{0,0,0}$ from this part is then calculated (recall that $\delta_1 = k_1d/(a,k_1d)$ and $\delta_2 = e_1e_2\delta_1am_1'/g_0$) to be at most
\begin{multline*}
    \sum_{a} \frac{1}{a^{3/2}} 
\sum_{c_2} \frac{1}{c_2^{3/2}} 
\sum_{d|c_2} d
\sum_{\substack{k_1  } }  k_1^{1/2}
\sum_{m_1'} \frac{1}{\sqrt{m_1'}} 
 \sum_{r_1 r_2 r_3 = \delta_1} 
 \sum_{e_1 | r_2 r_3} \sum_{e_2 | r_3}
 \sum_{g_0 | e_1 e_2 \delta_1 am_1'}
\frac{(g_0,c_2) q^{\varepsilon} }{\delta_1 e_1e_2 \sqrt{a m_1' k_1 c_2}} \\
\ll   q^{\varepsilon} \sum_{a} \frac{1}{a^2} \sum_{c_2} \frac{1}{c_2^{2}} \sum_{d | c_2}  \sum_{k_1} \frac{(k_1d, a)}{k_1} \sum_{m_1'} \frac{1}{m_1'} \sum_{e_1 | r_2 r_3} \sum_{e_2| r_3} \frac{1}{e_1e_2}
\sum_{g_0 | e_1 e_2 \delta_1 am_1'} (g_0,c_2),
\end{multline*}
where all the summations may be truncated at some fixed power of $q$ 
(cf. the convention in Section \ref{section:EstimatingT}). 
Summing over everything trivially using $(g_0,c_2) \leq c_2$ shows that the integral contribution to $\overline{\mathcal{S}}_{0,0,0}$ is $O(q^{\varepsilon})$.

\noindent {\bf The $s = 1-w - u$ residue.}
This residue contributes to $\overline{\mathcal{S}}_{0,0,0}'''$ the following:
\begin{multline}
\label{eq:S0bar'''residue}
      \int_{(4\vep)}\frac{\gamma(3/2-w-u, \kappa)^3 G(1-w-u)^3}{\delta_1 e_1 e_2 \gamma(1/2, \kappa)^3 (1-w-u)^3}                  
 \widetilde{V}(w) \left(\frac{1}{m_1' k_1 c_2}\right)^{w} \int_{(-2\vep)} M_2^{u} \widetilde{\omega}(u, \cdot) 
\\
\frac{q^{\frac{1-w-u}{2}-u}}{a^{3(1-w-u)}} 
\frac{\zeta(2-2w-2u)^2 }{c_2^{2-2w-2u} (k_1 k_1^*)^{}} (g_0, c_2)^{2-2w-2u} 
 \frac{\Delta \mathcal{H}(s,w,u,\kappa)}{(g_0  k_1)^{1 - 2w - 2u}}  (am_1' k_1 c_2)^{  1/2-w-u}  \frac{\d u \d w}{(2 \pi i)^2}.
 \end{multline}
    Now move the line of integration in $w$ to $\Re(w) = 1-\vep$. This will pass over an apparent double pole of $\zeta(2 - 2w - 2u)$ but the triple zero of $G(1 -w - u)^3$ cancels it. 
    Then by a trivial bound, we have that the residue is 
    \begin{equation}    \label{eq:residueBound}
  \ll \frac{g_0 q^{\varepsilon}}{\delta_1 e_1e_2 a^{1/2} c_2^{3/2} k_1^{3/2} k_1^* m_1'^{3/2}}.
    \end{equation}

The contribution coming from the residue can be bounded as follows:
\begin{multline*}
        \sum_{a} \frac{1}{a^{2}} 
\sum_{c_2} \frac{1}{c_2^{3}} 
\sum_{d|c_2} d
\sum_{\substack{k_1  } }  \frac{1}{k_1k_1^*}
\sum_{m_1'} \frac{1}{m_1'^2} 
 \sum_{r_1 r_2 r_3 = \delta_1} \frac{1}{\delta_1}
 \sum_{e_1 | r_2 r_3} \sum_{e_2 | r_3}\frac{1}{e_1e_2}
 \sum_{g_0 | e_1 e_2 \delta_1 am_1'}
 g_0.
\end{multline*}
Let us trivially bound $g_0 \leq e_1e_2\delta_1 am_1'$.  All the remaining sums are easily bounded, so this part is also $O(q^{\varepsilon})$.

This completes the proof of Theorem \ref{thm:S000bound}.

\subsection{One of the $p_i$ is zero}
This case is the easiest, since (as it turns out) we may bound everything trivially and obtain the desired bound $\mathcal{S}_{0} \ll q^{\varepsilon}$.

The original sum is symmetric in $p_1,p_2$ and $p_3$, so it suffices to estimate the terms with $p_3 = 0$, and $p_1, p_2 \neq 0$ (the expression for $A$ from Lemma \ref{lem:Pnonzero} may not appear symmetric in the $p_i$, but of course it must be due to the original definition \eqref{eq:Adef}).
We apply a dyadic partition of unity to the $p_1$ and $p_2$ variables.
Let $P_1,P_2 \neq 0$, set $P_3 = 0$, and let $P = (P_1,P_2,0)$, and consider
\begin{equation*}
	\mathcal{S}_{P}''' = 
	\sum_{\substack{(c_0,g_0 m_1') = 1\\ c_0\equiv 0 \shortmod{q k_1 k_1^*} }} \frac{1}{c_0}
\sum_{\substack{( k_0',\delta_2 c_0) = 1 \\ k_0' \equiv 0 \shortmod{\frac{c_2}{(g_0, c_2)}} }}   
\sum_{\substack{p_1 \asymp P_1\\ p_2\asymp P_2} } \frac{1}{k_0'^3} A(p_1,p_2,0;k_0') B(p_1,p_2,0).
\end{equation*}
Here
\begin{equation*}
	A(p_1,p_2,0;k_0') = k_0' \sum_{f | (p_2,k_0')} f S(p_1,0;k_0'/f) 
	\ll k_0'^{1+\varepsilon} (p_1 p_2, k_0').
\end{equation*}


Note that we only need to consider the non-oscillatory cases for $B$, where
    $B$ is given by \eqref{eq:BasymptoticNonOsc}, since in the oscillatory case all the $p_i$ must be nonzero or else $B$ is very small.  Then
    \begin{equation*}
        \mathcal{S}_{P}''' \ll
        \sum_{\substack{(c_0,g_0 m_1') = 1\\ c_0\equiv 0 \shortmod{q k_1 k_1^*} }} \frac{1}{c_0}
        \sum_{\substack{( k_0',\delta_2 c_0) = 1 \\ k_0' \equiv 0 \shortmod{\frac{c_2}{(g_0, c_2)}} }}   \frac{1}{k_0'^2} 
        \sum_{\substack{p_1, p_2 \neq 0 } }
        \Big(\frac{\sqrt{aMN}}{C}\Big)^{\exponent} \frac{\sqrt{M_2 N}}{h} (p_1 p_2, k_0') ,
\end{equation*} 
where recall $\exponent = \kappa-1 \geq 1$ in the pre-transition non-oscillatory range, and $\exponent = -1$ in the post-transition range. Recall $P_1 P_2 \ll q^{\varepsilon} \frac{k_0'^2}{N_2' N_3'} \ll q^{\varepsilon} \frac{k_0'^2 h}{N_2 N_3}$. 
Therefore,
\begin{equation*}
\mathcal{S}_{P}''' \ll q^{\varepsilon} \Big(\frac{\sqrt{aMN}}{C}\Big)^{\exponent} \frac{\sqrt{M_2 N}}{N_2 N_3}   \frac{K(g_0,c_2)}{g_0 k_1 c_2} \frac{1}{q k_1 k_1^*}.        
\end{equation*}

It is then not difficult to see that
\begin{equation*}
\mathcal{S}_P \ll  q^{\varepsilon} \max_a \Big(\frac{\sqrt{aMN}}{C}\Big)^{\exponent} \frac{\sqrt{M_2 N} K}{q N_2 N_3 \sqrt{a}} \sqrt{M_1}.
\end{equation*}
In the {\bf Post-transition case}, this bound becomes
\begin{equation*}
\mathcal{S}_P \ll q^{\varepsilon} \max_a \frac{ M_1 N_1}{q} \ll q^{\varepsilon}.
\end{equation*}
A calculation shows the {\bf Pre-transition, non-oscillatory case} leads to the same bound.  

 In all cases summing, over the dyadic values of $P$ gives $\mathcal{S}_0 \ll q^{\varepsilon}$, as desired.

    \subsection{Two of the $p_i$ are zero}
    \label{section:TwopiAreZero}
We finally consider the case where say $p_1 \neq 0$, and $p_2 = p_3 = 0$.    
This case leads to some new subtleties not present in the case with all $p_i = 0$.  The first step is to extend the sum to all $p_1 \in \mathbb{Z}$, and then subtract back the term with $p_1 = 0$.  We already showed with Theorem \ref{thm:S000bound} that the term with all $p_i = 0$ is bounded in an acceptable way.  After this, we apply Poisson summation backwards.  The net effect is precisely the same as only applying Poisson summation in the $n_2$- and $n_3$-variables, and setting $p_2 = p_3 = 0$ (up to the term with all $p_i=0$).

It is perhaps easiest to return to \eqref{eq:S'''prePoisson}.  Define $\mathcal{Q}$  to be the term we get from this, after Poisson in $n_2$ and $n_3$, and substitution of $p_2=p_3=0$, so that
\begin{equation*}
\mathcal{Q} = \sum_{\substack{(c_0,g_0 m_1') = 1\\ c_0\equiv 0 \shortmod{qk_1 k_1^*}}} 
 \frac{1}{c_0}
\sum_{\substack{( k_0',\delta_2 c_0) = 1 \\ k_0' \equiv 0 \shortmod{\frac{c_2}{(g_0, c_2)}} }}   
\sum_{\substack{n_1  \geq 1 }} 
\frac{A^*(n_1;k_0') B^*(n_1)}{\sqrt{n_1 }} 
,
\end{equation*}
where
\begin{equation*}
A^*(n_1;k_0') = \frac{1}{k_0'^2} \sum_{x_2, x_3 \shortmod{k_0'}} e\Big(\frac{  \delta_2   n_1 x_2 x_3
 \overline{c_0}}{k_0' }\Big),
\end{equation*} 
and
\begin{multline*}
B^*(n_1) = \frac{1}{\sqrt{\delta_1 e_1 e_2} } \int_0^{\infty} \int_0^{\infty} F_a\Big(\frac{r_1 e_1 n_1}{\sqrt{q}}, \frac{r_2 e_2 t_2}{\sqrt{q}}, \frac{r_3 t_3}{\sqrt{q}}\Big)
\omega\left(\frac{n_1 e_1r_1}{N_1},\frac{t_2e_2r_2}{N_2} , \frac{t_3r_3}{N_3}\right)
\\
\times
 J_*(e_1 e_2 \delta_1 n_1 t_2 t_3, a, m_1', c_0, g_0 k_0' , c_2,k_1) \frac{dt_2 dt_3}{\sqrt{t_2 t_3}}.
\end{multline*}
Since some of the details are similar (and easier) than the case where all $p_i = 0$, we will be more brief in such occasions.  We may evaluate $A^*$ directly from the definition, using a similar method to the proof of Lemma \ref{lem:Pnonzero}, which gives
\begin{equation*}
A^*(n_1;k_0') = \frac{1}{k_0'} \sum_{f | k_0'} \varphi\Big(\frac{k_0'}{f}\Big) \delta(n_1 \equiv 0 \shortmod{\frac{k_0'}{f}}).
\end{equation*}
We have $J_*(\cdots) \ll M_2^{1/2}$, which follows from bounding $J_{\kappa-1}(x) \ll 1$, and so
\begin{equation*}
B^*(n_1) \ll \Big(\frac{M_2 N_2 N_3}{r_1 r_2 r_3 e_1 e_2 e_2 r_2 r_3}\Big)^{1/2}.
\end{equation*}
In turn, this leads to the estimate
\begin{equation*}
\mathcal{Q} \ll q^{\varepsilon} \frac{\sqrt{M_2 N} (g_0,c_2)}{\delta_1 e_1 e_2 q k_1 k_1^* c_2}.
\end{equation*}
Then the contribution to $\mathcal{S}$ from $\mathcal{Q}$ is seen to be
$O(q^{-1+\varepsilon} (MN)^{1/2}) = O(q^{1/4+\varepsilon})$. 

The next step is to replace the sum of the $P_2,P_3$ partitions of unity by $1$, as in Section \ref{section:RecombiningPartitionsOfUnity}.  Since $\mathcal{Q} \ll q^{1/4+\varepsilon}$, the error in doing so is expected to be at most $O(q^{\varepsilon})$, as described in the 
paragraph preceding Lemma \ref{lemma:BDeltabound}. 
We omit the details, as this case is easier than the case with all $p_i = 0$.
Since $n_1 \geq 1$ automatically, we may easily sum over the $N_1$-partition (avoiding the analytic problems near the origin).  
We also reassemble the $C$ and $K$-partitions, at no cost.
Define $\overline{\mathcal{Q}}$ to be the sum obtained after all these partitions are removed, and $\overline{B^*}(n_1)$ to be the new function. For ease of notation we will not change the name of the function $J_*$. Then by the change of variables $y = e_1 e_2 \delta_1 n_1 t_2 t_3$ (viewing $t_3$ as fixed), we have
\begin{equation*}
\overline{B^*}(n_1) =  \int_0^{\infty}
 \int_0^{\infty} \frac{J_*(y, a, m_1', c_0, g_0 k_0' , c_2,k_1)}{\delta_1 e_1 e_2 \sqrt{n_1}}
F_a\Big(\frac{r_1 e_1 n_1}{\sqrt{q}}, \frac{\frac{y}{e_1 r_1 r_3 n_1 t_3}}{\sqrt{q}}, \frac{r_3 t_3}{\sqrt{q}}\Big)
  \frac{ dt_3}{t_3} \frac{dy}{\sqrt{y}}.
\end{equation*}
By an exercise with Mellin inversion, one may show
\begin{multline*}
\int_0^{\infty} F_a\Big(\frac{r_1 e_1 n_1}{\sqrt{q}}, \frac{\frac{y}{e_1 r_1 n_1 r_3 t_3}}{\sqrt{q}}, \frac{r_3 t_3}{\sqrt{q}}\Big)
  \frac{ dt_3}{t_3}   
  = 
   \int_{(1)} \int_{(1)} 
  \frac{\gamma(\tfrac12 + s_1,\kappa) G(s_1) }{\gamma(\tfrac12,\kappa) s_1}
  \frac{\gamma(\tfrac12 + s,\kappa)^2 G(s)^2 }{\gamma(\tfrac12,\kappa)^2 s^2}  
\\
\times  
 a^{-2s-s_1} \zeta_q(1 + s_1 + s)^2  \zeta_q(1 + 2s)
\Big(\frac{\sqrt{q}}{e_1 r_1 n_1}\Big)^{s_1}  
\Big(\frac{q e_1 r_1 n_1}{y}\Big)^{s}
     \frac{\d s_1 \d s}{(2\pi i)^{2}}.
\end{multline*}
Then using \eqref{eq:JMellinformula}, we derive
\begin{multline*}
\overline{B^*}(n_1) =
  \int_{(1-2\vep )}   \frac{\widetilde{V}(w)}{\delta_1 e_1 e_2 \sqrt{n_1}} \int_{(1)} \frac{\gamma(\tfrac12 + s_1,\kappa) G(s_1) }{\gamma(\tfrac12,\kappa) s_1} 
\int_{(1)} 
  \frac{\gamma(\tfrac12 + s,\kappa)^2 G(s)^2 }{\gamma(\tfrac12,\kappa)^2 s^2}
  \\  
 \int_{(0)} M_2^{u} \widetilde{\omega}(u, \cdot)  a^{-2s - s_1}
  \zeta_q(1 + s_1 + s)^2  \zeta_q(1 + 2s)
\Big(\frac{\sqrt{q}}{e_1 r_1 n_1}\Big)^{s_1}  
(q e_1 r_1 n_1)^{s}
\\
    \left(\frac{q}{m_1}\right)^{w} 
 \frac{\mathcal{H}(s,w,u,\kappa)}{k^{s - w - u}}  \frac{(am_1)^{s - 1/2}}{c^{s + w + u-1}}   \frac{\d u  \d s  \d s_1 \d w}{(2\pi i)^{4}},
\end{multline*}
where recall that $k = g_0 k_0' k_1$, $m_1 = m_1' k_1 c_2$, and $c = c_0 c_2$.  
Moreover, as in Section \ref{section:boundingallpi=0term}, we have $\widetilde{\omega}(u,\cdot) = \widetilde{\omega}(u) \omega(m_1/M_1)$, since we have summed over $N_1, N_2, N_3, C$, and $K$.

Applying these changes of variables, and inserting this into the definition of $\overline{\mathcal{Q}}$, we obtain
\begin{multline*}
\overline{\mathcal{Q}} = \frac{1}{\delta_1 e_1 e_2}
\sum_{\substack{(c_0,g_0 m_1') = 1\\ c_0\equiv 0 \shortmod{q k_1 k_1^*} }} 
 \frac{1}{c_0}
\sum_{\substack{( k_0',\delta_2 c_0) = 1 \\ k_0' \equiv 0 \shortmod{\frac{c_2}{(g_0, c_2)}} }}  
\frac{1}{k_0'} \sum_{f | k_0'} \varphi\Big(\frac{k_0'}{f}\Big) 
\sum_{\substack{n_1  \equiv 0 \shortmod{\frac{k_0'}{f}} }} \frac{1}{n_1} 
\\
  \int_{(1-2\vep )} \widetilde{V}(w) \int_{(1)} \frac{\gamma(\tfrac12 + s_1,\kappa) G(s_1) }{\gamma(\tfrac12,\kappa) s_1} \int_{(1-\varepsilon)} 
  \frac{\gamma(\tfrac12 + s,\kappa)^2 G(s)^2 }{\gamma(\tfrac12,\kappa)^2 s^2}  
\\
  \int_{(0)} M_2^{u} \widetilde{\omega}(u)  
 a^{-2s-s_1} \zeta_q(1 + s_1 + s)^2  \zeta_q(1 + 2s)
\Big(\frac{\sqrt{q}}{e_1 r_1 n_1}\Big)^{s_1}  
(q e_1 r_1 n_1)^{s}
    \\
    \left(\frac{q}{m_1' k_1 c_2}\right)^{w} 
 \frac{\mathcal{H}(s,w,u,\kappa)}{(g_0 k_0' k_1)^{s - w - u}}  \frac{(am_1' k_1 c_2)^{s - 1/2}}{(c_0 c_2)^{s + w + u-1}}  
 \frac{ \d u  \d s  \d s_1 \d w}{(2\pi i)^{4}}.
\end{multline*}

With the displayed lines of integration, all the outer sums converge absolutely.  Indeed, we have
\begin{multline}
\label{eq:SomeRandomDirichletSeries}
 \sum_{\substack{(c_0,g_0 m_1') = 1\\ c_0\equiv 0 \shortmod{q k_1 k_1^*}}} 
 \frac{1}{c_0^{s+w+u}}
\sum_{\substack{( k_0',\delta_2 c_0) = 1 \\ k_0' \equiv 0 \shortmod{\frac{c_2}{(g_0, c_2)}} }}  
\frac{1}{k_0'^{1+s-w-u}} \sum_{f | k_0'} \varphi\Big(\frac{k_0'}{f}\Big) 
\sum_{\substack{n_1  \equiv 0 \shortmod{\frac{k_0'}{f}} }} \frac{1}{n_1^{1+s_1-s}} 
\\
=
\zeta(1+s_1-s) 
\sum_{\substack{(c_0,g_0 m_1') = 1\\  c_0 \equiv 0 \shortmod{qk_1 k_1^*}}} 
 \frac{1}{c_0^{s+w+u}}
 \sum_{(f, \delta_2 c_0) =1} \frac{1}{f^{1+s-w-u}} 
 \sum_{\substack{(\ell, \delta_2 c_0) = 1 \\ \ell \equiv 0 \shortmod{\frac{\frac{c_2}{(g_0,c_2)}}{(f, \frac{c_2}{(g_0,c_2)}}}}} \frac{\varphi(\ell)}{\ell^{2+s_1-w-u}}.
 \end{multline}
As long as we assume that
\begin{equation*}
 \text{Re}(1+s_1 - w - u) > 0, \quad \text{Re}(1+s-w-u) > 0, \quad \text{Re}(s+w+u) > 0,
\end{equation*}
then the coprimality conditions are benign. 
Then we have that the Dirichlet series in \eqref{eq:SomeRandomDirichletSeries} is of the form
\begin{equation*}
 \zeta(1+s_1-s) \frac{\zeta(s+w+u)}{(q k_1 k_1^*)^{s+w+u}} \zeta(1+s-w-u) \zeta(1+s_1-w-u) \Big(\frac{(g_0, c_2)}{c_2}\Big)^{1+\min(s,s_1) - w- u} \Delta,
\end{equation*}
where $\Delta$ is holomorphic and bounded by $q^{\varepsilon}$, and $\min(s,s_1)$ means the variable with the smaller real part.  The factors $\zeta_q(1+s_1+s)^2 \zeta_q(1+s)$ may be absorbed into the definition of $\Delta$ provided that $\text{Re}(s), \text{Re}(s_1) > 0$.

Moving the summations to the inside, we derive
\begin{multline}
\label{eq:QbarIntegralFormula}
\overline{\mathcal{Q}} =
  \int_{(1-2\vep )} \frac{\widetilde{V}(w)}{\delta_1 e_1 e_2} \int_{(1)} \frac{\gamma(\tfrac12 + s_1,\kappa) G(s_1) }{\gamma(\tfrac12,\kappa) s_1} \int_{(1-\varepsilon)} 
  \frac{\gamma(\tfrac12 + s,\kappa)^2 G(s)^2 }{\gamma(\tfrac12,\kappa)^2 s^2}  
 \int_{(0)} M_2^{u} \widetilde{\omega}(u)  
\\
  \frac{\zeta(1+s_1-s)}{a^{2s+s_1}} 
  \frac{\zeta(s+w+u)}{(q k_1 k_1^*)^{s+w+u}} \zeta(1+s-w-u) \zeta(1+s_1-w-u) 
   (q e_1 r_1)^{s}
   \Big(\frac{\sqrt{q}}{e_1 r_1}\Big)^{s_1} 
  \\ 
\Big(\frac{(g_0, c_2)}{c_2}\Big)^{1+\min(s,s_1) - w- u}
    \left(\frac{q}{m_1' k_1 c_2}\right)^{w} 
 \frac{\Delta \mathcal{H}(s,w,u,\kappa)}{(g_0  k_1)^{s - w - u}}  \frac{(am_1' k_1 c_2)^{s - 1/2}}{ c_2^{s + w + u-1}}  
\frac{\d u  \d s  \d s_1 \d w}{(2\pi i)^{4}} 
 .
\end{multline}

For ease of reference, we list all the constraints on the variables (using $\kappa \geq 2$):
\begin{equation}
\label{eq:integralconstraints}
\begin{split}
0 &< \text{Re}(s+w+u), \quad \text{Re}(s) < 1, 
\\
\text{Re}(w+u) &< 1+ \min(\text{Re}(s), \text{Re}(s_1)), \quad
\text{Re}(s), \text{Re}(s_1) > 0.
\end{split}
\end{equation}

Now we move the contours as follows.  First, move $w$ from $1-2\varepsilon$ to $4\varepsilon$, which does not involve crossing any poles.  Following this, move $s_1$ to $5 \varepsilon$, which crosses a pole at $s_1 = s$ only.  Next we move $s$ to $6 \varepsilon$, which crosses a pole at $s+w+u=1$ only.  We will deal with this pole momentarily.  

{\bf The pole at $s_1 = s$.}  This contributes to $\overline{Q}$
\begin{multline*}
\frac{1}{\delta_1 e_1 e_2}
  \int_{(4\vep )} \widetilde{V}(w)  \int_{(1-\varepsilon)} 
  \frac{\gamma(\tfrac12 + s,\kappa)^3 G(s)^3 }{\gamma(\tfrac12,\kappa)^3 s^3}  
\int_{(0)} M_2^{u} \widetilde{\omega}(u)  
   \frac{\zeta(s+w+u)}{(q k_1 k_1^*)^{s+w+u}}  q^{\frac{3s}{2}}
\\   
   \frac{\zeta(1+s-w-u)^2}{a^{3s}}
   \Big(\frac{(g_0, c_2)}{c_2}\Big)^{1+s - w- u} 
    \left(\frac{q}{m_1' k_1 c_2}\right)^{w} 
 \frac{\Delta \mathcal{H}(s,w,u,\kappa)}{(g_0  k_1)^{s - w - u}}  \frac{(am_1' k_1 c_2)^{s - 1/2}}{ c_2^{s + w + u-1}} 
\frac{\d u  \d s  \d w}{(2\pi i)^{3}} 
  .
\end{multline*}
A careful scrutiny of this formula shows that is is essentially identical to \eqref{eq:S0bar'''formula2} (we did not check that the $\Delta$ function is literally equal in the two cases, but this would not be surprising). 
Here we need that we can move $w$ to $4 \varepsilon$ and then $u$ to $-2\varepsilon$ without crossing any poles;  this move in $w$ was our first step following \eqref{eq:S0bar'''formula2}, so this is easily checked.
 Therefore, by the work in the case with all $p_i = 0$, the contribution to $\mathcal{S}_{0,0}$  from this pole is $O(q^{\varepsilon})$.

{\bf The new contour.}  On the new line, with all the variables at multiples of $\varepsilon$, we have that the contribution to $\overline{Q}$ is
\begin{equation*}
\ll q^{\varepsilon} \frac{1}{\delta_1 e_1 e_2 } \frac{c_2^{1/2}}{(am_1' k_1)^{1/2}} \frac{(g_0, c_2)}{c_2}.
\end{equation*}
Recalling that $\delta_1 = \frac{k_1 d}{(a, k_1 d)}$, it is not hard to see that inserting this bound into \eqref{eq:S''intoS'''}, \eqref{eq:S''andS'}, \eqref{eq:SintoS'}, gives a final contribution to $\mathcal{S}_{0,0}$ of size $O(q^{\varepsilon})$.  

{\bf The pole at $s=1-u-w$.}  Call this contribution to $\overline{\mathcal{Q}}$ by $\overline{\mathcal{Q}}_{\text{Res}}$.  Then
\begin{multline*}
 \overline{\mathcal{Q}}_{\text{Res}} =  
  \int_{(4\vep )}  \frac{\widetilde{V}(w)}{\delta_1 e_1 e_2} \int_{(5 \varepsilon)} \frac{\gamma(\tfrac12 + s_1,\kappa) G(s_1) }{\gamma(\tfrac12,\kappa) s_1} 
  \frac{\gamma(\tfrac32-u-w,\kappa)^2 G(1-u-w)^2 }{\gamma(\tfrac12,\kappa)^2 (1-u-w)^2}  
\\
 \int_{(0)} M_2^{u} \widetilde{\omega}(u)  
   \frac{\zeta(s_1+u+w)}{(q k_1 k_1^*)^{}} \zeta(2-2w-2u) \zeta(1+s_1-w-u) 
  \\
  \Big(\frac{(g_0, c_2)}{c_2}\Big)^{1+\min(1-u-w,s_1) - w- u} 
\Big(\frac{\sqrt{q}}{e_1 r_1}\Big)^{s_1}  
\frac{(q e_1 r_1)^{1-u-w}}{a^{s_1 + 2(1-u-w)}}
\\
    \left(\frac{q}{m_1' k_1 c_2}\right)^{w} 
 \frac{\Delta \mathcal{H}(1-u-w,w,u,\kappa)}{(g_0  k_1)^{1 - 2w - 2u}}  (am_1' k_1 c_2)^{1/2-u-w}  \frac{\d u    \d s_1 \d w}{(2\pi i)^{3}}.
\end{multline*}
The constraints $\text{Re}(w+u) < 1 + \text{Re}(s)$ and $0 < \text{Re}(s) < 1$ with $s=1-u-w$ simply become
$0 < \text{Re}(u+w) < 1$.
  
Finally, we move $w$ to $1-10\varepsilon$, crossing a pole at $w=s_1 - u$ only.  On the new lines of integration, the contribution  to $\overline{\mathcal{Q}}$ is
\begin{equation*}
 \ll q^{\varepsilon} \frac{g_0 k_1}{\delta_1 e_1 e_2} \frac{1}{ k_1 k_1^*} \frac{1}{m_1' k_1 c_2}   \frac{1}{\sqrt{a m_1' k_1 c_2}}
\ll \frac{a^{1/2} q^{\varepsilon}}{c_2^{3/2} \sqrt{m_1'} k_1^* k_1^{3/2}}
 ,
\end{equation*}
using only the weak bound $g_0 \leq \delta_1 e_1 e_2 m_1'a$. 
It is easy to see that the final contribution to $\mathcal{S}_{0,0}$ from this is $O(q^{\varepsilon})$.

{\bf The pole at $w = s_1-u$}.  
This contributes
\begin{multline*}
 \overline{\mathcal{Q}}_{\text{Res}'} := \frac{1}{\delta_1 e_1 e_2}
  \int_{(5\vep )} \widetilde{V}(s_1-u)  \frac{\gamma(\tfrac12 + s_1,\kappa) G(s_1) }{\gamma(\tfrac12,\kappa) s_1} 
  \frac{\gamma(\tfrac12 + 1 - s_1,\kappa)^2 G(1-s_1)^2 }{\gamma(\tfrac12,\kappa)^2 (1-s_1)^2}  
\\
 \int_{(0)} M_2^{u} \widetilde{\omega}(u)  
   \frac{\zeta(2s_1)\zeta(2-2s_1)}{(q k_1 k_1^*)^{}}  
  \Big(\frac{(g_0, c_2)}{c_2}\Big)^{1+\min(1-s_1,s_1) - s_1} 
\Big(\frac{\sqrt{q}}{e_1 r_1}\Big)^{s_1}  
(q e_1 r_1)^{1-s_1}
\\
   \left(\frac{q}{m_1' k_1 c_2}\right)^{s_1-u} 
 \frac{\Delta \mathcal{H}(1-s_1,s_1-u,u,\kappa)}{(g_0  k_1)^{1 - 2s_1} a^{s_1 + 2(1-s_1)}}  (am_1' k_1 c_2)^{1/2-s_1}  \frac{\d u \d s_1}{(2\pi i)^2}.
\end{multline*}
In terms of $q$, this part is $O(q^{\varepsilon})$, but the problem now is that the sum over $m_1'$ will not be absolutely convergent.  The way around this roadblock is to move the contour to a location where the $m_1'$-sum converges absolutely, and shift the contour back.  Having $G(1/2) = 0$ once again is crucial. To this end, it is important to sum over the partition of unity in the $M_1$- and $M_2$-variables.

One may check that $\mathcal{H}(1-s_1, s_1-u, u, \kappa)$ is actually independent of $u$.  Therefore, it is easy to sum $\overline{\mathcal{Q}}_{\text{Res}'}$ over $M_2$:
  It is not hard to show that if $D(u)$ is a Dirichlet series absolutely convergent on the line $\text{Re}(u) = 0$, then
\begin{equation*}
 \sum_{M_2 \text{ dyadic}} \frac{1}{2 \pi i} \int_{(0)} M_2^u \widetilde{\omega}(u) D(u) du = D(0).
\end{equation*}

Now we move the $s_1$-contour to $3/4$ (crossing no poles since $G(1/2)=0$), and sum over $m_1'$ and $M_1$, giving  
\begin{multline*}
\sum_{M_1} \sum_{m_1'} \frac{\omega_{M_1}(m_1')}{\sqrt{m_1'}}   \sum_{M_2} \overline{\mathcal{Q}}_{\text{Res}'} = \frac{1}{\delta_1 e_1 e_2}
  \int_{(3/4 )} \widetilde{V}(s_1-u)  \frac{\gamma(\tfrac12 + s_1,\kappa) G(s_1) }{\gamma(\tfrac12,\kappa) s_1} 
\\  
  \frac{\gamma(\tfrac32  - s_1,\kappa)^2 G(1-s_1)^2 }{\gamma(\tfrac12,\kappa)^2 (1-s_1)^2}    
   \frac{\zeta(2s_1)^2 \zeta(2-2s_1) }{(q k_1 k_1^*)^{}} 
  \Big(\frac{(g_0, c_2)}{c_2}\Big)^{1+\min(1-s_1,s_1) - s_1} 
  \\
\Big(\frac{\sqrt{q}}{e_1 r_1}\Big)^{s_1}  
\frac{(q e_1 r_1)^{1-s_1}}{a^{s_1 + 2(1-s_1)} }
   \left(\frac{q}{k_1 c_2}\right)^{s_1} 
 \frac{\Delta \mathcal{H}(1-s_1,s_1,0,\kappa)}{(g_0  k_1)^{1 - 2s_1}}  (a k_1 c_2)^{1/2-s_1}    \frac{\d s_1}{2\pi i}.
\end{multline*}
Now we move the $s_1$-contour back to $\varepsilon$, which shows that this term is bounded by
\begin{equation*}
\frac{q^{\varepsilon}}{\delta_1 e_1 e_2} \frac{1}{ k_1 k_1^*} \frac{(g_0, c_2)}{c_2} \frac{e_1 r_1}{g_0 k_1} \frac{1}{a^2} (a k_1 c_2)^{1/2}.
\end{equation*}
Using the crude bounds $\frac{(g_0, c_2)}{g_0} \leq 1$, $\frac{e_1 r_1}{\delta_1 e_1 e_2} \leq 1$, and summing trivially over $k_1, d, c_2$, and $a$ shows that this part contributes $O(q^{\varepsilon})$ to $\mathcal{S}_{0,0}$.

\bibliographystyle{amsalpha}
\bibliography{MehmetBib3}

\end{document}